\documentclass[a4paper]{article}

\usepackage{amsmath,amsfonts,amssymb,amscd,amsthm,amsbsy}
\usepackage[latin1]{inputenc}   % pour les accents
\usepackage{mathtools}
\usepackage[a4paper]{geometry}
\usepackage{tikz}

\usepackage{pgf,tikz}
\usetikzlibrary{arrows}

\def\ds{\displaystyle}
\def\rg{\rangle}
\def\lg{\langle} 
\def \to{\rightarrow}
\def\de{\delta}

\def\vep{\varepsilon}
\def \states{\mathbb{T}^d}
\def \T{\mathbb{T}}
\def \R {\mathbb{R}}
\def \N {\mathbb{N}}
\def \Z {\mathbb{Z}}
\def \mes{\mathcal{P}}
\def\Pk{\mes(\states)}

\def\dk{{\bf d}_1}

\def \ep{\varepsilon}

\def\inte{\int_{\T^d}}
\def\dive{{\rm div}}
\def\rife#1{(\ref{#1})}
\newcommand{\be}{\begin{equation}}
\newcommand{\ee}{\end{equation}}

\newtheorem{Theorem}{Theorem}[section]
\newtheorem{Definition}[Theorem]{Definition}
\newtheorem{Proposition}[Theorem]{Proposition}
\newtheorem{Lemma}[Theorem]{Lemma}
\newtheorem{Corollary}[Theorem]{Corollary}
\newtheorem{Remark}[Theorem]{Remark}

\title{Long time behavior of the master equation in mean-field game theory}
\author{Pierre Cardaliaguet\thanks{Universit\'e Paris-Dauphine, PSL Research University, CNRS, Ceremade, 75016 Paris, France.}, Alessio Porretta\thanks{Dipartimento di Matematica, Universit\`a di Roma ``Tor Vergata''} }

\begin{document}

\maketitle

\tableofcontents
\bigskip

\begin{abstract} Mean Field Game (MFG) systems describe equilibrium configurations in  games with infinitely many interacting controllers. We are interested in the behavior of this system as the horizon becomes large, or as the discount factor tends to $0$. We show that, in the two cases,  the asymptotic behavior of the Mean Field Game system is strongly related with  the long time behavior of the so-called master equation   and with the vanishing discount limit of the discounted master equation, respectively. Both equations are nonlinear transport equations in the space of measures. We prove the existence of a solution to an ergodic master equation, towards which the  time-dependent master equation converges as the horizon becomes large, and towards which the discounted master equation converges as the discount factor tends to $0$. The whole analysis is based on the obtention of new estimates for the exponential rates of convergence of the time-dependent MFG system and the discounted MFG system.
\end{abstract}

Given a terminal time $T$ and an initial measure $m_0$, we  consider the solution to the mean field game (MFG) system 
\be\label{intro.MFG}
\left\{\begin{array}{l}
-\partial_t u^T  -\Delta u^T +H(x,D   u^T) = F(x, m^T)\qquad {\rm in}\; (0,T)\times \T^d,\\
\partial_tm^T -\Delta  m^T -\dive( m^T H_p(x,Du^T))= 0\qquad {\rm in}\; (0,T)\times \T^d,\\
m^T(0,\cdot)=m_0, \; u^T(T,\cdot)= G(\cdot, m^T(T)) \qquad {\rm in}\;  \T^d,
\end{array}\right.
\ee
where $\T^d$ is the $d$-dimensional flat torus $\R^d /\mathbb Z^d$, $F,G$ are functions defined on $\T^d\times \Pk$ (the space of probability measures on $\T^d$) and $H$ is a function, defined on $\T^d\times \R^d$, which is convex in the second variable.

Let us recall that the following system appears in mean field games  theory (introduced by Lasry and Lions \cite{LL06cr1, LL06cr2, LL07mf}, and by Huang, Caines and Malham\'e \cite{HCMieeeAC06}). Mean field games are dynamic games with infinitely many players. The first equation in \eqref{intro.MFG} can be interpreted as the value function of a small player whose cost depends on the density $m(t)$ of the   players, while the second equation describes the evolution in time of the density of the players. Note that the first equation is backward in time (and with a terminal condition) while the second one is forward, with the initial condition $m(0)=m_0$, $m_0$ being the initial repartition of the players. 

The study of the long time average of the MFG system has been initiated in \cite{LLperso} and then discussed in several different contexts,   \cite{CLLP2012, C2013, CLLP2, GMS}. 

In \cite{CLLP2} the long time average of $u^T$ is investigated  when $H(x,p)=\frac12|p|^2$ and $F(x,m),G(x,m)$ satisfy suitable smoothing conditions with respect to the measure $m$. Then it is proved that there exists a constant $\bar \lambda\in \R$ such that the scaled function $(s,x)\to u^T(Ts,x)/T$ locally uniformly converges to the map $(s,x)\to -\bar \lambda s$ as $T\to \infty$ on $(0,1)\times \T^d$, while the rescaled measure $(s,x)\to m^T(sT,x)$ converges to a time invariant measure $\bar m$ in $L^1((0,1)\times \T^d)$. The constant $\bar \lambda$ and the measure $\bar m$ are characterized as solutions of  the ergodic MFG system; namely,  there exists a unique  triple $(\bar \lambda, \bar u, \bar m)$ which solves  
\be\label{e.MFGergo}
\left\{\begin{array}{l}
\ds \bar \lambda -\Delta \bar u +H(x,D  \bar u) = F(x,\bar m)\qquad {\rm in}\; \T^d\\\
\ds -\Delta \bar m -\dive(\bar m H_p(x,D\bar u))= 0\qquad {\rm in}\; \T^d\\
\ds \bar m\geq0, \; \inte\bar m=1 \,\,\; \inte\bar u=0\qquad {\rm in}\; \T^d\,.
\end{array}\right.
\ee
and $Du^T(sT, x)$ actually converges to $D\bar u(x)$.  
The result holds under a monotonicity condition on $F$ and $G$: 
$$
\inte (F(x,m)-F(x,m'))(m-m')(dx)\geq 0, \qquad \inte (G(x,m)-G(x,m'))(m-m')(dx)\geq 0,
$$
for any $m,m'\in \Pk$. Moreover it is proved in \cite{CLLP2} that the convergence holds with an exponential rate.  Precisely, under some additional condition on the smoothing properties of the coupling terms $F$ and $G$, one has
$$
\| m^T(t)-\bar m\|_{C^{2+\alpha}}  + \|Du^T(t)-D\bar u\|_{C^{2+\alpha}} \leq C (e^{-\omega t}+ e^{-\omega(T-t)}) 
$$ 
for some constants $C, \omega>0$ and $\alpha\in (0,1)$. 

This paper is devoted to the long time behavior of $u^T$, i.e., the convergence, as $T\to \infty$, of the map $(t,x)\to u^T(t,x)- \bar \lambda (T-t)$.  This question is  inspired by results of Fathi  \cite{Fathi1, Fathi2}, Roquejoffre  \cite{Ro98}, Namah-Roquejoffre \cite{NR99} and Barles-Souganidis  \cite{BaSo} for Hamilton-Jacobi equations. In that framework, it is known that if $u$ solves the (forward) Hamilton-Jacobi equation  
$$
\partial_t u  -\Delta u +H(x,D   u) = 0\qquad {\rm in}\; (0,+\infty)\times \T^d,
$$
with associated  ergodic constant $\bar \lambda$, then $u(t,x)-\bar \lambda t$ converges, as $t\to+\infty$, to a solution $\bar u$ of the associated ergodic problem. One may wonder what remains of this result  for the MFG system. 

The convergence of the difference $u^T(t,\cdot)- \bar \lambda (T-t)$,  as $T\to\infty$,  has been an open (and puzzling) question since \cite{CLLP2}. 
We prove in this paper that  the limit of $u^T(t,\cdot)- \bar \lambda (T-t)$ indeed exists, although it cannot be described just in terms of the $\bar u$ component of the MFG ergodic system \eqref{e.MFGergo}. In order to describe this long-time behavior, we have to keep track of the initial measure $m_0$. To do so, we rely on the master equation, which is the following (backward) transport equation in the space of measures: 
\be\label{eq:masterIntro}
\left\{\begin{array}{l}
\ds -\partial_t U -\Delta_x U +H(x,D_xU) -F(x,m)  -\inte \dive(D_mU(t,x,m,y))dm(y) \\
\ds +\inte D_mU(t,x,m,y). H_p(t,y,D_xU(t,y,m))dm(y) = 0 \;   {\rm in}\; (-\infty,0)\times \T^d\times \Pk \\
U(0,x,m)= G(x,m)\qquad {\rm in}\;  \T^d\times \Pk.
\end{array}\right.
\ee
In the above equation, the unknown $U=U(t,x,m)$ depends on time, space and measure on the space; moreover, the notation  $D_mU$ denotes a  suitable derivative with respect to probability measures, which will be described in  subsection \ref{subsec:derivative}. Note that, in contrast with the MFG system, the master equation is a classical evolution equation, so its long time behavior may be described in a usual form. We recall (see \cite{LLperso}, \cite{GS14-2}, \cite{CgCrDe} and \cite{CDLL}) that the master equation is well-posed under the monotonicity condition on $F$ and $G$ and that the MFG system \eqref{intro.MFG} plays the role of characteristics for this equation. Namely, if $(u^T,m^T)$ solves \eqref{intro.MFG}, then 
$$
U(-T,x,m_0)= u^T(0,x)\qquad \forall x\in \T^d.
$$

Our main result (Theorem \ref{thm.main}) states that $U(t, \cdot, \cdot)+\bar \lambda t$ has a limit $\chi=\chi(x,m)$ as $t\to-\infty$. This limit solves (in a weak sense) the ergodic master equation
\be\label{eq:ergomasterIntro}
\begin{array}{l}
\ds \bar \lambda -\Delta_x \chi(x,m) +H(x,D_x \chi(x,m)) -\inte \dive(D_m \chi(x,m,y))dm(y)\\
\ds \qquad \qquad  +\inte D_m \chi(x,m,y). H_p(y,D_x\chi(y,m))dm(y) = F(x,m) \qquad {\rm in }\; \T^d\times \Pk.
\end{array}
\ee
As a consequence, the limit $u^T(0, \cdot)-\bar \lambda T$ exists as $T\to\infty$ and is equal to $\chi(\cdot, m_0)$. Note that, in general, $u^T(0,\cdot)-\bar \lambda T$ does not converge  to $\bar u$, since it is not always true that  $\chi(\cdot, m_0)=\bar u$ (even up to an additive constant): this is however the case if $m_0=\bar m$. 

We are also interested in the infinite horizon MFG system 
\be\label{e.MFGih}
\left\{\begin{array}{l}
-\partial_t u^\delta +\delta u^\delta -\Delta u^\delta +H(x,Du^\delta) = F(x, m^\delta(t))\qquad {\rm in}\; (0,+\infty)\times \T^d,\\
\partial_t m^\delta-\Delta m^\delta -\dive(m^\delta H_p(x,D u^\delta))= 0\qquad {\rm in}\; (0,+\infty)\times \T^d,\\
m^\delta(0,\cdot)= m_0\; {\rm in}\; \times \T^d, \qquad u^\delta\; {\rm bounded.} 
\end{array}\right.
\ee
In the first order stationary Hamilton-Jacobi (HJ) setting, where the equation reads 
$$
\delta u^\delta +H(x,Du^\delta) = 0\qquad {\rm in}\;  \T^d,
$$
Davini, Fathi, Iturriaga and Zavidovique \cite{DFIZ} have proved the convergence of $u^\delta-\delta^{-1}\bar \lambda$ as $\delta$ tends to $0$ and characterized the limit. The result has been generalized to the second order HJ setting by Mitake and Tran \cite{MT14} (see also Mitake and Tran \cite{MT15} and Ishii, Mitake and Tran \cite{IM16}). In the viscous case, the result is that, if $u^\delta$ solves the infinite horizon problem
$$
\delta u^\delta -\Delta u^\delta +H(x,Du^\delta) = 0\qquad {\rm in}\;  \T^d,
$$
then $u^\delta-\delta^{-1}\bar \lambda$ converges as $\delta \to0$ to the unique solution $\bar u$ of the ergodic cell problem 
$$
 -\Delta \bar u +H(x,D \bar u) = 0\qquad {\rm in}\;  \T^d
$$
such that $\ds \inte \bar u \bar m=0$, where $\bar m$ solves 
$$
-\Delta \bar m-\dive\left(\bar mH_p(x,D\bar u)\right)=0\qquad {\rm in}\;  \T^d, \qquad \bar m\geq 0, \; \inte \bar m=1. 
$$

Here again, one may wonder if such a result remains true for the infinite horizon MFG system \eqref{e.MFGih} (which, in contrast with the Hamilton-Jacobi case, is time dependent). As for the time-evolution MFG problem, we rely on a master equation. Following \cite{CDLL}, this infinite horizon master equation takes the form:
\be\label{e.MasterDiscont}
\begin{array}{l}
\ds \delta U^\delta-\Delta_x U^\delta +H(x,D_xU^\delta) -\inte \dive(D_mU^\delta(x,m,y))dm(y) \\
\ds\qquad \qquad +\inte D_mU^\delta(x,m,y). H_p(y,D_xU^\delta(y,m))dm(y) = F(x,m)
\;{\rm in}\; \T^d\times \Pk.
\end{array}
\ee
Our second main result  (Theorem \ref{thmdisc}) is that $U^\delta- \delta^{-1}\bar \lambda $ converges to  the unique solution $\chi$ of the master ergodic problem  \eqref{eq:ergomasterIntro} satisfying $\chi(x,\bar m)= \bar u$, where $\bar u$ is the unique solution of   the ergodic MFG system \eqref{e.MFGergo} for which  the following (new) linearized ergodic MFG system has a solution $(\bar v, \bar \mu)$:
$$
\left\{\begin{array}{l}
\ds \bar u  -\Delta \bar v +H_p(x,D\bar u).D\bar v = \frac{\delta F}{\delta m}(x, \bar m)(\bar \mu)\qquad {\rm in}\; \T^d,\\
\ds -\Delta \bar \mu -\dive(\bar \mu H_p(x,D \bar u))-\dive (\bar m H_{pp}(x,D\bar u)D\bar v)= 0\qquad {\rm in}\; \T^d,\\
\ds \inte \bar \mu= \inte \bar v=0,
\end{array}\right.
$$
(the definition of the derivative $\frac{\delta F}{\delta m}$ is explained in Section \ref{sec:notations}).  This implies the convergence of  $u^\delta(0,\cdot)-\delta^{-1}\bar \lambda$ to $\chi(\cdot, m_0)$ as $\delta$ tends to $0$. Note that if $F\equiv 0$, i.e., in the Hamilton-Jacobi case, one recovers the condition $\inte \bar u\bar m=0$ by integrating the $\bar v-$equation against the measure $\bar m$. The MFG setting is more subtle since it keeps track of the coupling between the equations. \\

Let us now say a few words about the method of proofs. As in the Hamilton-Jacobi setting, the argument relies on compactness arguments and, therefore, on the regularity (Lipschitz  estimates) for the solution $U$ to the master equation \eqref{eq:masterIntro} and for the solution $U^\delta$ of the  infinite horizon master equation \eqref{e.MasterDiscont}.  The main difficulty comes from the fact that these equations {\it  do not satisfy a comparison principle} (in contrast to the HJ equation). Moreover, as can be seen plainly from \eqref{eq:masterIntro} and \eqref{e.MasterDiscont}, the equations  do not provide easy bounds on the derivatives with respect to $m$ of $U$ and $U^\delta$. 

The key Lipschitz estimates come from the fact that the characteristics \eqref{intro.MFG} and \eqref{e.MFGih} of these master equations stabilize exponentially fast in time to the solution of the ergodic MFG system \eqref{e.MFGergo} for $U$ and to the solution of the following time invariant infinite horizon problem 
\be\label{eq.barumdelta.intro}
\left\{\begin{array}{l}
\delta \bar u^\delta-\Delta \bar u^\delta+H(x, D\bar u^\delta)= F(x, \bar m^\delta) \qquad {\rm in }\; \T^d,\\
-\Delta \bar m^\delta-\dive (\bar m^\delta H_p(x,D\bar u^\delta))=0\qquad {\rm in }\; \T^d,\\
\ds \inte \bar m^\delta=1, \; \inte \bar u^\delta=0.
\end{array}\right.
\ee
for $U^\delta$ (see Theorem \ref{thm:CvExpMFG} and Theorem \ref{them.Cvexpo.discounted} respectively). These exponential convergence rates were only known for system \eqref{intro.MFG} when $H(x,p)=|p|^2$ (see  \cite{CLLP2}), where the argument relied on some commutation properties which do not hold for general Hamiltonians. To prove the exponential convergence in our setting, we use a technique developed by the second author with E.~Zuazua \cite{PZ} to establish the so-called turnpike property for optimal control problems. The exponential rate for the infinite horizon MFG system is new, but uses similar ideas. 

The starting point of this analysis consists in studying the linearized MFG systems. For simplicity, let us explain this idea for the time dependent problem, i.e., for $U$. In this framework, the MFG linearized system reads 
$$
\left\{\begin{array}{l}
\ds-\partial_t v  -\Delta v +H_p(x,D u).Dv = \frac{\delta F}{\delta m}(x,  m)(\mu(t))\qquad {\rm in}\; (0,T)\times \T^d,\\
\ds \partial_t \mu-\Delta \mu -\dive(\mu H_p(x,D  u))-\dive ( m H_{pp}(x,D u)Dv)= 0\qquad {\rm in}\; (0,T)\times \T^d,\\
\mu(0,\cdot)= \mu_0,\;  \qquad v(T,x)=\frac{\delta G}{\delta m}(x, m)(\mu(T))\qquad {\rm in}\;  \T^d,
\end{array}\right.
$$
where $(u, m)$ is the solution of \eqref{intro.MFG} and $\mu_0$ is given. When $(u,m)=(\bar u, \bar m)$, the analysis of the above system (the exponential decay of the solutions) provides an exponential convergence of the solution of the MFG system to $(\bar u, \bar m)$---at least, this holds true for the $m$ component. A very interesting point is that this linearized system turns out to be also strongly related with the derivative of $U$ with respect to $m$: indeed, as explained in \cite{CDLL}, we have 
$$
\inte \frac{\delta U}{\delta m}(0,x,m_0,y)\mu_0(y)dy= v(0,x)\qquad \forall x\in \T^d.
$$
Thus controlling $v$ allows us to control the variations of $U$ with respect to $m$. 
Once the Lipschitz estimates for $U$ and for $U^\delta$ are obtained, the construction of a corrector $\chi$ (solution of the ergodic master equation \eqref{eq:ergomasterIntro}) follows in a standard way: see Theorem \ref{thm.MasterCell}.

However, the convergence of the solution of the time dependent master 
equation \eqref{eq:masterIntro} requires new ideas since, in contrast with the Hamilton-Jacobi setting (see Fathi \cite{Fathi3} or Barles and Souganidis \cite{BaSo}), there is no obvious quantity which is monotone in time: the reason is that the master equation does not satisfy a comparison principle. To overcome this issue, we rely again on the exponential convergence rate from which we derive  a suitable convergence of the solution of the master equation when evaluated at $\bar m$ as time tends to $-\infty$ (see Proposition \ref{prop:CvuT(0)}). Then we obtain the convergence of the map $U$ by compactness argument and using again  the convergence of the characteristics. 

The convergence of $U^\delta$, is more subtle: the key point is that two solutions of the ergodic master equation differ only by a constant. Thus we only have to show that $U^\delta(\cdot, m) -\delta^{-1}\bar \lambda$ has a limit for some $m$. The good choice turns out to be $m=\bar m^\delta$, where $(\bar u^\delta, \bar m^\delta)$ solves \eqref{eq.barumdelta.intro}: indeed, we have then $U^\delta(\cdot, \bar m^\delta)= \bar u^\delta$ and we expect $(\bar u^\delta,\bar m^\delta)$ to be close to $(\bar u, \bar m)$ in some sense, where $(\bar u, \bar m)$ satisfies \eqref{e.MFGergo}. Actually a formal expansion yields to $(\bar u^\delta, \bar m^\delta)= (\delta^{-1}\bar \lambda+\bar u+\bar \theta+ \delta \bar v, \bar m+\delta \bar \mu)$, where $(\bar \theta, \delta \bar v,\bar \mu)$ solves 
$$
\left\{\begin{array}{l}
\ds \bar u +\bar \theta -\Delta \bar v +H_p(x,D\bar u).D\bar v = \frac{\delta F}{\delta m}(x, \bar m)(\bar \mu)\qquad {\rm in}\; \T^d,\\
\ds -\Delta \bar \mu -\dive(\bar \mu H_p(x,D \bar u))-\dive (\bar m H_{pp}(x,D\bar u)D\bar v)= 0\qquad {\rm in}\; \T^d,\\
\ds \inte \bar \mu= \inte \bar v=0.
\end{array}\right.
$$
The rigorous justification is given in Proposition \ref{prop.cvbarudelta}.

The paper is organized in the following way. In Section 1 we recall the notion of derivative in the space of measures and state our main assumptions.   We also recall some decay and regularity estimates which hold separately for the two equations of the system and we provide the basic fundamental estimates for \rife{intro.MFG} which are independent of the horizon $T$.  Section \ref{sec:ExpCvRates} is devoted to the exponential convergence rate, as $T\to \infty$,  of solutions of \rife{intro.MFG} towards the couple $(\bar u, \bar m)$ solution of \rife{e.MFGergo}.  To this purpose, at first we develop decay estimates in $L^2$ for the linearized system, then we export the estimates (in stronger norms) to $(u^T-\bar u, m^T-\bar m)$   by using a  fixed point   argument. A similar strategy is used in Section 
\ref{sec:ExpCvRates2}  for the infinite horizon discounted problem \rife{e.MFGih}; in this case we prove the exponential convergence  as $t\to \infty$ towards the stationary couple $(\bar u^\de, \bar m^\de)$ solution of \rife{eq.barumdelta.intro}. In both Section \ref{sec:ExpCvRates} and Section \ref{sec:ExpCvRates2}, the  analysis of  the linearized systems is a crucial step, and this will also play a key role in the study of the master equations, both the time-dependent \rife{eq:masterIntro} and the stationary one \rife{e.MasterDiscont}, respectively. This is the content of Sections  \ref{sec:mastercell}--\ref{sec:discount}. More precisely, in Section \ref{sec:mastercell} we prove the existence of  a solution to the ergodic master equation, obtained as the limit, when $\de\to 0$,  of a subsequence of solutions of \rife{e.MasterDiscont}.  The long-time behavior of the time-dependent master equation \rife{eq:masterIntro} is addressed in Section \ref{sec:longtime}. Finally,  the limit of the whole sequence of solutions of \rife{e.MasterDiscont} is proved in section \ref{sec:discount}. \\

{\bf Acknowledgement:} The first author was partially supported by the ANR (Agence Nationale de la Recherche) project ANR-16-CE40-0015-01. The second author was partially supported  by the Indam  Gnampa project 2015 {\it Processi di diffusione degeneri o singolari legati al controllo di dinamiche stocastiche}. 

The work was completed as the first author was visiting Universit\`{a} Roma Tor Vergata. He   wishes to thank  the university for its hospitality. 

%%%%%%%%%%%%%%%%%%%%%%%%%%%
\section{Notation, assumptions and preliminary estimates}\label{sec:notations}

\subsection{Notation and assumptions}\label{subsec:derivative}

Throughout the paper we work on the $d-$dimensional torus $\T^d:=\R^d/\Z^d$: this means that all equations are $\Z^d-$periodic in space. This assumption is standard in the framework of the long-time behavior. We denote by $\Pk$ the set of Borel probability measures on $\T^d$, endowed with the Monge-Kantorovich distance $\dk$:
$$
\dk(m,m')=\sup_{\phi} \inte \phi \,d(m-m') \qquad \forall m,m'\in \Pk,
$$
where the supremum is taken over all $1-$Lipschitz continuous maps $\phi:\T^d\to \R$. 

For $\alpha\in [0,1]$, we denote by $C^\alpha([0,T],\Pk)$ the set of maps $m:[0,T]\to\Pk$ which are $\alpha-$H\"older continuous if $\alpha\in (0,1)$, continuous if $\alpha=0$.

Next we recall the notion of derivative of a map $U:\Pk\to \R$ as introduced in \cite{CDLL}. We say that $U$ is $C^1$ if there exists a continuous map $\frac{\delta U}{\delta m}: \Pk\times \T^ d\to \R$ such that 
$$
U(m')-U(m)=\int_0^1\inte \frac{\delta U}{\delta m}((1-t)m+tm',y)\, d(m'-m)(y) dt \qquad \forall m,m'\in \Pk.
$$
We observe that if $U$ can be extended to $L^2(\T^d)$ then $y\mapsto \frac{\delta U}{\delta m}(m,y)$ is nothing but the representation in $L^2$ of the Gateaux derivative of $U$ computed at $m$. The fact that $U$ is defined on probability measures, i.e. with the constraint of mass one, lets $\frac{\delta U}{\delta m}(m,y)$ be defined up to a constant. We normalize the derivative   by the condition 
\be\label{Convention}
\inte \frac{\delta U}{\delta m}(m,y)\, dm(y)=0 \qquad \forall m\in \Pk.
\ee
We write indifferently $\ds \frac{\delta U}{\delta m}(m)(\mu)$ and $\ds \inte \frac{\delta U}{\delta m}(m,y)\,d\mu(y)$ for a signed measure $\mu$ with finite mass.

When the map  $\frac{\delta U}{\delta m}= \frac{\delta U}{\delta m}(m,y)$ is differentiable with respect to the last variable, we denote by $D_mU(m,y)$ its gradient: 
$$
D_mU(m,y):= D_y\frac{\delta U}{\delta m}(m,y). 
$$
Let us recall \cite{CDLL} that $D_mU$ can be used to estimate the Lipschitz regularity of $U$ in the $m$ variable:
$$
\left|U(m)-U(m')\right| \leq \dk(m,m')\, \left[\sup_{m''\in \Pk, y\in \T^d}\left| D_mU(m'',y)\right|\, \right]\qquad \forall m,m'\in \Pk.
$$

For $p=1, 2, \infty$, we denote by $\|\cdot\|_{L^p}$ the $L^p$ norm of a map on $\T^d$ (we often use the notation $\|\cdot\|_\infty$ for $\|\cdot\|_{L^\infty}$).  For $k\in \N$ and $\alpha\in (0,1)$, we denote by $\|\cdot\|_{C^k}$ (respectively $\|\cdot\|_{C^{k+\alpha}}$) the standard norm on the set of maps defined on $\T^d$ and which are of class $C^k$ (respectively of class $C^k$ with a $k-$th derivative which is $\alpha$-H\"older continuous). By $\|\cdot\|_{(C^{k+\alpha})'}$ we mean the norm in the dual space: 
$$
\|\phi\|_{(C^{k+\alpha})'} :=\sup\{ \inte \phi \psi, \; \|\psi\|_{C^{k+\alpha}}\leq 1\}.
$$
For a map $\phi$  depending of two spatial variables, we denote by $\|\phi(\cdot,\cdot)\|_{k+\alpha, k'+\alpha}$ the supremum of the $\alpha-$H\"older norm of the partial derivatives of order $l\leq k$ and $l'\leq k'$ respectively of the map $\phi$. 

Finally, if $\phi=\phi(x)$, we systematically denote by $\lg \phi\rg := \inte \phi(x)dx$ the average of $\phi$.

If $u:[0,T] \times \T^d\to \R$ is a sufficiently smooth map, we denote by $Du(t,x)$ and $\Delta u(t,x)$ its spatial gradient and spatial Laplacian and by $\partial_t u(t,x)$ its partial derivative with respect to the time variable. We will also use the classical parabolic H\"older spaces: for $\alpha\in (0,1)$, we denote by $C^{\alpha/2,\alpha}$ the set  of maps which are $\alpha-$H\"older in space and $\alpha/2-$H\"older in time and by $C^{1+\alpha/2, 2+\alpha}$ the set of maps $u$ such that $\partial_t u$ and $D^2u$ are in $C^{\alpha/2,\alpha}$. 
\bigskip

\noindent {\bf Assumptions.} The following assumptions are in force throughout the paper.

\begin{itemize}
\item[(H)]  The Hamiltonian $H=H(x,p):\T^d\times \R^d\to \R$ is of class $C^2$ and the function $p\mapsto D^2_{pp}H(x,p)$ is Lipschitz continuous, uniformly w.r.t. $x$, and satisfies the growth condition:
\be\label{cond:Hcoercive}
C^{-1}I_d \leq D^2_{pp}H(x,p)\leq CI_d\qquad \forall (x,p)\in \T^d\times \R^d .
\ee
Moreover we suppose that 
there exists $\theta\in (0,1)$ and $C>0$ such that 
\be\label{Htheta}
\left| D_{xx}H(x,p)\right| \leq C(1+ |p|)^{1+\theta}, \qquad 
\left| D_{xp}H(x,p)\right| \leq C(1+ |p|)^{\theta}, \qquad \forall (x,p)\in \T^d\times \R^d.
\ee
This latter assumption is a little awkward, since it requires the quadratic part of $H$ to be homogeneous, but we actually  need it  in order to ensure uniform Lipschitz regularity of $u^T$ (solution to \eqref{intro.MFG}) and of $u^\delta$ (solution to \eqref{e.MFGih}) independently of $T$ and $\delta$: see Lemma \ref{lem.SemiConc} and Lemma \ref{lem.udeltamdelta}.  If the same bounds were available with different arguments, then we could get rid of this condition, since in the rest of the paper we do not use it at all.

\item[(FG)]  The coupling functions $F,G: \T^d\times \Pk\to \R$ are supposed to be of class $C^1$ and their first derivatives satisfy the following Lipschitz conditions:

\begin{itemize}

\item[(FGa)] $F, G$ are twice differentiable in the $x$ variable and $F_{xx} (x,m), G_{xx}(x,m) $ are bounded uniformly in $\T^d\times \Pk$.
 
\item[(FGb)] $\frac{\delta F}{\delta m}(x,m,y)$, $\frac{\delta G}{\delta m}(x,m,y)$ are differentiable with respect to $(x, y)$ and Lipschitz continuous
in $T^d \times \Pk \times T^d$ (i.e. globally Lipschitz in the three variables).
\end{itemize}

Even if this will not be strictly needed, an extra regularity condition is assumed in order to access to smooth solutions of the master equation as stated in \cite{CDLL}. Namely we assume that:

\begin{itemize}

\item[(FGc)]  For any $\alpha\in (0,1)$, $F(\cdot,m)$ and $\frac{\delta F}{\delta m}(\cdot,m,\cdot)$ are of class $C^{2+\alpha}$ in all space variables, uniformly in $m$, and $\frac{\delta F}{\delta m}$ is Lipschitz continuous in $m$ with respect to $C^{2+\alpha}$ in space. The same holds for $G$ in norm $C^{3+\alpha}$.

\item[(FGd)] The maps $F$ and $G$ are assumed to be monotone:  for any $m\in\Pk$ and for any centered Radon measure $\mu$, 
\be\label{hyp:mono}
\inte\inte \frac{\delta F}{\delta m}(x,m,y)\mu(x)\mu(y)dxdy\geq 0, \qquad \inte\inte \frac{\delta G}{\delta m}(x,m,y)\mu(x)\mu(y)dxdy\geq 0.
\ee

\end{itemize}
\end{itemize}

 Let us comment upon our assumptions. \\

The regularity of $H$ as well as the uniform convexity with respect to the second variable are standard in MFG theory. Here these assumptions are all the more important that we make a systematic use of the duality inequality (see \cite{LL07mf}) which provides  uniqueness and quantified stability for the MFG system under this strong convexity assumption. 

The regularity assumption on  $\frac{\delta F}{\delta m}$ (and on $\frac{\delta G}{\delta m}$) allows for instance inequalities of the form: 
$$
\left\|\frac{\delta F}{\delta m}(\cdot, m)(\mu)\right\|_{C^2}\leq C \|\mu\|_{(C^2)'}
$$
for any $m\in \Pk$ and any distribution $\mu$ on $\T^d$. \\

The monotonicity assumption \eqref{hyp:mono} implies (and, under our regularity assumptions, is equivalent to) the more standard one: 
$$
\inte (F(x,m)-F(x,m'))d(m-m')(x)\geq 0, \qquad \inte (G(x,m)-G(x,m'))d(m-m')(x)\geq 0, 
$$
for any measures $m,m'\in \Pk$. This condition ensures the well-posedness of the MFG system \eqref{intro.MFG} for large times intervals and the well-posedness of the ergodic MFG system \eqref{e.MFGergo}. Without this assumption, these MFG systems may have several solutions and the long time average (and {\it a fortiori} the long time behavior) of the MFG system \eqref{intro.MFG} is not known. 
\vskip1em

Let us stress that, in the following, we will denote  generically  by $C$ possibly different constants appearing in the estimates which depend on the data $F, G$ and $H$ through the above assumptions. In particular, those constants will depend on  the sup-norm of $F_{xx}, G_{xx}$ (which are bounded uniformly w.r.t. $x$ and $m$ from (FGa)),    the Lipschitz constants of $\frac{\delta F}{\delta m}, \frac{\delta G}{\delta m}$ and the conditions \rife{cond:Hcoercive}-\rife{Htheta}, respectively.

%%%%%%%%%%%%%%
\subsection{Preliminary estimates}

We will use repetitively the following estimates on linear equations which are independent of the time horizon. The first one is about  linear equations in divergence form (see \cite[Lemma 7.1, Lemma 7.6]{CLLP2}). 

\begin{Lemma} \label{lemCLLP2.1} Let $V$ be a bounded vector field on $(0,T)\times \T^d$, $B\in L^2((0,T)\times \T^d)$ and let $\mu$ be the solution to 
\be\label{eq.kolmo}
\begin{cases}
\partial_t \mu-\Delta \mu +\dive (\mu V)= \dive( B)& \qquad {\rm in} \; (0,T)\times \T^d
\\
\mu(0)=\mu_0 & \qquad {\rm in} \;   \T^d,
\end{cases}
\ee
with $\inte \mu_0=0$. 

There exist constants $\omega>0$ and $C>0$, depending only on $\|V\|_\infty$, such that 
$$
\|\mu(t)\|_{L^2}\leq Ce^{-\omega t} \|\mu_0\|_{L^2} +C\left[\int_0^t \|B(s)\|_{L^2}^2ds\right]^{1/2}.
$$
If $B\equiv 0$, we also have, for any $\tau>0$,  
$$
\|\mu(t)\|_\infty\leq C_\tau e^{-\omega t}\|\mu_0\|_{L^1} \qquad \forall t\geq \tau,
$$
where the constant $C_\tau$ depends on $\tau$ and $\|V\|_\infty$ only. 
\end{Lemma}

The second Lemma is about a viscous transport equation (see \cite[Lemma 7.4]{CLLP2} and \cite[Lemma 7.5]{CLLP2}). 

\begin{Lemma} \label{lemCLLP2.2} Let $V$  be a bounded vector field, $A\in L^2((0,T)\times \T^d)$ and $v$ be the solution to the backward equation
\be\label{eq.backward}
-\partial_t v-\Delta v +V\cdot Dv= A \qquad {\rm in} \; (0,T)\times \T^d. 
\ee
There exist constants $\omega>0$ and $C>0$, depending only on $\|V\|_\infty$, such that 
$$
\|v(t)-\lg v(t)\rg \|_{L^2}\leq Ce^{-\omega (T-t)} \|v(T)-\lg v(T)\rg \|_{L^2} +C\int_t^T e^{-\omega(s-t)}\|A(s)\|_{L^2}ds
$$
and, if $A \in L^\infty((0,T)\times \T^d)$, 
$$
\|v(t)-\lg v(t)\rg \|_{L^\infty}\leq Ce^{-\omega (T-t)} \|v(T)-\lg v(T)\rg \|_{L^\infty} +C\int_t^T e^{-\omega(s-t)}\|A(s)\|_{L^\infty}ds,
$$
where $\lg \phi\rg=\inte \phi$ for any map $\phi$. Moreover, for any $0\leq t < t_0\leq T$, 
$$
(t_0-t)\|Dv(t)\|_{L^2}\leq C(t_0-t+1)\left(\|v(t_0)-\lg v(t_0)\rg \|_{L^2} +\|A\|_{L^2((t,t_0)\times \T^d)}+ \|v-\lg v\rg\|_{L^2((t,t_0)\times \T^d)}\right).
$$
\end{Lemma}

We note for later use a simple consequence of Lemma \ref{lemCLLP2.1}: 

\begin{Corollary}\label{cor.StaMeas} Let $V$ and $B$ be  (time-independent) vector fields. Then any $L^2$ solution of 
$$
-\Delta \mu +\dive (\mu V)= \dive( B)\qquad {\rm in}\; \T^d
$$
with $\inte \mu =0$, satisfies
$$
\|\mu\|_{H^1}\leq C \, \|B\|_{L^2},
$$
where $C$ depends only on $\|V\|_\infty$. 
\end{Corollary}

\begin{proof} It is enough to apply Lemma  \ref{lemCLLP2.1}: 
$$
\|\mu\|_{L^2}\leq Ce^{-\omega t} \|\mu\|_{L^2}+  C \|B\|_{L^2}t^{1/2}, 
$$
Choosing $t$ large enough, this gives 
$$
\|\mu\|_{L^2}\leq C\, \|B\|_{L^2}\,.
$$
Then, multiplying the equation by $\mu$, the standard energy estimate gives
$$
  \|D\mu\|_{L^2}\leq  \left[ \|V\|_\infty\|\mu\|_{L^2} + \|B\|_{L^2}\right],
$$
which gives the result. 
\end{proof}

We conclude this Section with  a further bound for the solutions of the Fokker-Planck equation.

\begin{Lemma}\label{lem.fromm0meas}
Let $V$ be a  bounded vector field on $(0,T)\times \T^d$ with bounded space derivatives and $\mu$ be a weak solution to \eqref{eq.kolmo} with $B\equiv0$. Then, for any $\tau>0$,  
$$
\|\mu(t)\|_\infty\leq C_\tau e^{-\omega t}\|\mu_0\|_{(C^{2+\alpha})'} \qquad \forall t\geq \tau,
$$
where $\omega$ is given by Lemma \ref{lemCLLP2.1},   $\alpha\in (0,1)$ and $C_\tau>0$ depends on $\|V\|_{L^\infty}$, $\|DV\|_{L^\infty}$ and $\tau$.
% and $T$ only. 
\end{Lemma}

\begin{proof}
Let $\tau>0$ and $v$ be the solution to the transport equation 
$$
\left\{\begin{array}{l}
-\partial_t v-\Delta v +V\cdot Dv = 0\qquad  {\rm in}\; (0,\tau)\times \T^d\\
v(\tau,x)=v_\tau(x)\qquad  {\rm in}\; \T^d. 
\end{array}\right.
$$
where $v_\tau$ is in $C^\infty(\T^d)$. One easily checks that 
$$
\sup_t \|v(t)\|_{L^2}+\|Dv\|_{L^2((0,\tau)\times \T^d)} \leq C\|v_\tau\|_{L^2},
$$ 
where $C$   depends on $\|V\|_\infty$ and $\tau$ only. Standard parabolic regularity
 (Theorem III.11.1 of \cite{LSU}) then implies that 
$$
\|Dv\|_{C^{\alpha/2,\alpha}([0,\tau/2]\times \T^d)}\leq C \|v_\tau\|_{L^2}
$$ 
for some $\alpha$ and $C$ depending on $\|V\|_\infty$ and $\tau$  only. For any $i\in \{1, \dots, d\}$, the derivative $v_{x_i}$ solves 
$$
-\partial_t v_{x_i}-\Delta v_{x_i} +V\cdot Dv_{x_i}+ V_{x_i}\cdot Dv = 0\qquad  {\rm in}\; (0,\tau/2)\times \T^d.
$$
By parabolic regularity (Theorem III.11.1 of \cite{LSU}), we infer that 
$$
\|D^2v\|_{C^{\alpha/2,\alpha}([0,\tau/4]\times \T^d)}\leq C\|Dv\|_{L^\infty((0,\tau/2)\times \T^d)} \leq C \|v_\tau\|_{L^2}
$$ 
for some $\alpha$ and $C$ depending on $\|V\|_\infty$, $\|DV\|_\infty$ and $\tau$  only. 
We have, by duality,
$$
\inte v_\tau \mu(\tau) =\inte v(0)d\mu_0(x).
$$
So taking the supremum over  $v_\tau$ such that $\|v_\tau\|_{L^2}\leq 1$, we infer that
$$
\|\mu(\tau)\|_{L^2}\leq C_\tau\,  \|\mu_0\|_{(C^{2+\alpha})'}\qquad \forall \tau >0.
$$
We can then derive the conclusion by Lemma \ref{lemCLLP2.1}. 
\end{proof}

%%%%%%%%%%%%%%
\subsection{Regularity of the MFG system}

The aim of this section is to provide additional basic estimates on the solution to the  MFG system
\be\label{basic-sys}
\left\{\begin{array}{l}
-\partial_t u +\bar \lambda -\Delta u +H(x,D   u) = F(x, m)\qquad {\rm in}\; (0,T)\times \T^d\\
\partial_tm -\Delta  m -\dive( m H_p(x,Du))= 0\qquad {\rm in}\; (0,T)\times \T^d\\
m(0,\cdot)=m_0, \; u(T,\cdot)= g\qquad {\rm in}\;  \T^d
\end{array}\right.
\ee
where $m_0\in \Pk$. Let us recall that $\bar \lambda \in \R$ is the unique ergodic constant and $(\bar u,\bar m)$ the unique solution to the ergodic MFG system \eqref{e.MFGergo}.  

The following estimates are mostly well-known since \cite{CLLP2}, but we collect them for the sake of completeness.   The whole point is to get estimates which are independent of the time horizon or of the discount rate.  To this purpose we rely on conditions \rife{cond:Hcoercive}-\rife{Htheta} as well as on the smoothing assumption  (FGa) for the couplings.

\begin{Lemma}\label{lem.SemiConc} For any $M>0$, there exists a constant $C>0$ such that for any horizon $T>0$, if $(u,m)$ is the solution to the MFG system \rife{basic-sys} and $\|g\|_{C^{2}(\T^d)}\leq M$, then 
$$
\|Du\|_\infty\leq C. 
$$
\end{Lemma}

\begin{proof} As in Lemma 3.2 in \cite{CLLP2}, the proof relies on the uniform semi-concavity of the solution. Let us recall that, for any smooth map $\phi\in C^\infty(\T^d)$, we have
\be\label{SemiConcImpliesLipschitz}
\|D\phi\|_\infty\leq d^{1/2}\sup_{x\in \T^d, \; |z|\leq 1} ( D^2\phi(x) z\cdot z)_+.
\ee
Let $\xi$ with $|\xi|\leq 1$ be a direction for which $C_0:=\sup_{(t,x)}  D^2u(t,x)\xi\cdot \xi$ is maximal (and thus nonnegative). We set $w(t,x)=D^2u(t,x)\xi\cdot \xi= u_{\xi\xi}(t,x)$. Then $w$ solves 
$$
-\partial_t w-\Delta w +H_{\xi\xi}(x,Du)+ 2 H_{\xi p}(x,Du)\cdot Du_\xi
+  H_{pp}(x,Du) Du_\xi\cdot Du_\xi+ H_p(x,Du)\cdot Dw= F_{\xi\xi}(x,m(t)). 
$$
If the maximum of $w$ is reached at $T$, then 
$$
C_0\leq \max_{x\in \T^d} D^2g(x)\xi\cdot \xi \leq M. 
$$
Otherwise, one has at the maximum point $(t,x)$ of $w$: 
$$
H_{\xi\xi}(x,Du)+ 2 H_{\xi p}(x,Du)\cdot Du_\xi
+  H_{pp}(x,Du) Du_\xi\cdot Du_\xi \leq  F_{\xi\xi}(x,m(t)),
$$
where by our standing assumptions on $H$ we have 
\begin{align*}
H_{\xi\xi}(x,Du) &  \geq -C\, (1+ |Du|)^{1+\theta}
\\
H_{pp}(x,Du) Du_\xi\cdot Du_\xi + 2H_{\xi p}(x,Du)\cdot Du_\xi  &\geq C^{-1} |Du_\xi|^2 - C (1+ |Du|)^{2\theta}\,.
\end{align*}
Since  \eqref{SemiConcImpliesLipschitz} implies that $ \|Du\|_\infty \leq d^{\frac12} \, C_0$,  we deduce that
$$
-C (1+C_0)^{1+\theta}-C\,(1+ C_0)^{2\theta} + 
 C^{-1} |Du_\xi|^2 \leq  C
$$
and since $|Du_\xi|\geq C_0$ at the maximum point of $w(t,x)$, being $\theta<1$ we conclude that $C_0$ is bounded. 
By \eqref{SemiConcImpliesLipschitz}, we infer the Lipschitz estimate for $u$. 
\end{proof}

\begin{Remark} Thanks to Lemma \ref{lem.SemiConc}, the drift $H_p(x,Du)$ in the Fokker-Planck equation is uniformly bounded. As a consequence, as it is well-known, the  solution $m$ satisfies the following H\"older continuity estimate in time:
\be\label{holder-m}
\dk(m(t), m(s)) \leq C \, |t-s|^{\frac12}\qquad \forall t, s \in (0,T)\,: |t-s| \leq 1\,,
\ee
for some constant $C$ independent of $T$.
\end{Remark}

Next result exploits the stability of the system which stems from  the monotonicity of $F$ and the convexity of $H$ (see \cite{LL07mf}). In particular, whenever $H$ is uniformly convex, as is assumed in \rife{cond:Hcoercive}, the following estimate holds for any couple of solutions $(u_1, m_1)$ and $(u_2, m_2)$ of the system \rife{basic-sys}:
\be\label{dualrela}
C^{-1}\inte (m_1+m_2) |D(u_1-u_2)|^2\, \leq  - \frac d{dt} \inte (u_1-u_2)(m_1-m_2)
\ee

\begin{Lemma}\label{lem.boundm} For any $\ep>0$ and $M>0$, there exists times $\hat T>\tau>0$ (depending only on $\ep$, $M$ and the data of the problem) such that, if $T\geq \hat T$ and if $(u,m)$ be the solution to the MFG system \rife{basic-sys} and $\|g\|_{C^{2}(\T^d)}\leq M$, we have,   for some $\alpha\in (0,1)$,
$$
 \|m(t)-\bar m\|_{C^\alpha} + \| Du(t)-D\bar u\|_{C^\alpha}  \leq \ep\qquad \forall t\in [\tau, T-\tau]\,.
$$
\end{Lemma}

\begin{proof}We follow closely the argument of Lemma 3.5 of \cite{CLLP2} (in the case $H=|p|^2$) and, for this reason, we only sketch the proof. By Lemma \ref{lem.SemiConc}, $u$ is uniformly Lipschitz continuous in space, with a Lipschitz constant depending only on the regularity of $H$, $F$ and on $\|Dg\|_\infty+\|D^2g\|_\infty$. So, by Lemma \ref{lemCLLP2.1}, we have 
$$
\sup_{t\geq 1} \|m(t)\|_\infty\leq C, 
$$
where $C$ depends only on $\|H_p(\cdot, Du(\cdot))\|_\infty$, and thus only on the data. Applying \rife{dualrela} to $(u,m)$ and $(\bar u, \bar m)$, and using $\bar m>0$ in $\T^d$, we have 
$$
C^{-1}\int_{t_1}^{t_2}  \|D(u(t)-\bar u)\|_{L^2}^2dt \leq -\left[  \inte (u(t)-\bar u)(m(t)-\bar m)\right]_{t_1}^{t_2} .
$$
Thus
$$
\int_0^{T} \|D(u(t)-\bar u)\|_{L^2}^2 dt \leq C,
$$
because $u$ is uniformly Lipschitz continuous in space and $m(t)$ and $\bar m$ are probability measures. In particular, if $T\geq 3\ep^{-1}$, there exists  times $t_1\in [1, \ep^{-1}]$, $t_2\in [T-\ep^{-1},T]$ such that 
\be\label{Du(ti)-Dbaru}
\|D(u(t_i)-\bar u)\|_{L^2}\leq C\ep^{1/2}\qquad {\rm for}\; i=1,2. 
\ee
Coming back to the duality relation, we infer by Poincar\'e's inequality that 
$$
\begin{array}{l}
\ds C^{-1}\int_{1/\ep}^{T-1/\ep}  \|D(u(t)-\bar u)\|_{L^2}^2dt\leq C^{-1}\int_{t_1}^{t_2}  \|D(u(t)-\bar u)\|_{L^2}^2dt\\
\qquad \qquad \ds  \leq \|D(u(t_1)-\bar u)\|_{L^2}\|m(t_1)-\bar m\|_{L^2}
+\|D(u(t_2)-\bar u)\|_{L^2}\|m(t_2)-\bar m\|_{L^2}\leq C\ep^{1/2}.
\end{array}
$$
As $\mu:=m-\bar m$ satisfies
\be\label{muu}
\partial_t \mu-\Delta \mu -\dive(\mu H_p(x,D u))= -\dive (\bar m (H_p(x, D \bar  u)-H_p(x,Du))), 
\ee
and still using the fact that $Du$ is bounded, we have from Lemma \ref{lemCLLP2.1} that, for any $t\in [1/\ep,T-1/\ep]$,  
$$
\begin{array}{rl}
\ds \|m(t)-\bar m\|_{L^2}\; \leq & \ds Ce^{-\omega (t-1/\ep)} \|m(1/\ep)-\bar m\|_{L^2}+ 
C \left[\int_{1/\ep}^{T-1/\ep}   \|D(u(t)-\bar u)\|_{L^2}^2dt \right]^{\frac12}\\
\leq & \ds C(e^{-\omega (t-1/\ep)}+ \ep^{1/4}).
\end{array}
$$
So we can choose $\tau$ large enough (depending only on $\ep$, on the data and on $M$) such that the right-hand side is less than $C\ep^{1/4}$ if $t\in [\tau-1, T-1/\ep]$. 

Let us now upgrade  this inequality into an $L^\infty$ estimate for the interval $[\tau, T-1/\ep]$. For this, we recall from \rife{muu} that $\mu$ solves a parabolic equation of the type
$$
\partial_t \mu-\Delta \mu -\dive(\mu \, b + B)= 0,
$$
where $b$ is bounded in $L^\infty$ and $B$ is bounded in $L^p$ for any   $p\geq 2$ since
$$
\int_{\frac1\vep}^{T-\frac1\vep}\| B(t)\|_{L^p}^p \leq  C \int_{\frac1\vep}^{T-\frac1\vep} \inte |D( u(t)-\bar u)|^p   \leq C  \int_{\frac1\vep}^{T-\frac1\vep}\inte |D( u(t)-\bar u)|^2  \leq C \vep^{\frac12}
$$
where we used the global bound for $Du(t)$. Since we already know that $\|\mu(t)\|_{L^2}\leq C \vep^{\frac14}$,
by choosing $p$  sufficiently  large we deduce (see e.g. \cite[Theorem III.8.1 p. 196]{LSU}) that $\mu$ is bounded in $C^{\alpha/2, \alpha}$ for some $\alpha\in (0,1)$ and
$$
\|\mu(t)\|_{C^\alpha} \leq C \left( \sup\limits_{s\in (\tau-1, T-\frac1\vep) }\|\mu(s)\|_{L^2} + \| B\|_{L^p((\frac1\vep, T-\frac1\vep)\times \T^d)}\right) \leq C (\vep^{\frac14}+ \vep^{\frac1{2p}})\,
$$
for any $t\in [\tau, T-1/\ep]$. This concludes the bound for $\|m(t)-\bar m\|_{C^\alpha}$.
%
%with bounded coefficients and that $\|m(t)\|_{L^2}$ is bounded for $t\geq 1$. So $m$ is bounded in $C^{\beta/2,\beta}$ (from \cite[Theorem III.8.1 p. 196]{LSU}
%combined with \cite[Theorem III.10.1 p. 204]{LSU}) on $[1, T]$) for some $\beta\in (0,1)$.  By interpolation, we infer that 
%$$
%\|m(t)-\bar m\|_{\infty}\leq C \|m(t)-\bar m\|_{L^2}^{2\beta/(2\beta+d)}\|m(t)-\bar m\|_{C^\beta}^{d/(2\beta+d)}\leq C\ep^{\beta/(4\beta+2d)}. 
%$$
%for any $t\in [\tau, T-1/\ep]$. 
In order to prove the estimate for $u$, let us note that $v=u-\bar u$ satisfies 
$$
-\partial_t v-\Delta v+V\cdot Dv = F(x,m(t))-F(x,\bar m),
$$
where $V$ is the bounded vectory field $V(t,x)=\int_0^1 H_p(x,\lambda Du(t,x)+ (1-\lambda) D\bar u(x))d\lambda$. By Lemma \ref{lemCLLP2.2} we have, for $t\in [1/\ep,T-1/\ep]$,  
$$
\begin{array}{rl}
\ds \|v(t)-\lg v(t)\rg\|_\infty \; \leq & \ds 
\|v(T-1/\ep)-\lg v(T-1/\ep)\rg\|_\infty e^{-\omega(T-1/\ep-t)}\\
& \ds \qquad \qquad  + C\int_t^{T-1/\ep} 
e^{-\omega(s-t)} \|F(x,m(t))-F(x,\bar m)\|_\infty ds  \\
\leq & C\left(e^{-\omega(T-1/\ep-t)}  + \ep^{\frac1{2p}}\right).
\end{array}
$$
Choosing $\tau>1/\ep$ large enough then implies that 
$$
\|v(t)-\lg v(t)\rg\|_\infty\leq C \ep^{\frac1{2p}}\qquad \forall t\in [\tau,T-\tau].
$$
Finally, we can replace the left-hand side by $\|Dv(t)\|_{C^\alpha}$ by using again Lemma \ref{lemCLLP2.2}. Indeed, 
whenever $v$ satisfies
$$
-\partial_t v-\Delta v+V\cdot Dv = A
$$
with $V$, $A$ being bounded, we estimate, for any interval $[t, t+1]$,
\begin{align*}
\|Dv(t)\|_{C^\alpha} &  \leq C  \, \sup\limits_{s\in (t, t+1/2)} [\|v(s)-\lg v(s)\rg\|_\infty + \|A(s)\|_\infty+ \| Dv(s)\|_{L^2}]
\\ &
\leq C  \, \sup\limits_{s\in (t, t+1)} [\|v(s)-\lg v(s)\rg\|_\infty + \|A(s)\|_\infty]\,.
\end{align*}
Since $A= F(x,m(t))-F(x,\bar m)$, the previous estimates give the conclusion.
\end{proof}

%%%%%%%%%%%%%%%%%%%%%%%%%%%
%%%%%%%%%%%%%%%%%%%%%%%%%%%%
\section{Exponential rate of convergence for the finite horizon MFG system}\label{sec:ExpCvRates}

In this section we provide several convergence results with an exponential rate of convergence for finite horizon MFG systems. The results of this section extend to general Hamiltonians the main results of \cite{CLLP2} (though requiring slightly stronger assumptions  on the coupling $F$).  If the results are interesting for themselves, they are nevertheless motivated by the rest of the paper, in which they play a central role. 

The method of proof for these exponential rates differs completely from \cite{CLLP2}, where it relied on an algebraic structure of the linearized system. We start with the linearized systems and first get a crude estimate on the solution. Using the monotonicity assumption, the duality method shows that a suitable quantity is monotonous   in time and bounded (thanks to the rough estimate). A compactness argument, borrowed from  \cite{PZ}, then shows that the limit of this quantity must vanish. We then use the linearity property of the system to get an exponential rate of convergence. The non linear equations are treated as perturbations of the linear ones. Note that the key argument is inspired by \cite{PZ}, where the long time behavior of optimality systems is analyzed by using the stabilizing properties of the Riccati feedback operator. However, in contrast with \cite{PZ}, our system does not come from an optimal control problem in general, which makes a substantial difference. 

%%%%%%%%%%%%%%%%%%%%%%%%%%%%%%
\subsection{Estimates for the linearized system}\label{sectionlinea}

In this subsection  we study the linearized MFG system around the stationary ergodic solution  $(\bar u,\bar m)$,  namely
 \be\label{MFGlili}
\left\{\begin{array}{l}
\ds-\partial_t v  -\Delta v +H_p(x,D\bar u).Dv = \frac{\delta F}{\delta m}(x, \bar m)(\mu(t))\qquad {\rm in}\; (0,T)\times \T^d\\
\ds \partial_t \mu-\Delta \mu -\dive(\mu H_p(x,D \bar u))-\dive (\bar m H_{pp}(x,D\bar u)Dv)= 0\qquad {\rm in}\; (0,T)\times \T^d\\
\mu(0,\cdot)= \mu_0,\;  \qquad v(T,x)=\frac{\delta G}{\delta m}(x,\bar m)(\mu(T))+v_T(x)\qquad {\rm in}\;  \T^d
\end{array}\right.
\ee
with $\inte \mu_0=0$.

Thanks to the assumptions made upon  $\frac{\delta F}{\delta m}$ and $\frac{\delta G}{\delta m}$,  and to the smoothness of $(\bar u, \bar m)$, problem \rife{MFGlili} can be considered in a  standard  framework of weak solutions with finite energy, i.e. $v,m \in L^2((0,T); H^1(\T^d))$. Solutions will eventually be more regular, but we are not considering this issue here; our main purpose, which  is the following result, is to show the $L^2$ decay estimates for $\mu$ and $Dv$, assuming the same regularity on the initial-terminal conditions.

\begin{Proposition}\label{Prop:expdecay1} There exists $C_0>0$, $\lambda >0$ such that, if $(v,\mu)$ is a solution to the MFG linearized system 
\rife{MFGlili} with $\inte \mu_0=0$, then we have  
$$
\|\mu(t)\|_{L^2}+\|Dv(t)\|_{L^2}\leq C_0\left(e^{-\lambda t}+e^{-\lambda (T-t)}\right)\left(\|\mu_0\|_{L^2}+\|Dv_T \|_{L^2}\right) \qquad \forall t\in [0,T].
$$
\end{Proposition}

Let us start the proof with a Lemma which explains that the solution is uniformly bounded, with a bound depending on $\|\mu_0\|_{L^2}$ only. 

\begin{Lemma}\label{lem:boundsLS}  There is a constant $C_0>0$, depending only on the data, such that, if $(v,\mu)$ is  a solution of the linearized problem \eqref{MFGlili}, then  
\be\label{estate}
\int_0^T \| Dv\|_{L^2}^2+ \sup_{t\in [0,T]} \left(\|\mu(t)\|_{L^2}^2+ \|Dv(t)\|_{L^2}^2\right) \leq C_0\left(\|\mu_0\|^2_{L^2}+\|Dv_T \|_{L^2}^2\right). 
\ee
\end{Lemma}

\begin{proof} Note that $\inte \mu(t)=0$ for any $t$. By standard duality techniques, we have, for any $0\leq t_1\leq t_2\leq T$, 
$$
\begin{array}{l}
\ds \int_{t_1}^{t_2} \int_{\T^d\times \T^d}\frac{\delta F}{\delta m} (x, \bar m,y)\mu(t,y)\mu(t,x)dydxdt+ \int_{t_1}^{t_2} \inte \bar m H_{pp}(x,D\bar u(x))Dv(t,x)\cdot Dv(t,x)\ dxdt \\
\qquad \ds = -\left[ \inte v\mu\right]_{t_1}^{t_2},
\end{array}
$$
so that, by monotonicity of $F$ and $G$ (see assumption \eqref{hyp:mono}), 
\be\label{neww1}
\begin{split}
C^{-1}\int_0^T \|Dv(t)\|_{L^2}^2dt&  \leq   \inte (v(0)-\lg v(0)\rg ) \mu_0 - \inte (v_T-\lg v_T\rg )\mu(T) \\
& \leq  C (\|Dv(0)\|_{L^2} \|\mu_0\|_{L^2}+ \|Dv_T\|_{L^2}\|\mu(T)\|_{L^2}),
\end{split}
\ee
thanks to Poincar\'e inequality. Using Lemma \ref{lemCLLP2.1}, we have 
$$
\begin{array}{rl}
\ds \|\mu(t)\|_{L^2} \; \leq& \ds Ce^{-\omega t} \|\mu_0\|_{L^2} +C \left[\int_0^t \|\bar m H_{pp}(\cdot, D\bar u)Dv\|_{L^2}^2\right]^{1/2}\\
\leq & \ds Ce^{-\omega t} \|\mu_0\|_{L^2} +C \left[\int_0^T \|Dv\|_{L^2}^2\right]^{1/2} \\
\leq & \ds  Ce^{-\omega t} \|\mu_0\|_{L^2} +C (\|Dv(0)\|_{L^2}^{1/2} \|\mu_0\|_{L^2}^{1/2}+ \|Dv_T\|_{L^2}^{1/2}\|\mu(T)\|_{L^2}^{1/2}).
\end{array}
$$
For $t=T$, we get, after simplification,  
$$
 \|\mu(T)\|_{L^2} \; \leq\; C \left( \|\mu_0\|_{L^2} +\|Dv(0)\|_{L^2}^{1/2} \|\mu_0\|_{L^2}^{1/2}+ \|Dv_T\|_{L^2}\right),
 $$
 from which we deduce that
\be\label{neww2}
\sup_{t\in [0,T]} \|\mu(t)\|_{L^2}  \; \leq\; C \left( \|\mu_0\|_{L^2} +\|Dv(0)\|_{L^2}^{1/2} \|\mu_0\|_{L^2}^{1/2}+ \|Dv_T\|_{L^2}\right).
\ee
Note that the derivative $v_{x_i}$ of $v$ satisfies 
\be\label{eqDv}
\left\{\begin{array}{l}
\ds -\partial_t v_{x_i} -\Delta v_{x_i} +H_p\cdot Dv_{x_i} + D_{x_i}\left[H_p\right] \cdot Dv  = D_{x_i}\frac{\delta F}{\delta m}(x, \bar m)(\mu(t))
\qquad {\rm in}\;  (0,T)\times \T^d,\\
\ds v_{x_i}(T,x)= D_{x_i}\frac{\delta G}{\delta m}(x,\bar m)(\mu(T))+D_{x_i}v_T(x)\qquad {\rm in}\;  \T^d
\end{array}\right.
\ee
where, to simplify the notation, we have set $H_p=H_p(x,D\bar u)$, etc... 
Then Lemma \ref{lemCLLP2.2} gives, in view of our assumptions on $\frac{\delta F}{\delta m}$ and $\frac{\delta G}{\delta m}$,  
\be\label{lkajnzoedb}
\begin{array}{rl}
\ds \|v_{x_i}(t)\|_{L^2} \; \leq & \ds Ce^{-\omega (T-t)}\left(\|Dv_T\|_{L^2}+ \left\|D_{x_i}\frac{\delta G}{\delta m}(\cdot,\bar m)(\mu(T))\right\|_{L^2}\right)\\
& \qquad \ds + C\int_t^T e^{-\omega (s-t)} \left( \left\| D_{x_i}\left[H_p\right] \cdot Dv\right\|_{L^2}+ \left\| D_{x_i}\frac{\delta F}{\delta m}(\cdot, \bar m)(\mu(t))\right\|_{L^2}\right)ds \\
\leq & \ds  \ds Ce^{-\omega (T-t)}(\|Dv_T\|_{L^2}+\|\mu(T)\|_{L^2})+C\int_t^T e^{-\omega (s-t)}\left( \left\|Dv\right\|_{L^2}+ \left\| \mu(t)\right\|_{L^2}\right)ds \\
\leq & \ds  \ds Ce^{-\omega (T-t)}\|Dv_T\|_{L^2}+ C\left(\int_t^T \left\|Dv\right\|_{L^2}^2\right)^{1/2}+ C\sup_{s\geq t} \left\| \mu(s)\right\|_{L^2}.
\end{array}
\ee
Combining this with \rife{neww1} and with the estimate for $\mu$ in \rife{neww2}, we find, for any $t\in [0,T]$, 
\begin{eqnarray*}
\ds \|Dv(t)\|_{L^2} & \leq &  C \left( \|\mu_0\|_{L^2} +\|Dv(0)\|_{L^2}^{1/2} \|\mu_0\|_{L^2}^{1/2}+ \|Dv_T\|_{L^2}\right).
\end{eqnarray*}
In particular, for $t=0$, we get, after simplification:  
$$
\|Dv(0)\|_{L^2}\leq  C \left( \|\mu_0\|_{L^2} + \|Dv_T\|_{L^2}\right), 
$$
which jointly with \rife{neww1} and \rife{neww2} gives the desired statement. 
\end{proof}

\begin{Remark} The above Lemma also provides an argument for proving the existence of  a solution $(v,m)$ to \rife{MFGlili}. Indeed, the a  priori estimate \rife{estate}   allows for a standard application of Schaefer's fixed point theorem by freezing $\mu$ in the right-hand side as well as in the final value of the equation of $v$. 
\end{Remark}

\begin{proof}[Proof of Proposition \ref{Prop:expdecay1}]  For $\tau\geq 0$, let us set 
$$
\rho(\tau)= \sup_{T, t, \mu_0, v_T} \left| \inte \mu(t)v(t)\right|
$$
where the supremum is taken over $T\geq  2\tau$, $t\in [\tau, T-\tau]$,  $\|\mu_0\|_{L^2}\leq 1$ and $\|Dv_T\|_{L^2}\leq 1$, the pair $(v,\mu)$ being the solution to \eqref{MFGlili}. According to Lemma \ref{lem:boundsLS}, $\rho(\tau)$ is bounded for any $\tau$, since, using that $\mu$ has zero average, one has  for any $t$
$$
\left |\inte \mu(t)v(t) \right| \leq C \, \|\mu(t)\|_{L^2} \, \|Dv(t)\|_{L^2}
$$
by Poincar\'e inequality.
By definition, the map $\rho$ is non increasing (since we take the supremum over a set indexed by $\tau$ which decreases for the inclusion as $\tau$ increases). 
Let us denote by $\rho_\infty$ the limit of $\rho(\tau)$ as $\tau \to +\infty$. The key step consists in proving that $\rho_\infty=0$. 

Let $\tau_n\to +\infty$, $T_n\geq 2\tau_n$, $t_n\in [\tau_n,T_n-\tau_n]$, $\mu_0^n$ with $\|\mu_0^n\|_{L^2}\leq 1$ and $v^n_T$ with $\|Dv^n_T\|_{L^2}\leq 1$ be such that 
$$
\left|\inte \mu^n(t_n)v^n(t_n)\right|\geq \rho_\infty-1/n. 
$$
We set 
$$
\tilde \mu^n(t,x)= \mu^n(t_n+t,x), \; \tilde v^n(t,x)= v^n(t_n+t,x)-\lg v^n(t_n)\rg \qquad \forall t\in[-t_n,T_n-t_n], \; x\in \T^d.
$$
By the estimates of Lemma \ref{lem:boundsLS}, the $(\tilde v^n, \tilde \mu^n)$ are locally bounded in $L^2$. By parabolic regularity (from \cite[Theorem III.8.1 p. 196]{LSU} combined with \cite[Theorem III.10.1 p. 204]{LSU} and \cite[Theorem III.11.1 p. 211]{LSU}), the $\tilde v^n$ and $D\tilde v^n$ are locally bounded in $C^{\alpha/2,\alpha}$ while the $\tilde \mu^n$ are bounded in $C^{\alpha/2,\alpha}$ for some $\alpha\in (0,1)$. So the pair $(\tilde v^n, \tilde \mu^n)$ locally uniformly converges to some $(v,\mu)$ which satisfies the linearized MFG system on $\R\times \T^d$. Moreover, we have
$$
\left|\inte \mu(0)v(0)\right|= \lim_n \left| \inte \mu^n(t_n)v^n(t_n)\right| = \rho_\infty. 
$$
On the other hand, for any $t\in \R$ and for $n$ large enough, we have that $t_n+t\in [\tau_n-|t|, T_n-(\tau_n-|t|)]$, so that
$$
\left|\inte \mu(t)v(t)\right| =\lim_n \left|\inte \mu^n(t_n+t)v^n(t_n+t)\right|  \leq \lim_n \rho(\tau_n-|t|) = \rho_\infty. 
$$
The duality equality implies that, for any $t_1\leq t_2$,  we have  
\be\label{lkjernresd}
C^{-1}\int_{t_1}^{t_2} \|Dv\|_{L^2}^2 \leq -\left[\inte \mu v\right]_{t_1}^{t_2}. 
\ee
Therefore the map $t\to \inte \mu(t)v(t)$ is non increasing, with a derivative bounded above by $-\|Dv(0)\|_{L^2}^2$ at $t=0$, while the map $t\to \left|\inte \mu(t)v(t)\right|$ has a maximum $\rho_\infty$ at $t=0$: this implies that $Dv(0)=0$. As $\inte v(0)=0$, we can infer that
$$
\rho_\infty= \left|\inte \mu(0)v(0)\right|=0. 
$$

We now prove that $\rho(t)$ converges to $0$ with an exponential rate. Let $T>0$ and $(v,\mu)$ solution of the MFG linearized system with $\|\mu(0)\|_{L^2}\leq 1$ and $\|Dv_T\|_{L^2}\leq1$. Using Lemma \ref{lemCLLP2.1} and \rife{lkjernresd},
 we have,  for $\tau\geq 0$ and $t\in [\tau, T-\tau]$:
$$
\|\mu(t)\|_{L^2}\leq Ce^{-\omega (t-\tau/2)}\|\mu(\tau/2)\|_{L^2} +C\left(- \left[\inte \mu v \right]_{\tau/2}^{t} \right)^{1/2}\leq 
Ce^{-\omega \tau / 2}+C\left[2 \rho(\tau/2) \right]^{1/2},
$$
because $\mu$ is uniformly bounded in $L^2$ (Lemma \ref{lem:boundsLS}). Thus 
\be\label{jrnezeoli}
\sup_{t\in [\tau,T-\tau]}\|\mu(t)\|_{L^2} \leq C\left(e^{-\omega \tau / 2}+(\rho(\tau/2))^{1/2}\right).
\ee
%As $\rho(\tau)\to0$ when $\tau\to+\infty$, we can choose $\tau$ large enough (independent of $\mu_0$) so that 
%$$
%\sup_{t\in [0,T-2\tau]}\|\mu(\tau+t)\|_{L^2}\leq \theta,
%$$
%where $\theta\in (0,1)$ is arbitrary. \\
%A REPRENDRE\\
Coming back to  \eqref{lkajnzoedb}, we have,   for all $t\in [2\tau, T-2\tau]$, 
\begin{equation}\label{newttau}
\begin{array}{rll}
\|Dv(t)\|_{L^2} & \leq &\ds Ce^{-\omega (T-\tau-t)}\|Dv(T-\tau)\|_{L^2}+ C\left(\int_{t}^{T-\tau} \left\|Dv\right\|_{L^2}^2\right)^{1/2}+ C\sup_{s\in [t, T-\tau]}\left\| \mu(s)\right\|_{L^2} \\
& \leq &\ds Ce^{-\omega\tau}+ C\left(-\left[\inte\mu(s)v(s)\right]_t^{T-\tau}\right)^{1/2}+ C\sup_{s\in [t, T-\tau]}\left\| \mu(s)\right\|_{L^2} \\
& \leq & Ce^{-\omega \tau}+ C \rho^{1/2}(\tau) + C\left(e^{-\omega \tau / 2}+(\rho(\tau/2))^{1/2}\right),
\end{array}
\end{equation}
because $Dv$ is uniformly bounded in $L^2$ (Lemma \ref{lem:boundsLS}). In view of \eqref{jrnezeoli} and \eqref{newttau}, we can fix $\tau>0$ large enough so that, for any $T\geq 4\tau$ and any $(v,\mu)$ as above, one has
$$
\sup_{t\in [2\tau, T-2\tau]} \left( \|\mu(t)\|_{L^2}+\|Dv(t)\|_{L^2}\right) \leq 1/2.
$$
 Notice that, by definition, this also implies that $\rho(2\tau)\leq \frac14$. Now   we can iterate the previous estimate. Indeed,  for  $T\geq 4\tau$, the restriction to $[2\tau, T-2\tau]$ of $(v,\mu)$ is a solution of the linearized MFG system \eqref{MFGlili} on $[2\tau,T-2\tau]$ with boundary conditions $\|\mu(2\tau)\|_{L^2}\leq 1/2$ and $\|Dv(T-2\tau)\|_{L^2}\leq 1/2$. As the problem is invariant by time translation, we deduce that
$$
\sup_{t\in [4\tau, T-4\tau]} \left( \|\mu(t)\|_{L^2}+\|Dv(t)\|_{L^2}\right) \leq 1/4\,,
$$
(and similarly $\rho(4\tau)\leq \frac1{4^2}$). By a standard iteration, this shows that  there exists $\lambda$ such that 
 $$
\|\mu(t)\|_{L^2}+   \|Dv(t)\|_{L^2} \;  \leq  \; C(e^{-\lambda t}+e^{-\lambda (T-t)}) \qquad \forall t\in [0,T].
$$
 \end{proof}
 
\begin{Proposition}\label{Prop:expdecay3} Let $\lambda$ be as in Proposition \ref{Prop:expdecay1}. There exists $C_1$ such that,  if  $B=B(t,x)$ satisfies 
\be\label{condB}
\|B(t)\|_{L^2}\leq e^{-\lambda t}+ e^{-\lambda (T-t)},
\ee
and if $(v,\mu)$ is a solution to the MFG linearized system 
\be\label{MFGlinB}
\left\{\begin{array}{l}
\ds-\partial_t v  -\Delta v +H_p(x,D\bar u).Dv = \frac{\delta F}{\delta m}(x, \bar m)(\mu(t))\qquad {\rm in}\; (0,T)\times \T^d\\
\ds \partial_t \mu-\Delta \mu -\dive(\mu H_p(x,D \bar u))-\dive (\bar m H_{pp}(x,D\bar u)Dv)= \dive (B)\qquad {\rm in}\; (0,T)\times \T^d\\
\mu(0,\cdot)= 0,\;  \qquad v(T,x)=0\qquad {\rm in}\;  \T^d
\end{array}\right.
\ee
then   
$$
\|\mu(t)\|_{L^2}+\|Dv(t)\|_{L^2}\leq C_1\left( (1+t)\, e^{-\lambda t}+ (1+T)\, e^{-\lambda (T- t)}\right)\qquad \forall t\in [0,T].
$$
\end{Proposition}
 
 \begin{proof} Let us first prove that $(v,\mu)$ is bounded. The duality relation gives, for any $0\leq t_1\leq t_2\leq T$, 
$$
C^{-1} \int_{t_1}^{t_2}\|Dv\|_{L^2}^2 dt \leq  -\left[ \inte v\mu\right]_{t_1}^{t_2}- \int_{t_1}^{t_2} \inte B\cdot Dv.
$$
Thus, by Young's inequality,  
$$
C^{-1} \int_{t_1}^{t_2}\|Dv\|_{L^2}^2 dt \leq  -\left[ \inte v\mu\right]_{t_1}^{t_2}+ \int_{t_1}^{t_2} \|B\|_{L^2}^2ds.
$$
Using the homogeneous boundary conditions at $t=0$, $t=T$, we obtain the bound
$$
\int_{0}^{T}\|Dv\|_{L^2}^2 dt\leq C\int_0^T \|B\|_{L^2}^2ds.
$$
This implies, with the same arguments as in Lemma \ref{lem:boundsLS}, 
$$
\sup_{t\in [0,T]} \; \|\mu(t)\|_{L^2} + \|Dv(t) \|_{L^2} \leq C \left[ \int_{0}^{T} \|B\|_{L^2}^2\right]^{1/2}\leq C,
$$
where the last inequality comes from \eqref{condB}. 

 For $\tau\geq 0$,  we set 
\be\label{rho90}
 \rho(\tau) = \sup_{T, t, B} \left(\|\mu(t)\|_2+ \|Dv(t)\|_{L^2}\right)
 \ee
 where the supremum is taken over any $T\geq 2\tau$, $t\in [\tau, T-\tau]$ and any $B$ satisfying \eqref{condB}, the pair $(v,\mu)$ being the solution to \eqref{MFGlinB}. In view of the previous discussion, $\rho(\tau)$ is bounded for any $\tau$. 
 
The restriction $(\tilde v, \tilde \mu)$ of $(v,\mu)$ to $[\tau, T-\tau]$ can be written as 
 $$
 (\tilde v, \tilde \mu)= (\tilde v_1, \tilde \mu_1)+(\tilde v_2, \tilde \mu_2),
 $$
where $(\tilde v_1, \tilde \mu_1)$ solves the homogeneous MFG linearized system \eqref{MFGlili} with boundary conditions $\tilde v_1(T-\tau)= v(T-\tau)$ and $\tilde \mu_1(\tau)= \mu(\tau)$ while $(\tilde v_2, \tilde \mu_2)$ solves the linearized MFG system \eqref{MFGlinB} on the time interval $[\tau, T-\tau]$ with homogeneous boundary conditions. 

From Proposition \ref{Prop:expdecay1}, we have, for any $t\in [\tau, T-\tau]$,  
$$
\begin{array}{rl}
\ds \|\tilde \mu_1(t)\|_{L^2}+\|D\tilde v_1(t)\|_{L^2}\; \leq & \ds  C_0\left( e^{-\lambda (t-\tau)}+e^{-\lambda (T-\tau- t)}\right)\left( \|\mu(\tau)\|_{L^2}+\|Dv(T-\tau)\|_{L^2}\right) \\
\leq & \ds  C \left( e^{-\lambda (t-\tau)}+e^{-\lambda (T-\tau- t)}\right).
\end{array}
$$
Note that the restriction of $B$ to $[\tau, T-\tau]$ satisfies 
$$
\|B(t)\|_{L^2} \leq e^{-\lambda \tau} \left[ e^{-\lambda (t-\tau)}+ e^{-\lambda (T-\tau -t)}\right].
$$
So by the linearity and the invariance in time of the equation, we get 
$$
\|\tilde \mu_2(t)\|_2+ \|D\tilde v_2(t)\|_{L^2}
 \leq  e^{-\lambda \tau} \rho(t-\tau)\qquad \forall t\in [\tau, T-\tau].
 $$
Putting together the estimates of  $(\tilde v_1, \tilde \mu_1)$ and  $(\tilde v_2, \tilde \mu_2)$, we obtain, for any $t\geq \tau$,   
\begin{eqnarray*}
\ds  \sup_{s\in [t+\tau, T-\tau-t]} \left(\|\mu(s)\|_{L^2}+ \|Dv(s)\|_{L^2}\right)
& \leq & \sup_{s\in [t+\tau, T-\tau-t]} C(e^{-\lambda (s-\tau)}+e^{-\lambda (T-\tau- s)})+e^{-\lambda \tau} \rho(s-\tau) \\
& \leq & Ce^{-\lambda t} +e^{-\lambda \tau} \rho(t) .
\end{eqnarray*}
Taking the supremum over $(v,\mu)$ and multiplying by  $e^{\lambda(t+\tau)}$ gives
$$
e^{\lambda(t+\tau)}\rho(t+\tau)\leq Ce^{\lambda \tau} +e^{\lambda t} \rho(t) , 
$$
from which we infer that 
$$
\rho(t)\leq C(1+t) e^{-\lambda t}.
$$
By definition of $\rho$ in \rife{rho90}, this implies the conclusion when choosing $\tau=t$ if $t\in [0,T/2]$ and $\tau= T-t$ otherwise.
   \end{proof}
 
Collecting the  above Propositions we finally  obtain: 
\begin{Theorem}\label{theo:CvExpLin} Let $\lambda$ be as in Proposition \ref{Prop:expdecay1}. There exists $C_0>0$ such that,  if $A=A(t,x)$ and $B=B(t,x)$ satisfy 
\be\label{condAB}
\|A(t)\|_{L^2}+ \|B(t)\|_{L^2}\leq M\left(e^{-\lambda t}+ e^{-\lambda (T-t)}\right),
\ee
and if $(v,\mu)$ is a solution to the MFG linearized system 
\be\label{MFGlinBbis}
\left\{\begin{array}{l}
\ds-\partial_t v  -\Delta v +H_p(x,D\bar u)\cdot Dv = \frac{\delta F}{\delta m}(x, \bar m)(\mu(t))+A(t,x)\qquad {\rm in}\; (0,T)\times \T^d,\\
\ds \partial_t \mu-\Delta \mu -\dive(\mu H_p(x,D \bar u))-\dive (\bar m H_{pp}(x,D\bar u)Dv)= \dive (B)\qquad {\rm in}\; (0,T)\times \T^d,\\
\mu(0,\cdot)= \mu_0,\;  \qquad v(T,x)=\frac{\delta G}{\delta m}(x,\bar m)(\mu(T)) + v_T(x)\qquad {\rm in}\;  \T^d,
\end{array}\right.
\ee
with $\inte \mu_0=0$, we have:  
$$
\|\mu(t)\|_{L^2}+\|Dv(t)\|_{L^2}\leq C_0\left( (1+t)e^{-\lambda t}+ (1+T)e^{-\lambda (T- t)}\right)\left(\|Dv_T\|_{L^2}+\|\mu_0\|_{L^2}+ M\right)
$$
for any $t\in [0,T]$.
\end{Theorem}

\begin{proof} Let $\tilde v$ be the solution to 
$$
\left\{\begin{array}{l}
-\partial_t \tilde v  -\Delta \tilde v +H_p(x,D\bar u).D\tilde v = A(t,x)\qquad {\rm in}\; (0,T)\times \T^d,\\
\tilde v(T,x)=0\qquad {\rm in}\;  \T^d.
\end{array}\right.
$$
Note for later use that, assuming $\lambda<\omega$, we have
\be\label{Dtildev}
\| D\tilde v(t) \|_{L^2} \leq C M (e^{-\lambda t}+ e^{-\lambda (T-t)}). 
\ee
Indeed, using Lemma \ref{lemCLLP2.2}, we have  
$$
\| \tilde v(t)-\lg \tilde v(t)\rg \|_{L^2} \leq C\int_t^Te^{-\omega(s-t)} \|A(s)\|_{L^2}ds\leq C M (e^{-\lambda t}+ e^{-\lambda (T-t)}). 
$$
Then the regularizing property of the equation leads to \eqref{Dtildev}. 

The pair $(v_1,\mu_1):= (v-\tilde v,\mu)$ solves 
$$
\left\{\begin{array}{l}
\ds-\partial_t v  -\Delta v_1 +H_p(x,D\bar u).Dv_1 = \frac{\delta F}{\delta m}(x, \bar m)(\mu(t))\qquad {\rm in}\; (0,T)\times \T^d,\\
\ds \partial_t \mu_1-\Delta \mu_1 -\dive(\mu_1 H_p(x,D \bar u))-\dive (\bar m H_{pp}(x,D\bar u)Dv_1)\\
\ds \qquad \qquad \qquad \qquad \qquad \qquad \qquad \qquad \qquad \qquad = \dive (B+\bar m H_{pp}(x,D\bar u)D\tilde v)\; {\rm in}\; (0,T)\times \T^d,\\
\mu_1(0,\cdot)= \mu_0,\;  \qquad v_1(T,x)=\frac{\delta G}{\delta m}(x,\bar m)(\mu_1(T))+ v_T(x)\qquad {\rm in}\;  \T^d,
\end{array}\right.
$$
where, by \eqref{condAB} and \eqref{Dtildev}, 
$$
\|B(t)+\bar m H_{pp}(x,D\bar u)D\tilde v(t)\|_{L^2}\leq C M \left(e^{-\lambda t}+ e^{-\lambda (T-t)}\right).
$$
Using Propositions \ref{Prop:expdecay1} and \ref{Prop:expdecay3}, we get:
$$
\|\mu_1(t)\|_{L^2}+\|Dv_1(t)\|_{L^2}\leq C\left( (1+t)e^{-\lambda t}+ (1+T)e^{-\lambda (T- t)}\right)\left(\|Dv_T\|_{L^2}+\|\mu_0\|_{L^2}+M\right)
$$
for any $t\in [0,T]$.
Recalling the definition of $(v_1,\mu_1)$ and using again inequality \eqref{Dtildev} gives the result. 
\end{proof}

%%%%%%%%%%%%
\subsection{Estimates for the nonlinear system}\label{stimeT}

Now we consider the nonlinear MFG systems. For the finite horizon problem, we have:

\begin{Theorem} \label{thm:CvExpMFG}
There exists $\gamma>0$ and $C>0$ such that, if  $(u,m)$ is the solution of the MFG system with initial condition $m_0\in \Pk$:
\be\label{MFGtheoExp}
\left\{\begin{array}{l}
\ds-\partial_t u  -\Delta u +H(x,Du)= F(x,m(t))\qquad {\rm in}\; (0,T)\times \T^d\\
\ds \partial_t m-\Delta m -\dive( mH_p(x,D u))=0\qquad {\rm in}\; (0,T)\times \T^d\\
m(0,\cdot)= m_0,\;  \qquad u(T,x)=G(x,m(T))\qquad {\rm in}\;  \T^d
\end{array}\right.
\ee
then,  for some $\alpha\in (0,1)$,
$$
\|Du(t)-D\bar u\|_{C^{1+\alpha}}\leq C\left( e^{-\gamma t}+ e^{-\gamma (T- t)}\right)\qquad  \forall t\in [0, T]
$$
and 
$$
\|m(t)-\bar m\|_{C^\alpha}\leq C\left( e^{-\gamma  t}+ e^{-\gamma (T- t)}\right)\qquad \qquad  \forall t\in [1, T].
$$
In  particular,
$$
\sup_{(t,x)\in [0,T]\times \T^d} \left|u(t,x)-\bar u(x)-\bar \lambda (T-t)\right|\leq C.
$$
\end{Theorem}

\begin{proof} We use a fixed point argument. Let us start with the proof for  initial  and terminal conditions which are sufficiently close to $\bar m$ and $\bar u$ respectively. Let $\hat K>0$ be small enough and $\gamma\in (\lambda/2,\lambda)$, where $\lambda $ is given by Proposition \ref{Prop:expdecay1}. Let $E$ be the set of continuous maps $(v,\mu)$ on $[0,T]\times \T^d$ such that $Dv$ is also continuous and 
$$
 \|Dv(t)\|_{L^\infty}+ \|\mu(t)\|_{L^\infty}\leq \hat K\left(e^{-\gamma t}+ e^{-\gamma(T-t)}\right) .
$$
We suppose that $\hat K$ is such that
$$
\bar m(x) >\hat K \qquad \forall x\in \T^d.
$$
We also assume that the initial condition $m_0$ and the terminal condition $u_T$ are close to $\bar m$ and $\bar u$ (plus constant) respectively, namely that $\mu_0:=m_0-\bar m$ and $v_T:= u_T-\bar u$ satisfy
\be\label{kak}
\|\mu_0\|_{L^\infty}+\|Dv_T\|_\infty \leq \hat K^2. 
\ee
We may suppose further that $\mu_0$ and $Dv_T$ belong to $C^\alpha(\T^d)$ for some $\alpha\in (0,1)$.

For $(v,\mu)\in E$, we consider the solution $(\tilde v,\tilde \mu)$ to the linearized system
$$
\left\{\begin{array}{l}
\ds-\partial_t \tilde v  -\Delta \tilde v +H_p(x,D\bar u)\cdot D\tilde v = \frac{\delta F}{\delta m}(x, \bar m)(\tilde \mu(t))+A(t,x)\qquad {\rm in}\; (0,T)\times \T^d\\
\ds \partial_t \tilde \mu-\Delta \tilde \mu -\dive(\tilde \mu H_p(x,D \bar u))-\dive (\bar m H_{pp}(x,D\bar u)D\tilde v)= \dive (B)\qquad {\rm in}\; (0,T)\times \T^d\\
\mu(0,\cdot)= \mu_0,\;  \qquad v(T,x)=v_T(x)\qquad {\rm in}\;  \T^d
\end{array}\right.
$$
with 
$$
A(t,x)= -H(x,D(\bar u+ v))+ H(x,D\bar u)+ H_p(x, D\bar u)\cdot Dv +F(x, \bar m +\mu)-F(x, \bar m)-\frac{\delta F}{\delta m}(x, \bar m)(\mu)
$$
and 
$$
B(t,x)= 
(\bar m +\mu)H_p(x, D(\bar u+ v))-\bar m H_p(x, D\bar u)
- \mu H_p(x,D\bar u)-\bar m H_{pp}(x,D\bar u)Dv.
$$
We note  that $ \bar m +\mu\geq 0$ on $[0,T]\times \T^d$ and 
$$
\|A(t)\|_{L^\infty}+\|B(t)\|_{L^\infty} \leq C\hat K^2 \left( e^{-2\gamma t}+ e^{-2\gamma(T-t)}\right).
$$
 Here we used that $m \mapsto \frac{\delta F}{\delta m}(x,m,y)$ is Lipschitz (uniformly with respect to $(x,y)$) and the Lipschitz character of $H_{pp}$ as well.

{From} Theorem \ref{theo:CvExpLin} we have, as $\gamma\in (\lambda/2,\lambda)$, 
$$
\|\tilde \mu(t)\|_{L^2}+ \|D\tilde v(t)\|_{L^2} \leq C\hat K^2 \left((1+t)e^{-\lambda t}+(1+T)e^{-\lambda(T-t)}\right).
$$
We upgrade the previous estimates to $L^\infty$ norms with our usual arguments: from Lemma \ref{lemCLLP2.2} we have
\begin{align*}
\|\tilde v(t)-\lg \tilde v(t)\rg \|_{L^\infty} & \leq Ce^{-\omega (T-t)} \|v(T)-\lg v(T)\rg \|_{L^\infty} 
\\
& \qquad \qquad +C\int_t^T e^{-\omega(s-t)} (\|\frac{\delta F}{\delta m}(x, \bar m)(\tilde\mu(s))\|_
{L^\infty}+ \|A(s)\|_{L^\infty})ds 
\\
& \leq 
C\hat K^2 \left((1+t)e^{-\lambda t}+(1+T)e^{-\lambda(T-t)}\right)\,.
\end{align*}
Then, in any   interval $[t,t+1]$, we have
$$
\|D\tilde v(t)\|_\infty  \leq C  \, \sup\limits_{s\in (t, t+1)} [\|\tilde v(s)-\lg \tilde v(s)\rg\|_\infty + \|\frac{\delta F}{\delta m}(x, \bar m)(\tilde\mu(s))\|_
{L^\infty}+\|A(s)\|_\infty]\,,
$$
and this concludes the estimate for $\|D\tilde v(t)\|_{L^\infty}$. Now, using the bound for $D\tilde v$   and $B$, we have
$$
\| \tilde \mu(t)\|_\infty\leq C \sup\limits_{s\in (t-1, t)} [\|\tilde \mu(s) \|_{L^2}+ \|D\tilde v(s)\|_
{L^\infty}+\|B(s)\|_\infty]
$$
and we conclude the estimate for $\| \tilde \mu(t)\|_\infty$. Notice that the above bounds hold up to $t=0$ and $t=T$ by using the condition \rife{kak}  assumed on $\mu_0$ and $v_T$. Eventually, we obtained that 
%
%As $\tilde \mu$, $\tilde v$ and $D \tilde v$ solve linear parabolic equations with bounded coefficients, parabolic estimates (\cite[Theorem III.8.1 p. 196]{LSU}, \cite[Theorem III.10.1 p. 196]{LSU} and \cite[Theorem III.11.1 p. 196]{LSU})  imply that $\tilde \mu$, $\tilde v$ and $D\tilde v$ are locally bounded in $C^{\alpha/2,\alpha}$ (for any $\alpha\in (0,1)$) by a constant depending on the $L^2$ norm of the solution. Hence,  the above estimate implies that 
$$
\|\tilde \mu(t)\|_{L^\infty}+ \|D\tilde v(t)\|_{L^\infty} \leq C\hat K^2 \left((1+t)e^{-\lambda t}+(1+T)e^{-\lambda(T-t)}\right).
$$
Since $\gamma<\lambda$,  for $\hat K$ small enough we infer  that
$$
\|\tilde \mu(t)\|_{L^\infty}+ \|D\tilde v(t)\|_{L^\infty} \leq  \hat K\left(e^{-\gamma t}+ e^{-\gamma(T-t)}\right)
$$
and $(\tilde v, \tilde \mu)$ belongs to $E$. In addition, $\tilde \mu$ and  $\tilde v- \lg \tilde v \rg$ solve linear parabolic equations with bounded coefficients, so classical parabolic estimates (\cite[Theorem III.8.1 p. 196]{LSU}, \cite[Theorem III.10.1 p. 196]{LSU} and \cite[Theorem III.11.1 p. 196]{LSU})  imply that $\tilde \mu$  and $D\tilde v$  are locally bounded in $C^{\alpha/2,\alpha}$  for some $\alpha\in (0,1)$, with bounds that only depend on the $L^\infty$ norm of the coefficients. In particular,  the map $(v,\mu)\to (\tilde v,\tilde \mu)$ is compact and  it has a fixed point $(v,\mu)$. Then $(u,m):= (\bar u,\bar m)+ (v,\mu)$ is a solution to \eqref{MFGtheoExp}  with terminal condition $u_T$ and  which satisfies the decay 
$$
\|Du(t)-D\bar u(t)\|_{C^\alpha}+ \|m(t)-\bar m\|_{C^\alpha}\leq \hat K\left(e^{-\gamma t}+ e^{-\gamma(T-t)}\right) .
$$
  We now remove the smallness and regularity assumptions on the initial condition $m_0$  and  the terminal condition $u_T$. Let $(u,m)$ be the solution to \eqref{MFGtheoExp}. From Lemma \ref{lem.boundm} there exists $0<\tau<\hat T$ 
such that, if  
$T\geq \hat T$, then the solution to \eqref{MFGtheoExp} satisfies, again for some $\alpha\in (0,1)$, 
\be\label{akjbzednl}
\|m(t)-\bar m\|_{C^\alpha} + \|Du(t)-D\bar u\|_{C^\alpha}  \leq \hat K^2\qquad \forall t\in [\tau,T-\tau].
\ee
From the first step we conclude that 
$$
\|m(t)-\bar m\|_{C^\alpha}+ \|Du(t)-D\bar u\|_{C^\alpha}\leq \hat K\left(e^{-\gamma (t-\tau)}+ e^{-\gamma (T-\tau-t)}\right)\qquad \forall t\in [\tau, T-\tau].
$$
Using Lemma \ref{lemCLLP2.1} and changing the constant if necessary, we can extend this inequality for $m$ to the time interval $[1,T]$. Moreover,  $Du(t)-D\bar u$ also satisfies a parabolic equation with uniformly  bounded coefficients. Thus it is bounded in $C^{1+\alpha/2,1+\alpha}$ (for some possibly different $\alpha$, depending on the data only) and we can improve the above inequality for $u$ into 
$$
\|Du(t)-D\bar u\|_{C^{1+\alpha}}\leq C\left(e^{-\gamma t}+ e^{-\gamma (T-t)}\right)\qquad \forall t\in [0, T].
$$

We finally prove the last bound on $v:=u-\bar u-\bar \lambda (T-t)$. Note that $v$ satisfies 
$$
-\partial_t v-\Delta v = A(t,x)
$$
where 
$$
A(t,x)= -(H(x,Du)-H(x,D\bar u))+F(x,m(t))-F(x,\bar m), 
$$
so that 
$$
\|A(t)\|_{L^\infty}\leq C\left(e^{-\gamma t}+ e^{-\gamma (T-t)}\right)\qquad \forall t\in [0, T].
$$
Thus, by standard heat estimate, 
$$
\|v(t)\|_{L^\infty} \leq Ce^{-\omega (T-t)} \|v(T)\|_{L^\infty} +C\int_t^Te^{-\omega (s-t)} \|A(s)\|_{L^\infty}ds\leq C.
$$
\end{proof}

Let us stress that  the above proof provides an explicit smallness estimate on $D(u-\bar u)$ and $m-\bar m$ for initial-terminal data which are correspondingly small. This allows us to  derive the convergence of $u^T(0,x)$ as the time horizon tends to infinity,  for the special case with initial measure $m_0=\bar m$. This result is a first  key argument in the analysis of the long time behavior of the general MFG system and of the master equation (Theorem \ref{thm.main} and Corollary \ref{cor:krjehnrf}). 

\begin{Proposition} \label{prop:CvuT(0)}
For any $T>0$, let $(u^T,m^T)$ be the solution to 
\be\label{eq:uTbarm}
\left\{\begin{array}{l}
\ds-\partial_t u^T  -\Delta u^T +H(x,Du^T)= F(x,m^T(t))\qquad {\rm in}\; (0,T)\times \T^d,\\
\ds \partial_t m^T-\Delta m^T -\dive( m^TH_p(x,D u^T))=0\qquad {\rm in}\; (0,T)\times \T^d,\\
m^T(0,\cdot)= \bar m,\;  \qquad u^T(T,x)=G(x,m(T))\qquad {\rm in}\;  \T^d.
\end{array}\right.
\ee
Then there exists a constant $\bar c$ such that 
$$
\lim_{T\to+\infty} u^T(0,x)-\bar \lambda T= \bar u(x)+\bar c,
$$
where the limit is uniform in $x\in \T^d$. 
\end{Proposition}

\begin{proof} The proof consists in showing that the quantity $u^T(0,x)-\bar \lambda T- \bar u(x)$ is  Cauchy in $T$ in the uniform topology and converges to a constant. In a first step, we show that there exists $\tau>0$ large enough such that $u^T(T-\tau)$ and $u^{T'}(T'-\tau)$ are close in $L^\infty$ for $T,T'\geq 2\tau$. Then we use Theorem \ref{thm:CvExpMFG} (and its proof) to extend this proximity up to time $t=0$.

Let us fix $\ep>0$ small. Theorem \ref{thm:CvExpMFG} states that
\be\label{lerhbzlb}
\|Du^T(t)-D\bar u\|_{C^{1+\alpha}}+ \|m^T(t)-\bar m\|_{L^\infty} \leq C\left( e^{-\gamma t}+ e^{-\gamma (T- t)}\right)\qquad  \forall t\in [1, T]
\ee
for some constant $C$ independent of $T$. Fix $\tau$ large enough and let $T,T'\geq 2\tau$. 
%So, for any $\delta>0$,   we can choose $\tau>0$ such that 
%\be\label{tau}
%\|Du^T(T-\tau)-D\bar u\|_{C^{1+\alpha}}+ \|m^T(T-\tau)-\bar m\|_{L^\infty} \leq \de
%\ee
%for all $T$ sufficiently large. 
If we consider $(\hat u^T,\hat m^T)(t,x):= (\hat u^T,\hat m^T)(t+T, x)$ and $(\hat u^{T'},\hat m^{T'})(t,x):= (\hat u^{T'},\hat m^{T'})(t+T',x)$,  those are both solutions of the MFG system in $(-\tau,0)$, so the energy inequality gives
\begin{align*}
C^{-1} \int_{-\tau}^0 \inte ({\hat m}^T+{\hat m}^{T'}) |D(\hat u^T-\hat u^{T'})|^2 & \leq -\left[  \inte (\hat u^T(t)-\hat  u^{T'}(t))(\hat m^T(t)-\hat m^{T'}(t))\right]_{-\tau}^{0}
\\
& \leq \inte (\hat u^T(-\tau )-\hat  u^{T'}(-\tau))(\hat m^T(-\tau)-\hat m^{T'}(-\tau)),
\end{align*}
where we used that $(\hat u^T-\hat  u^{T'})(0)= G(x,\hat m^T(0))-G(x,\hat m^{T'}(0))$ and the monotonicity of $G$. Using \eqref{lerhbzlb} and the fact that $T,T'\geq 2\tau$ we deduce that
$$
\int_{-\tau}^0 \inte ({\hat m}^T+{\hat m}^{T'}) |D(\hat u^T-\hat u^{T'})|^2 \leq C e^{-2\gamma \tau}\,,
$$
where $C$ is independent of $T,T'$.
Now we apply   Lemma \ref{lemCLLP2.1} and \eqref{lerhbzlb} to  $\hat m^T-\hat m^{T'}$ in the interval $(-\tau,0)$ and we get
$$
\begin{array}{rl}
\ds  \|\hat m^T(t)-\hat m^{T'}(t)\|_{L^2} \; \leq  & \ds C \|\hat m^T(-\tau)-\hat m^{T'}(-\tau)\|_{L^2} + C\left(  \int_{-\tau}^0 \inte ({\hat m}^{T'})^2 |D(\hat u^T-\hat u^{T'})|^2dt\right)^{1/2}\\
\leq & \ds C e^{-\gamma \tau}.
\end{array}
$$
In particular, by the assumptions on $F,G$, there exists $C>0$ such that 
$$
\sup_{t\in (-\tau,0)} \|F(x,\hat m^T(t))-F(x,\hat m^{T'}(t))\|_{L^\infty} + \|G(x,\hat m^T(0))-G(x,\hat m^{T'}(0))\|_{L^\infty} \leq   C e^{-\gamma \tau}\,.
$$
By comparison principle between $\hat u^T$ and $\hat u^{T'}$ in $(-\tau,0)$, we conclude that
\begin{align*}
\|\hat u^T(-\tau)-\hat u^{T'}(-\tau)\|_\infty \leq C (1+\tau) e^{-\gamma \tau}\,.
\end{align*}
Hence we can choose  $\tau$ sufficiently large such that 
\be\label{cqebsd1}
\|u^T(T-\tau)-u^{T'}(T'-\tau)\|_\infty \le \ep 
\ee
for any $T, T'$ large enough. 

Now  we extend the proximity of $u^T$ and $u^{T'}$ up to time $t=0$. Recalling that, by \eqref{lerhbzlb},  $\|Du^T(T-\tau)-D\bar u\|_\infty\leq \ep$ for any $T$ large enough, there  exists $\bar c_0(T)$ such that
\be\label{cqebsd2}
\|u^T(T-\tau)-\bar u-\bar c_0(T)\|_\infty \le C\ep.  
\ee
Note that \eqref{cqebsd1} implies that $(\bar c_0(T))$ is Cauchy as $T\to+\infty$ and thus converges to a limit $\bar c$. 
Let     $\gamma>0$ be defined in the first step of the  proof of Theorem \ref{thm:CvExpMFG}; since $(u^T, m^T)$ satisfy \rife{kak} with $\hat K=\ep^{1/2}$, we can choose   $\ep$  small enough so that the fixed point argument of Theorem \ref{thm:CvExpMFG} applies.  Then,  the restriction of  $(u^T,m^T)$ to $[0,T-\tau]$ satisfies:
\be\label{hbzred}
\|Du^T(t)-D\bar u\|_{L^\infty}+ \|m^T(t)-\bar m\|_\infty\leq \ep^{1/2} \left(e^{-\gamma t}+ e^{-\gamma(T-\tau-t)}\right)  \qquad \forall t\in [0,T-\tau].
\ee
Integrating in space  the equation satisfied by $u^T-\bar u$, we get 
$$
\begin{array}{l}
\ds  \left|\inte (u^T(0)-u^T(T-\tau)) -\bar \lambda (T-\tau)\right| \\
\qquad \ds   \leq \int_0^{T-\tau}\inte  |H(x,Du^T)-H(x,D\bar u)|+ |F(x,m^T(t))-F(x,\bar m)| \ dxdt \; \leq\; C\ep^{1/2}. 
\end{array}
$$
Using \eqref{hbzred} (at time $t=0$ and at time $t=T-\tau$) and Poincar\'e inequality, we infer therefore that 
\be\label{cqebsd3}
\|u^T(0)-u^T(T-\tau) -\bar \lambda (T-\tau)\|_\infty  \; \leq\; C\ep^{1/2}. 
\ee
Combining \eqref{cqebsd1}, \eqref{cqebsd2} and \eqref{cqebsd3},  we conclude that, for any $T,T'$ large enough, 
$$
\|u^T(0)-\bar u-\bar c_0(T)-\bar \lambda (T-\tau) \|_\infty \leq   C\ep^{1/2}. 
$$
From this we can deduce  that $( u^T(0,x)-\bar \lambda T)$  converges uniformly to $\bar u(x)+\bar c$ as $T$ tends to $\infty$. 
\end{proof}

We also deduce from Theorem \ref{thm:CvExpMFG} crucial estimates for the linearized system around \emph{any} solution $(u,m)$ of \eqref{MFGtheoExp}. 

\begin{Corollary}\label{coro.unifBoundT} There exists $\gamma>0$ and $C>0$ such that, if $(u,m)$ is the solution of the MFG system \eqref{MFGtheoExp},  and if $(v,\mu)$ is the solution to the linearized MFG system
$$
\left\{\begin{array}{l}
\ds -\partial_tv -\Delta  v +H_p(x,D   u)\cdot Dv = \frac{\delta F}{\delta m}(x, m)(\mu)\quad {\rm in}\; (0,T)\times \T^d,\\
\ds \partial_t \mu -\Delta  \mu -\dive( \mu H_p(x,D u))-\dive (mH_{pp}(x,Du)Dv)= 0\quad {\rm in}\; (0,T)\times \T^d,\\
\ds \mu(0,\cdot)=\mu_0, \; v(T, \cdot)= \frac{\delta G}{\delta m}(x,m(T))(\mu(T))\quad {\rm in}\;  \T^d,
\end{array}\right.
$$
with $\inte \mu_0=0$, we have
\be\label{corotesi1}
\|\mu(t)\|_{L^2}+ \|Dv(t)\|_{L^2} \leq C\left( e^{-\gamma t }+ e^{-\gamma (T- t) }\right)\|\mu_0\|_{L^2}
\ee
and, for some $\alpha\in (0,1)$ depending only on the data, 
\be\label{corotesi2}
\sup_{t\in [0,T]} \|v\|_{C^{2+\alpha}} \leq C\|\mu_0\|_{(C^{2+\alpha})'}. 
\ee
\end{Corollary}

\begin{proof} We first need a priori estimates on $(v,\mu)$. To this purpose we assume that $\mu_0\in L^2(\T^d)$, and we proceed exactly as in  Lemma \ref{lem:boundsLS}  obtaining 
\be\label{liabzslid}
\int_0^T\inte m | Dv|^2+ \sup_{t\in [0,T]} \left(\|\mu(t)\|_{L^2}^2+ \|Dv(t)\|_{L^2}^2\right) \leq C_0\|\mu_0\|_{L^2}^2. 
\ee
Next we note that $(v,\mu)$ is the solution to \eqref{MFGlinBbis} with 
$$
A= -(H_p(x,Du)-H_p(x,D \bar u))\cdot Dv + \frac{\delta F}{\delta m}(x, m(t))(\mu(t))- \frac{\delta F}{\delta m}(x, \bar m)(\mu(t)),
$$
$$
B=\mu (H_p(x,D u)-H_p(x,D\bar  u)) + (mH_{pp}(x,Du)- \bar mH_{pp}(x,D\bar u))Dv
$$
and 
$$
v_T(x)=  \frac{\delta G}{\delta m}(x,m(T))(\mu(T))- \frac{\delta G}{\delta m}(x,\bar m)(\mu(T)).
$$
Note that 
$$
\|A(t)\|_{L^2} \leq C\|Du-D\bar u\|_\infty \|Dv\|_{L^2} + C\,  \dk(m(t),\bar m) \,\|\mu(t)\|_{L^2}
$$
while 
$$
\|B(t)\|_{L^2} \leq C\|Du-D\bar u\|_\infty \|\mu(t)\|_{L^2} + C(\dk(m(t),\bar m)+ \|Du(t)-D\bar u\|_\infty) \|Dv(t)\|_{L^2}
$$
and 
$$
\|v_T\|\leq C\, \dk(m(T),\bar m)\, \|\mu(T)\|_{L^2}.
$$
 Here we used once more the Lipschitz character of  $m \mapsto \frac{\delta F}{\delta m}(x,m,y)$,  $m \mapsto \frac{\delta G}{\delta m}(x,m,y)$ and  $p\mapsto H_{pp}(x,p)$.

Using Theorem \ref{thm:CvExpMFG} and \eqref{liabzslid}, we deduce: 
$$
\|A(t)\|_{L^2} + \|B(t)\|_{L^2} \leq C \, \|\mu_0\|_{L^2} \left( e^{-\gamma t}+ e^{-\gamma (T- t)}\right)\,.
$$
Then Theorem \ref{theo:CvExpLin} (used with $\lambda=\gamma$) and the bounds \eqref{liabzslid} imply that 
$$
\|\mu(t)\|_{L^2}+\|Dv(t)\|_{L^2}\leq C\left( (1+t)e^{-\gamma t}+ (1+T)e^{-\gamma (T- t)}\right)\|\mu_0\|_{L^2}.
$$
So we deduce \eqref{corotesi1}, possibly for  a smaller value of $\gamma$. 

Now we upgrade the above estimate by using weaker norms for $\mu_0$ and stronger norms for $v$. For this, we use Lemma \ref{lem:kaeruhzh} below which states that 
$$
\|\mu(1)\|_{L^2}\leq C \|\mu_0\|_{(C^{2+\alpha})'}. 
$$
Applying our previous estimate \eqref{corotesi1} to the time interval $[1,T]$, we find that, for any $t\geq 1$, 
\begin{align*}
\|\mu(t)\|_{L^2}+\|Dv(t)\|_{L^2} & \leq C\left( e^{-\gamma (t-1)}+  e^{-\gamma (T- (t-1))}\right)\|\mu(1)\|_{L^2}\\
& \leq C\left( e^{-\gamma t}+  e^{-\gamma (T- t)}\right)  \|\mu_0\|_{(C^{2+\alpha})'}.
\end{align*}
Lemma \ref{lem:kaeruhzh} also states that 
\be\label{krvbzsdnl}
\sup_{t\in [0,T]} \| v(t)-\lg v(t)\rg \|_{C^{2+\alpha}}+ \sup_{t\in [0,T]}\|\mu(t)\|_{(C^{2+\alpha})'} \leq C \|\mu_0\|_{(C^{2+\alpha})'}\,,
\ee
so that we also have 
$$
\sup_{t\in [0,1]} \|D v(t)\|_{L^2}+ \sup_{t\in [0,1]}\|\mu(t)\|_{(C^{2+\alpha})'} \leq C \|\mu_0\|_{(C^{2+\alpha})'}.
$$
Integrating in space the equation for $v$ and using the above bounds on $Dv$ and $\mu$ then implies that 
$$
|\lg v(t)\rg |\leq C\, \|\mu_0\|_{(C^{2+\alpha})'}\qquad \forall t\in [0,T].
$$
We can the deduce \eqref{corotesi2} from \eqref{krvbzsdnl} and the above inequality. 
\end{proof}

\begin{Lemma}\label{lem:kaeruhzh} Under the assumptions of Corollary \ref{coro.unifBoundT}, there exists a constant $C>0$ (independent of $T$, $m_0$ and $\mu_0$) such that
$$
\sup_{t\in [0,T]} \| v(t)-\lg v(t)\rg \|_{C^{2+\alpha}}+ \sup_{t\in [0,T]}\|\mu(t)\|_{(C^{2+\alpha})'}+ \|\mu(1)\|_{L^2} \leq C \|\mu_0\|_{(C^{2+\alpha})'}\,.
$$
\end{Lemma}

\begin{proof} The duality estimate gives
\be\label{dvazero}
 c \int_0^t \inte m |D v|^2 \leq    \int_0^T \inte m H_{pp}(x,Du) D v \cdot  Dv  \leq   \inte v(0) \mu_0\,,
\ee
where we used that $\frac{\delta G}{\delta m}(x,m(T))$ is a nonnegative operator.
By duality,  we also have
$$
\inte \mu(t) \xi= -\int_0^t \inte mH_{pp}(Du) Dv\cdot D\psi + \inte \psi(0) \mu_0 
$$
where $\psi$ solves  (for some smooth terminal condition $\xi$ at time $t$):
$$
\left\{\begin{array}{l}
\ds -\partial_t \psi -\Delta  \psi +H_p(x,D   u)\cdot D\psi  = 0\quad {\rm in}\; (0,t)\times \T^d,\\
\psi(t, \cdot)= \xi\quad {\rm in}\;  \T^d.
\end{array}\right.
$$
Since, by Lemma \ref{lemCLLP2.2},  $\| \psi(s)-\lg \psi(s)\rg\|_{L^2} \leq c \,e^{-\omega(t-s)} \|\xi\|_{L^2}$, we have  by standard estimates:  
$$
\int_0^t \inte |D\psi|^2 \leq \|\xi\|^2_2 +  C \int_0^t \inte | \psi-\lg \psi\rg |^2 \leq C \|\xi\|_{L^2}^2.
$$
Therefore,
$$
\inte \mu(t) \xi   \leq C\left(\int_0^t\inte m |D v|^2 \right)^{\frac12} \|\xi\|_{L^2} +   \|  \psi(0)-\lg \psi(0)\rg\|_{C^{2+\alpha}} \|\mu_0\|_{(C^{2+\alpha})'}\,.
$$
From \eqref{dvazero} we deduce 
\be\label{2+al}
\inte \mu(t) \xi   \leq C\left(\| v(0)-\lg v(0)\rg\|_{C^{2+\alpha}} \|\mu_0\|_{(C^{2+\alpha})'}\right)^{\frac12} \|\xi\|_{L^2} +   \|  \psi(0)-\lg \psi(0)\rg\|_{C^{2+\alpha}} \|\mu_0\|_{(C^{2+\alpha})'}\,.
\ee
To estimate last term, we note that, if $t\leq 1$, we have by Schauder estimates that 
$$
\| \psi(0)- \lg \psi (0)\rg\|_{C^{2+\alpha}}\leq C\|\xi\|_{C^{2+\alpha}},
$$
while, if $t\geq 1$, we have, by Schauder interior estimates: 
\be\label{psio}
\| \psi(0)-\lg \psi(0)\rg\|_{C^{2+\alpha}}\leq C \| \psi(1)-\lg \psi(1)\rg\|_{L^2}\leq C \|\xi\|_{L^2}\leq C\|\xi\|_{C^{2+\alpha}}.
\ee
Coming back to \rife{2+al} and  taking the supremum over the $\xi$ with $\|\xi\|_{C^{2+\alpha}}\leq 1$,  this implies that:   
\be\label{jhzerbndfi}
\sup_{t\in [0,T]}\|\mu(t)\|_{(C^{2+\alpha})'}    \leq C\left(\| v(0)-\lg v(0)\rg\|_{C^{2+\alpha}}^{1/2} \|\mu_0\|_{(C^{2+\alpha})'}^{\frac12}  +   \|\mu_0\|_{(C^{2+\alpha})'}\right).
\ee
Similarly, from \rife{2+al} and  \rife{psio} we also estimate
\be\label{salah}
\| \mu(1)\|_{L^2} \leq C\left(\| v(0)-\lg v(0)\rg\|_{C^{2+\alpha}}^{1/2} \|\mu_0\|_{(C^{2+\alpha})'}^{\frac12}  +   \|\mu_0\|_{(C^{2+\alpha})'}\right).
\ee
We now have to estimate $ v(0)-\lg v(0)\rg$. First we have, by Lemma \ref{lemCLLP2.2}, that for any $t\in[0,T]$: 
\be\label{regdm}
\begin{split}
\|  v(t)-\lg v(t)\rg\|_\infty &  \leq e^{-\omega (T-t)} \| \frac{\delta G}{\delta m}(x,m(T)) \mu(T)\|_\infty + \int_t^T e^{-\omega (s-t)} \|\frac{\delta F}{\delta m}(\cdot, m(s))(\mu(s))\|_\infty)ds
\\
& 
\leq C \, \sup_{[0,T]} \| \mu(t) \|_{(C^{2+\alpha})'}\,,
\end{split}
\ee
where we used the $C^{2+\alpha}$ character  of  $\frac{\delta F}{\delta m}, \frac{\delta G}{\delta m}$ with respect to $y$.
We also estimate $Dv$ in $L^2$ in terms of the same quantity due to Lemma \ref{lemCLLP2.2}. Next, the regularizing property of the equation for $v-\lg v\rg$ (Theorem IV.9.1 of \cite{LSU}) implies that, for any $t\in[0,T-1/2]$ and any $\beta \in (0,1)$,  
\begin{align*}
\| v(t)-\lg v(t)\rg \|_{C^{1+\beta}} & \leq \| v(t+1/2)-\lg v(t+1/2)\rg\|_\infty+ C \sup_{s\in [t,t+1/2]} \| \mu(s) \|_{(C^{2+\alpha})'}\\
\leq &  C \sup_{[0,T]} \| \mu(s) \|_{(C^{2+\alpha})'},
\end{align*}
(where the constant depends on $\beta$).  Then considering the equation for $v_{x_i}$ (for $i\in\{1, \dots, d\}$) and using the uniform $C^2$ regularity of $u$ as well as the $C^2$ regularity of $D_x \frac{\delta F}{\delta m}$ in the $y$ variable as in \rife{regdm}, we obtain in the same way, for any $t\in [0,T-1]$: 
\begin{align*}
\|v_{x_i}(t)\|_{C^{1+\beta}} & \leq \|v_{x_i}(t+1/2)\|_\infty+ C\left( \sup_{s\in [t,t+1/2]} \|Dv(s)\|_\infty+  \sup_{s\in [t,t+1/2]} \| \mu(s) \|_{(C^{2+\alpha})'}\right)
\\
 & \leq  C \sup_{[0,T]} \| \mu(s) \|_{(C^{2+\alpha})'}.
\end{align*}
Choosing $\beta=\alpha$, we have proved therefore that 
$$
\sup_{s\in [0,T-1]} \| v(s)-\lg v(s)\rg\|_{C^{2+\alpha}} \leq  C \sup_{s\in [0,T]} \| \mu(s) \|_{(C^{2+\alpha})'}.
$$
Using  this inequality for $\| v(0)-\lg v(0)\rg\|_{C^{2+\alpha}}$ into \eqref{jhzerbndfi} then gives 
$$
\sup_{t\in [0,T]}\|\mu(t)\|_{(C^{2+\alpha})'}    \leq C\|\mu_0\|_{(C^{2+\alpha})'}, 
$$
which in turn implies that 
$$
\sup_{t\in [0,T-1]} \| v(s)-\lg v(s)\rg\|_{C^{2+\alpha}} \leq  C\|\mu_0\|_{(C^{2+\alpha})'}.
$$
Note that we can extend this inequality to the time interval $[T-1,T]$ by using the regularity of the equation satisfied by $v$ on this interval, the regularity of the terminal condition and the bound on $\|\mu(t)\|_{(C^{2+\alpha})'}$.

In the same way, from \rife{salah} we obtain 
$$
\|\mu(1)\|_{L^2}   \leq C\|\mu_0\|_{(C^{2+\alpha})'}.
$$
\end{proof}

\begin{Remark}\label{milder}
In order to estimate $v$ in the $C^2$- norm, we have used in Lemma \ref{lem:kaeruhzh} the regularity condition (FG3) on the couplings. However, by only using condition (FG2), we could similarly obtain a  milder estimate as
\be\label{c1}
\sup_{t\in [0,T]} \| v(t)-\lg v(t)\rg \|_{C^1}+ \sup_{t\in [0,T]}\|\mu(t)\|_{(C^1)'}  \leq C \|\mu_0\|_{(C^1)'}\,.
\ee
Indeed, a similar  estimate as \rife{regdm} would hold in terms of  $\| \mu(t) \|_{(C^{1})'}$ by using condition (FG2), since
\begin{align*}
\|  v(t)-\lg v(t)\rg\|_\infty &  \leq e^{-\omega (T-t)} \| \frac{\delta G}{\delta m}(x,m(T)) \mu(T)\|_\infty + \int_t^T e^{-\omega (s-t)} \|\frac{\delta F}{\delta m}(\cdot, m(s))(\mu(s))\|_\infty)ds
\\
& 
\leq C \, \sup_{[0,T]} \| \mu(t) \|_{(C^1)'}\,,
\end{align*}
where we only used that $\frac{\delta F}{\delta m}, \frac{\delta G}{\delta m}$ are $C^1$ and globally Lipschitz w.r.t. to $y$. 
Under the same condition the estimate for $Dv$ in $L^\infty$ would follow. Eventually, with the same strategy as in the above proof, by using $C^1$, rather than $C^{2+\alpha}$, estimates on $v$,  we would get at \rife{c1}.
\end{Remark}

%%%%%%%%%%%%%%%%%%%%%%%%%%%
%%%%%%%%%%%%%%%%%%%%%%%%%%%%
\section{Exponential rate of convergence for the infinite horizon MFG system}\label{sec:ExpCvRates2}

Throughout this section we study the infinite horizon discounted problem and show an exponential convergence towards a stationary solution. The existence of this solution is new, as well as the convergence rate towards this solution. The method of proof is close to  the one employed in the previous Section for the finite horizon.

%%%%%%%%%%%
\subsection{The stationary solution of the infinite horizon problem}

\begin{Proposition}\label{P31} There exists $\delta_0>0$ such that, if $\delta \in (0,\delta_0)$, there is a unique solution $(\bar u^\delta,\bar m^\delta)$ to the problem 
\be\label{barudeltamdelta}
\left\{\begin{array}{l}
\delta \bar u^\delta-\Delta \bar u^\delta+H(x, D\bar u^\delta)= F(x, \bar m^\delta) \qquad {\rm in }\; \T^d\\
-\Delta \bar m^\delta-\dive (\bar m^\delta H_p(x,D\bar u^\delta))=0\qquad {\rm in }\; \T^d\\
\ds \inte \bar m^\delta=1, \; \inte \bar u^\delta=0.
\end{array}\right.
\ee
Moreover, for any $\delta\in (0,\delta_0)$,
$$
\|D\bar u^\delta\|_\infty+  \delta \|\bar u^\delta\|_\infty + \|\bar m^\delta\|_{\infty} \leq C\qquad {\rm
and 
}\qquad
\bar m^\delta(x) \geq C^{-1} \quad \forall x\in \T^d,
$$
for some  constant $C>0$. 
\end{Proposition}

\begin{proof} The existence of a  solution can be achieved by a standard fixed point argument, so we omit it. In the same way, the regularity of $\bar u^\delta$ and $\bar m^\delta$ is standard. The strong maximum principle implies that  $m^\delta$ is bounded below by a constant independent of $\delta$.  For proving the uniqueness, we argue as usual by duality: let $(u_1, m_1)$ and $(u_2,m_2)$ be two solutions. By duality argument and Poincar\'e's inequality, we have 
$$
C^{-1} \|D(u_1-u_2)\|_{L^2}^2 \leq \delta \inte (u_1-u_2)(m_1-m_2) \leq C\delta \|D(u_1-u_2)\|_{L^2}\|m_1-m_2\|_{L^2}.
$$
 Thus 
\be\label{more1}
\|D(u_1-u_2)\|_{L^2}\leq C \delta \|m_1-m_2\|_{L^2}. 
\ee
On another hand, by Corollary \ref{cor.StaMeas}, we have 
\be\label{more2}
\|m_1-m_2\|_{L^2} \leq C \|H_p(\cdot, Du_1)-H_p(\cdot, Du_2)\|_{L^2} \leq C \|D(u_1-u_2)\|_{L^2} 
%\leq C\delta \|m_1-m_2\|_{L^2}.
\ee
For $\delta$ small enough, we deduce from \rife{more1}--\rife{more2} that $m_1=m_2$ and  $Du_1=Du_2$, whence $u_1=u_2$. 
\end{proof}

We now note that the solution $(\bar u^\delta, \bar m^\delta)$ is close to $(\bar u,\bar m)$, where $(\bar \lambda,\bar u, \bar m)$ is the solution of the ergodic problem \eqref{e.MFGergo}: 

\begin{Proposition}\label{prop.Cvbarudelta}
We have 
$$
\|\delta \bar u^\delta -\bar \lambda\|_\infty+ \|D(\bar u^\delta-\bar u)\|_{L^2}+\|\bar m^\delta -\bar m\|_{L^2}\leq C\delta^{1/2}.
$$
\end{Proposition}

\begin{proof} We use again the duality argument to get 
$$
C^{-1}\|D(\bar u^\delta-\bar u)\|_{L^2}^2 \leq \inte (\delta \bar u^\delta-\bar \lambda)(\bar m^\delta-\bar m)\leq C\delta \|D\bar u^\delta\|_\infty\leq C\delta.
$$
Thus 
$$
\|D(\bar u^\delta-\bar u)\|_{L^2}\leq C\delta^{1/2}.
$$
By Corollary \ref{cor.StaMeas}, we have
$$
\|\bar m^\delta -\bar m\|_{L^2} \leq C \|D(\bar u^\delta-\bar u)\|_{L^2} \leq C\delta^{1/2}.
$$
The estimate between $\delta \bar u^\delta$ and $\bar \lambda$ then comes from the comparison principle.
\end{proof}

%%%%%%%%%%
\subsection{Exponential rate for the linearized system}

Let $(\bar u^\delta, \bar m^\delta)$ be the solution to \eqref{barudeltamdelta}. We consider the linearized discounted problem around this solution: 
\be\label{eq.discApp1}
\left\{\begin{array}{l}
\ds-\partial_t v +\delta v -\Delta v +H_p(x,D\bar u^\delta)\cdot Dv = \frac{\delta F}{\delta m}(x, \bar m^\delta)(\mu(t))\qquad {\rm in}\; (0,+\infty)\times \T^d,\\
\ds \partial_t \mu-\Delta \mu -\dive(\mu H_p(x,D \bar u^\delta))-\dive (\bar m^\delta H_{pp}(x,D\bar u^\delta)Dv)= 0\qquad {\rm in}\; (0,+\infty)\times \T^d,\\
\mu(0,\cdot)= \mu_0\; {\rm in}\;  \T^d, \qquad v\; {\rm bounded,} 
\end{array}\right.
\ee
 with $\inte \mu_0=0$. As in Section \ref{sectionlinea}, the existence of a solution to \rife{eq.discApp1} can be proved for $\mu_0\in L^2(\T^d)$  by  using  fixed point arguments and 
relying on the conditions enjoyed by $\frac{\delta F}{\delta m}$ and the smoothness of $(\bar u^\delta,\bar m^\delta)$. In particular, one can first solve the system in a finite horizon $t\in (0,n)$ with terminal condition $v(n)=0$, and then obtain a solution to \rife{eq.discApp1} by letting $n\to \infty$. Since $\de>0$, here $\|\frac{\delta F}{\delta m}\|_\infty \, \de^{-1}$ is a uniform bound with respect to $n$ and leads to  a bounded $v$ in  \rife{eq.discApp1}.

In the rest of this paragraph, we are going to show that $v$ actually enjoys a bound which is uniform in $\delta$ and that  $\mu, Dv$ decay exponentially  in $L^2$   as $t\to \infty$, uniformly with respect to $\delta$.

\begin{Lemma}\label{lem:prelimDisc}  Let $(v,\mu)$ be a   solution to \eqref{eq.discApp1}. Then we have 
$$
 \inte \mu(t)v(t)\geq 0\qquad \forall t\geq 0
$$
and there exists a constant $C_0>0$, independent of $\mu_0$ and $\delta$,  such that, for any $t\geq 0$,  
$$
\|\mu(t)\|_{L^2}+ \|Dv(t) \|_{L^2}\leq C_0\|\mu_0\|_{L^2}e^{\delta t/2}.
$$
\end{Lemma}

\begin{proof} We consider the duality between $e^{-\delta  t} v$ and $\mu$; using properties of $(\bar u_\de, \bar m_\de)$ from Proposition \ref{P31} we get
\be\label{dualineq}
C^{-1}\int_{t_1}^{t_2} e^{-\delta t} \|Dv(t)\|_{L^2}^2dt \leq -\left[ e^{-\delta t} \inte  v(t)\mu(t)\right]_{t_1}^{t_2}.
\ee
Next we claim that 
\be\label{kuyqsldlck}
C^{-1}  \int_{0}^{\infty} e^{-\delta t} \|Dv(t)\|_{L^2}^2 dt \leq \inte \mu_0v(0)\leq C\|\mu_0\|_{L^2}\|v(0)-\lg v(0)\rg\|_{L^2}.
\ee
This inequality is obvious from the duality inequality if we know that  the limit $e^{-\delta t} \inte  v(t)\mu(t)$ vanishes as $t\to+\infty$. For this we need a first rough bound on $\mu$. By Lemma \ref{lemCLLP2.1} we have 
$$
\|\mu(t)\|_{L^2}\leq Ce^{-\omega t} \|\mu_0\|_{L^2} + C\left[ \int_0^t \|Dv(s)\|^2_{L^2}ds \right]^{1/2}.
$$
By \rife{dualineq}, we get
$$
\begin{array}{rl}
\ds \|\mu(t)\|_{L^2}\; \leq & \ds Ce^{-\omega t} \|\mu_0\|_{L^2} + Ce^{\delta t/2}\left[ \int_0^t e^{-\delta s} \|Dv(s)\|^2_{L^2}ds \right]^{1/2} \\
\leq & \ds Ce^{-\omega t} \|\mu_0\|_{L^2} + Ce^{\delta t/2} \|v\|_\infty^{\frac12} \left[ \|\mu_0\|_{L^2}^{1/2}+   e^{-\delta t/2}\|\mu(t)\|_{L^2}^{1/2}\right],
\end{array}
$$
so that 
$$
\|\mu(t)\|_{L^2}\; \leq\; C_\delta e^{\delta t/2}
$$
where $C_\delta$ depends on $\mu_0$ and $\delta$. This inequality then implies that 
$$
\lim_{t\to+\infty} e^{-\delta t} \inte \mu(t)v(t)= 0
$$
and \eqref{kuyqsldlck} holds. Note that \eqref{dualineq} implies that the map $t\to e^{-\delta t} \inte \mu(t)v(t)$ is nonincreasing, and we just proved that it has limit $0$ as $t\to+\infty$. Thus it is nonnegative. \\

In the light of \eqref{kuyqsldlck} we revisit the estimate of $\mu$. We have 
$$
\begin{array}{rl}
\ds \|\mu(t)\|_{L^2}\; \leq & \ds Ce^{-\omega t} \|\mu_0\|_{L^2} + Ce^{\delta t/2}\left[ \int_0^t e^{-\delta s} \|Dv(s)\|^2_{L^2}ds \right]^{1/2} \\
\leq & \ds Ce^{-\omega t} \|\mu_0\|_{L^2} + Ce^{\delta t/2} \|\mu_0\|_{L^2}^{1/2}\|v(0)-\lg v(0)\rg\|_{L^2}^{1/2}.
\end{array}
$$
We plug this inequality into the usual estimate for $v$: for any $0\leq t \leq t_1$, 
$$
\begin{array}{rl}
\ds \|v(t)-\lg v(t)\rg\|_{L^2} \; \leq & \ds Ce^{-\omega(t_1-t)} \|v(t_1)-\lg v(t_1)\rg\|_{L^2}+C\int_t^{t_1} e^{-\omega(s-t)} \|\mu(s)\|_{L^2}ds\\
\leq & \ds Ce^{-\omega(t_1-t)} \|v(t_1)-\lg v(t_1)\rg\|_{L^2}\\
& \ds \qquad +C\int_t^{t_1}e^{-\omega(s-t)} \left(e^{-\omega s} \|\mu_0\|_{L^2} + Ce^{\delta s/2} \|\mu_0\|_{L^2}^{1/2}\|v(0)-\lg v(0)\rg\|_{L^2}^{1/2}\right)ds\\
\leq & \ds Ce^{-\omega(t_1-t)} \|v(t_1)-\lg v(t_1)\rg\|_{L^2}+ C \|\mu_0\|_{L^2} e^{-\omega t}\\
& \qquad \ds + C \|\mu_0\|_{L^2}^{1/2}\|v(0)-\lg v(0)\rg\|_{L^2}^{1/2}e^{\delta t/2}.
\end{array}
$$
Letting $t_1\to+\infty$ gives 
$$
\begin{array}{rl}
\ds \|v(t)-\lg v(t)\rg\|_{L^2} \; \leq\; & \ds C \|\mu_0\|_{L^2} e^{-\omega t}
%\\
%& \qquad \ds 
+ C \|\mu_0\|_{L^2}^{1/2}\|v(0)-\lg v(0)\rg\|_{L^2}^{1/2}e^{\delta t/2}.
\end{array}
$$
Choosing $t=0$ and rearranging we find 
$$
\begin{array}{rl}
\ds \|v(0)-\lg v(0)\rg\|_{L^2} \; \leq\; & \ds C \|\mu_0\|_{L^2}.
\end{array}
$$
So we have for any $t\geq 0$
$$
 \|\mu(t)\|_{L^2}+ \|v(t)-\lg v(t)\rg\|_{L^2} \; \leq\;  C \|\mu_0\|_{L^2}e^{\delta t/2}.
$$
We can then conclude by Lemma \ref{lemCLLP2.2}. 
\end{proof}

\begin{Proposition}\label{prop:discountHomo}  Let $(\bar u^\delta, \bar m^\delta)$ be the solution to \eqref{barudeltamdelta}. There exist  $\de_0, C_0,\lambda >0$ such that, if $(v,\mu)$ is the solution to \eqref{eq.discApp1} associated with $(\bar u^\delta, \bar m^\delta)$  and $\inte \mu_0=0$, and if $\delta\in (0,\delta_0)$, then 
$$
\|Dv(t)\|_{L^2}+\|\mu(t)\|_{L^2}\leq C_0 \|\mu_0\|_{L^2}e^{-\lambda t} \qquad \forall t\geq0.
$$
In particular,
$$
\left\|v\right\|_{L^\infty} \leq C. 
$$
\end{Proposition}

\begin{proof}
Let us set 
$$
\rho^\delta(t): =\sup_{\mu_0} e^{-\delta t} \inte \mu(t)v(t),
$$
where the supremum is taken over $\|\mu_0\|_{L^2}\leq 1$ and where $(v,\mu)$ is the solution to \eqref{eq.discApp1} with initial condition $\mu(0)=\mu_0$. In view of the duality identity, the map $\rho^\delta$ is non increasing. Moreover, Lemma \ref{lem:prelimDisc} states that  $\rho^\delta(t)$ is bounded independently of $\delta$ and nonnegative. Then we set 
$$
\rho(t) = \limsup_{\delta \to0} \rho^\delta(t).
$$
Note that $\rho$ is also nonincreasing, nonnegative and bounded. We denote by $\rho_\infty$ its limit as $t\to+\infty$. We claim that $\rho_\infty=0$. 

Indeed, let $t_n\to+\infty$, $\delta_n\to 0$, and $\mu_0^n$ with $\|\mu_0^n\|_{L^2}\leq 1$ be such that
$$
e^{-\delta_n t_n} \inte \mu^n(t_n)v^n(t_n) \geq \rho_\infty-1/n.
$$
We let, for $s\in [-t_n, +\infty)$, 
$$
\tilde v^n(s)= e^{-\delta_n t_n/2} \left(v^n(t_n+s)-\lg v^n(t_n)\rg\right), \qquad \tilde\mu^n(s)= e^{-\delta_n t_n/2} \mu^n(t_n+s).
$$
 From Lemma \ref{lem:prelimDisc} we know that $\tilde v^n$, $D \tilde v^n$ and $\tilde\mu^n$ are locally bounded in $L^2$. As the pair $(\tilde v^n, \tilde \mu^n)$ satisfies an equation of the form \eqref{eq.discApp1}, standard regularity estimates for parabolic equations with bounded coefficients (Theorem III.10.1 of \cite{LSU})  imply that $\tilde v^n$, $D\tilde v^n$ and $\tilde \mu^n$ are locally bounded in $C^{\beta/2,\beta}$ for some $\beta\in (0,1)$. 
Therefore, up to a subsequence, denoted in the same way,  $(\tilde v^n)$ converges to $\tilde v$ and $(\tilde \mu^n)$ converges to $\tilde \mu$ locally uniformly, where by linearity $(\tilde v,\tilde \mu)$ solve 
$$
\left\{\begin{array}{l}
\ds-\partial_t \tilde v -\Delta \tilde v +H_p(x,D\bar u^\delta)\cdot D \tilde v = \frac{\delta F}{\delta m}(x, \bar m^\delta)(\tilde \mu(t))\qquad {\rm in}\; (-\infty,0)\times \T^d\\
\ds \partial_t \tilde\mu-\Delta \tilde\mu -\dive(\tilde\mu H_p(x,D \bar u^\delta))-\dive (\bar m^\delta H_{pp}(x,D\bar u^\delta)D \tilde v)= 0\qquad {\rm in}\; (-\infty,0)\times \T^d.
\end{array}\right.
$$
For any $s\leq 0$ and any $\tau\geq0$, we have, for $n$ large enough,  
$$
\inte \tilde \mu^n(s)\tilde v^n(s)= e^{-\delta_nt_n}\inte \mu^n(t_n+s)v^n(t_n+s)\leq e^{\delta_n s} \rho^{\delta_n}(t_n+s)\leq e^{\delta_n s} \rho^{\delta_n}(\tau), 
$$
so that 
$$
\inte \tilde \mu(s)\tilde v(s)\leq \rho(\tau).
$$
Letting $\tau\to+\infty$, we find therefore 
$$
\inte \tilde \mu(s)\tilde v(s)\leq \rho_\infty= \inte \tilde \mu(0)\tilde v(0)\,,\qquad \forall s\leq 0\,.
$$
However $\inte \tilde \mu(s)\tilde v(s)$ is non increasing, so we also have the reverse inequality, and we deduce that this quantity must be constant in $(-\infty,0]$.  The duality relation then implies that $D\tilde v=0$ for any $t\leq 0$, which gives $\rho_\infty=0$. \\

Next we claim that there exists $\gamma >0$, $C>0$ and $\delta_0>0$ such that, for $\delta\in (0,\delta_0)$, one has 
\be\label{kklaim2}
\rho^\delta(t)\leq C e^{-\gamma t}\qquad \forall t\geq0. 
\ee
Indeed, let $\ep>0$ small to be chosen later and let $T_0>0$, $\delta_0>0$ be such that 
\be\label{kjhaqzertnbs}
\rho^\delta(t)\leq \ep \qquad \forall t\geq T_0, \; \delta \in (0,\delta_0).
\ee
Fix $\delta\in (0,\delta_0)$ and let $(v,\mu)$ be a solution to \eqref{eq.discApp1}. Inequalities \eqref{dualineq} (combined with the fact that $\inte v\mu$ is nonnegative) and \eqref{kjhaqzertnbs} imply that 
$$
\int_{t_1}^{t_2} e^{-\delta s} \|Dv(s)\|_{L^2}^2ds \leq C \ep\|\mu_0\|_{L^2}^2\qquad \forall t_1,t_2\geq T_0, \; \delta \in (0,\delta_0).
$$
 Revisiting the estimate for $\mu$, we have, for any $t_1\geq 0$, 
$$
\|\mu(T_0+t_1)\|_{L^2}\leq Ce^{-\omega t_1} \|\mu(T_0)\|_{L^2} + C\left[ \int_{T_0}^{T_0+t_1} \|Dv(s)\|^2_{L^2}ds \right]^{1/2},
$$
so that, using Lemma \ref{lem:prelimDisc} and the above estimate on $Dv$,  
$$
\begin{array}{rl}
\ds \|\mu(T_0+t_1)\|_{L^2}\; \leq & \ds Ce^{-\omega t_1+\delta T_0/2} \|\mu_0\|_{L^2} + Ce^{\delta (T_0+t_1)/2}\left[ \int_{T_0}^{T_0+t_1} e^{-\delta s} \|Dv(s)\|^2_{L^2}ds \right]^{1/2} \\
\leq & \ds C\|\mu_0\|_{L^2}e^{\delta (T_0+t_1)/2} \left(e^{-(\omega  +\delta/2)t_1}  + \ep^{1/2}\right).
\end{array}
$$
We choose $t_1$ large enough (independently of $\ep$ and $\delta\in (0,\omega)$) so that $Ce^{-\omega t_1} \leq 1/4$ and $\ep$ so small that $C\ep^{1/2}\leq 1/4$. Setting $\tau:=T_0+t_1$, this yields to 
\be\label{lakejznsdl}
\|\mu(\tau)\|_{L^2}\;\leq \; 2^{-1} \|\mu_0\|_{L^2}e^{\delta \tau/2}.
\ee
Fix $(v,\mu)$ a solution to \eqref{eq.discApp1}. The pair $(\tilde v,\tilde \mu):= (v(\tau+\cdot), \mu(\tau+\cdot))$ is also a solution of \eqref{eq.discApp1} with initial condition $\tilde \mu(0)=\mu(\tau)$. The equation being linear in $\mu_0$ and the quantity $\inte \mu(t)v(t)$ being homogenous of degree $2$, we have therefore
$$
e^{-\delta t} \inte \tilde \mu(t)\tilde v(t) \leq \|\mu(\tau)\|_{L^2}^2\rho^\delta(t)\qquad \forall t\geq 0,
$$
where 
$$
e^{-\delta t} \inte \tilde \mu(t)\tilde v(t) = e^{\delta \tau} e^{-\delta(t+\tau)}\inte \mu(t+\tau)v(t+\tau).
$$
This implies that 
$$
e^{-\delta(t+\tau)}\inte \mu(t+\tau)v(t+\tau) \leq e^{-\delta \tau}\|\mu(\tau)\|_{L^2}^2\rho^\delta(t).
$$
Recalling estimate \eqref{lakejznsdl} and taking the supremum over $\|\mu_0\|_{L^2}\leq 1$, we find
$$
\rho^\delta(t+\tau)\leq \rho^\delta(t)/2\qquad \forall t\geq 0.
$$
This easily implies \eqref{kklaim2}. \\

We can now come back to the estimates of $\mu$ and $v$ for a given solution $(v,\mu)$ of \eqref{eq.discApp1} with $\delta \in (0,\delta_0)$. For $t>0$, we have, using  Lemma \ref{lem:prelimDisc}, \eqref{dualineq} and  \eqref{kklaim2} successively:
$$
\begin{array}{rl}
\ds \|\mu(t)\|_{L^2}\; \leq & \ds Ce^{-\omega t/2} \|\mu(t/2)\|_{L^2} + C\left[ \int_{t/2}^{t} \|Dv(s)\|^2_{L^2}ds \right]^{1/2}\\
\leq & \ds Ce^{-\omega t/2+\delta t/2} \|\mu_0\|_{L^2} + Ce^{\delta t/2}\left[ \int_{t/2}^{t} e^{-\delta s} \|Dv(s)\|^2_{L^2}ds \right]^{1/2} \\
\leq & \ds C\|\mu_0\|_{L^2}\left(e^{-\omega t/2+\delta t/2}+e^{\delta t/2-\gamma t/4} \right).
\end{array}
$$
For $\delta$ small enough, this implies that 
$$
\|\mu(t)\|_{L^2}\leq C \|\mu_0\|_{L^2}e^{-\lambda t}\qquad \forall t\geq 0,
$$
for some $\lambda \in (0,\omega)$. Thus,  by Lemma \ref{lem:prelimDisc} applied on the time-interval $[t/2, +\infty)$, 
\begin{align*}
\| Dv(t)\|_{L^2} &\leq C \| \mu (  t/2  )\|_{L^2} e^{\de t/4}
\\
& \leq C \| \mu_0 \|_{L^2} e^{-\lambda t}
\end{align*}
for some possibly different $\lambda >0$. 
%We obtain the estimate for $Dv$ by deriving the equation: for any $i\in \{1,\dots, d\}$, 
%$$
%-\partial_t v_{x_i} +\delta v_{x_i} -\Delta v_{x_i} +H_p\cdot Dv_{x_i}+\left[H_p\right]_{x_i}\cdot Dv = D_{x_i}\frac{\delta F}{\delta m}(x, \bar m^\delta)(\mu(t))
%$$
%Then we use the parabolic regularity to get the same estimate for $\|Dv\|_{L^2}$ (see Lemma \ref{lemCLLP2.2}). 
The bound on $ \left\|v\right\|_\infty$ follows directly from the equation for $v$  and our  regularity assumption  on $\frac{\delta F}{\delta m}$ which implies that 
$$
\left\|\frac{\delta F}{\delta m}(x,m^\delta)(\mu(t))\right\|_{\infty}\leq C\|\mu(t)\|_{L^2}\leq C \|\mu_0\|_{L^2}e^{-\lambda t}\qquad \forall t\geq 0.
$$ 
\end{proof}

In the next step we study a perturbed discounted linearized problem. 

\begin{Proposition}\label{prop:estiSLdisc} Let $(v,\mu)$ solve 
\be\label{eq.discApp}
\left\{\begin{array}{l}
\ds-\partial_t v +\delta v -\Delta v +H_p(x,D\bar u^\delta)\cdot Dv = \frac{\delta F}{\delta m}(x, \bar m^\delta)(\mu(t))+A(t,x) \qquad {\rm in}\; (0,+\infty)\times \T^d\\
\ds \partial_t \mu-\Delta \mu -\dive(\mu H_p(x,D \bar u^\delta))-\dive (\bar m^\delta H_{pp}(x,D\bar u^\delta)Dv)= \dive(B(t,x))\qquad {\rm in}\; (0,+\infty)\times \T^d\\
\mu(0,\cdot)= \mu_0\; {\rm in}\;  \T^d, \qquad v\; {\rm bounded.} 
\end{array}\right.
\ee
with $\inte \mu_0=0$, $\|\mu_0\|_{L^2}\leq 1$ and assume that, for some $\gamma>0$, 
\be\label{HypABdisc}
\|A(t)\|_{L^2}+ \|B(t)\|_{L^2} \leq e^{-\gamma t} \qquad \forall t\geq 0.
\ee
If $\delta\in (0,\delta_0)$, then
\be\label{estlinde}
\|\mu(t)\|_{L^2}+ \|Dv(t)\|_{L^2} \leq C(1+t)e^{-\theta t}
\ee
where $\theta:= \gamma \wedge \lambda$ and $\delta_0, \lambda>0$ are defined in Proposition \ref{prop:discountHomo}. 
\end{Proposition}

\begin{proof} Using Proposition \ref{prop:discountHomo} and the linearity of the equation, we can assume, without loss of generality, that $\mu_0=0$. We first assume that $A\equiv 0$. Throughout the proof, the constant $C$ can depend on $\gamma$. 

Let us start with preliminary estimates. The duality identity here implies:
$$
C^{-1}\int_{t_1}^{t_2} e^{-\delta s}\|Dv(s)\|_{L^2}^2 ds \leq - \left[e^{-\delta s} \inte v(s)\mu(s)\right]_{t_1}^{t_2} + C \int_{t_1}^{t_2} e^{-\delta s} \|B(s)\|_{L^2}^2 ds.
$$
One can check, exactly as for the proof of Lemma \ref{lem:prelimDisc}, that 
$$
\lim_{t\to+\infty} e^{-\delta t} \inte \mu(t)v(t)= 0.
$$
Then the duality inequality and our assumption on $B$ imply that
$$
\int_{0}^{+\infty} e^{-\delta s}\|Dv(s)\|_{L^2}^2 ds \leq C.
$$
Arguing as before we derive for $\mu$ that 
$$
\begin{array}{rl}
\ds \|\mu(t)\|_{L^2}\; \leq & \ds C\left[ \int_{0}^{t} \|Dv(s)\|^2_{L^2}+ \|B(s)\|_{L^2}^2 \ ds \right]^{1/2}\\
\leq & \ds Ce^{\delta t/2}\left[ \int_{0}^{t} e^{-\delta s} (\|Dv(s)\|^2_{L^2}+ \|B(s)\|_{L^2}^2 )ds \right]^{1/2} \; \leq \; Ce^{\delta t/2}.
\end{array}
$$
Thus, applying Lemma  \ref{lemCLLP2.2} (with $T\to \infty$) to $e^{-\delta t } v$, we deduce
$$
\begin{array}{rl}
\ds e^{-\delta t }\|v(t)-\lg v(t)\rg\|_{L^2} \; \leq & \ds C\int_t^{+\infty} e^{-\omega(s-t)} \|\mu(s)\|_{L^2} e^{-\delta s}ds\\
\leq & \ds C  \, e^{-\delta t/2},
\end{array}
$$
which gives
$$
\|v(t)-\lg v(t)\rg\|_{L^2} \; \leq C \,  e^{\delta t/2}\,.
$$
We set 
$$
\rho^\delta (t)=\sup_{B} \left[e^{-\delta t}\left(\|\mu(t)\|_{L^2}+ \|v(t)-\lg v(t)\rg\|_{L^2}\right)\right],
$$
where the supremum is taken over the $B$ that satisfy \eqref{HypABdisc} and where $(v,\mu)$ solves \eqref{eq.discApp} (with $A\equiv 0$ and $\mu_0=0$). 
Fix  $(v,\mu)$ solution to \eqref{eq.discApp} with $A\equiv 0$ and $\mu_0=0$ and let  us consider its restriction to a time interval $[\tau,+\infty)$. We can write 
$$
(v,\mu)= (v_1,\mu_1)+(v_2,\mu_2)
$$
where $(v_1,\mu_1)$ solves on $[\tau,+\infty)$ the homogeneous equation \eqref{eq.discApp1} with initial condition $\mu_1(\tau)=\mu(\tau)$ and $(v_2,\mu_2)$ solves on $[\tau, +\infty)$ the inhomogeneous equation \eqref{eq.discApp} with $\mu_2(\tau)=0$ and $A\equiv 0$. By Proposition \ref{prop:discountHomo} we have, for $\delta\in (0,\delta_0)$,  
$$
\|\mu_1(\tau + t)\|_{L^2}+\|Dv_1(\tau + t)\|_{L^2}\leq 
C_0e^{-\lambda t}\, \|\mu(\tau)\|_{L^2}\leq C_0e^{-\lambda t}\,   e^{\delta \tau/2} \qquad \forall t\geq 0, 
$$
while, as the restriction of $B$ to $[\tau,+\infty)$ satisfies 
$$
\|B(\tau+t)\|_{L^2}\leq e^{-\gamma \tau} e^{-\gamma t}\qquad \forall t\geq 0, 
$$
we have
$$
\|\mu_2(\tau+t)\|_{L^2}+ \|v_2(\tau+t)-\lg v_2(\tau+t)\rg\|_{L^2}\leq e^{-\gamma  \tau} \rho^\delta(t) \,   e^{\delta t}\qquad \forall t\geq0. 
$$
So 
$$
\|\mu(\tau+t)\|_{L^2}+ \|v(\tau+t)-\lg v(\tau+t)\rg\|_{L^2} \leq   Ce^{-\lambda t}  e^{\delta \tau/2}
+ e^{-\gamma  \tau} \rho^\delta(t)  e^{\delta t}.
$$
  Multiplying by $ e^{-\delta (t+\tau)} $ and taking the supremum over $B$ leads to 
$$
\rho^\delta(\tau+t)\leq   C \,e^{-(\lambda +\delta)t}
+ e^{-(\gamma+\delta)  \tau} \rho^\delta(t).
$$
Setting $\theta:= \gamma \wedge \lambda $ and considering the inequality satisfied by $e^{(\theta+\delta) t}\rho^\delta(t)$, we then obtain the exponential decay of $\rho^\delta$:
$$
\rho^\delta(t) \leq C (1+t)e^{-(\theta+\delta) t},
$$
which implies, by definition of $\rho^\delta(t)$, that
$$
\sup_B \left(\|\mu(t)\|_{L^2}+ \|v(t)-\lg v(t)\rg\|_{L^2}\right) \leq C (1+t)e^{-\theta t}.
$$
Once more we observe that, by Lemma \ref{lemCLLP2.2}, we can estimate $\|Dv(t)\|_{L^2}$ in terms of $ \|\mu(t)\|_{L^2}$ and $\|v(t)-\lg v(t)\rg\|_{L^2}$. Hence \eqref{estlinde} is proved when $A=0$. 

It remains to consider the case where $A\not \equiv 0$. Let $v_1$ be the unique bounded solution to 
$$
-\partial_t v_1 +\delta v_1 -\Delta v_1 +H_p(x,D\bar u^\delta)\cdot Dv_1 = A(t,x) \qquad {\rm in}\; (0,+\infty)\times \T^d.
$$
Using as before  Lemma \ref{lemCLLP2.2} for $e^{-\delta  t }v_1$ and with $T\to \infty$, we estimate
$$
\|v_1(t)-\lg v_1(t)\rg\|_{L^2} \leq C\int_t ^\infty e^{-(\omega+\delta)(s-t)} \|A(s)\|_{L^2}ds \leq C e^{-\gamma t}. 
$$
Finally, using again Lemma \ref{lemCLLP2.2} gives 
$$
\|Dv_1(t)\|_{L^2}\leq C e^{-\gamma t}.
$$
Note that, if $(v,\mu)$ is the solution to \eqref{eq.discApp}, then $(v-v_1,\mu)$ solves \eqref{eq.discApp} with $A\equiv 0$ and $B'= B+\bar m^\delta H_{pp}Dv_1$, so that, applying the above estimate gives 
$$
\|\mu(t)\|_{L^2}+ \|Dv(t)\|_{L^2}\leq C(1+t)e^{-\theta t}
$$
where $\theta:= \gamma \wedge \lambda$. 
\end{proof}

%%%%%%%%%%%%%%%%
\subsection{Exponential rate for the nonlinear system}

We now consider the infinite horizon discounted nonlinear MFG system \eqref{e.MFGih}.  Let us recall that this system is well-posed and that we have Lipschitz estimates: 
\begin{Lemma}\label{lem.udeltamdelta} Under our standing assumptions, for any $\delta\in (0,1)$ there exists a unique solution $(u^\delta,m^\delta)$ to \eqref{e.MFGih}. Moreover, for any $\alpha\in (0,1)$, there exists a constant $C>0$, independent of $\delta$, such that 
$$
\|Du^\de\|_{C^{(1+\alpha)/2,1+\alpha}}+ \sup_{t\in [1,\infty)} \|m^\de(t)\|_\infty\leq C.
$$
\end{Lemma}

\begin{proof} Existence and uniqueness of the solution rely on standard arguments (see  \cite{LLperso}).  Since the solution is unique, it can be obtained as limit of solutions in horizons $T_n\to \infty$ with the terminal condition $u(T_n)=0$; this way one can prove, exactly as in Lemma \ref{lem.SemiConc}, that $Du^\delta$ is uniformly bounded, and one also has  a uniform bound for $\|\de u^\de\|_\infty$.  As a consequence, $m^\delta$ is uniformly bounded in $[1,+\infty)$ thanks to Lemma \ref{lemCLLP2.1} and is (uniformly) H\"older continuous in time with values in $\Pk$, see estimate \rife{holder-m}. Finally,   by  considering the equation of $(u^\de)_{x_i}$, namely
$$
-\partial_t (u^\de)_{x_i} + \de (u^\de)_{x_i} - \Delta (u^\de)_{x_i} + H_{x_i} + H_p \cdot D(u^\de)_{x_i}= F_{x_i}\,,
$$
the parabolic regularity applied in any interval $(t,t+1)$, jointly with the uniform bound already found for $\|(u^\de)_{x_i}\|_\infty$,   implies the desired estimate upon $Du^\delta$.
More precisely,   by only using that $F_x(x,m)$ is uniformly bounded, and  the bound on $H_x$ and $H_p$,  we deduce a bound for $(u^\de)_{x_i}$ in $C^{(1+\alpha)/2,1+\alpha}$ for any $\alpha\in (0,1)$.
\end{proof}

The main result of this part is the following exponential convergence of the discounted problem. 
\begin{Theorem}\label{them.Cvexpo.discounted} Let  $(u^\delta, m^\delta)$ be the solution to the discounted MFG system \eqref{e.MFGih}. There exist $\gamma,\delta_0>0$ and $C>0$ such that, if $\delta \in (0,\delta_0)$, then  
$$
\|D(u^\delta(t)-\bar u^\delta)\|_{L^\infty}\leq Ce^{-\gamma t}\qquad \forall t\geq 0
$$
and
$$
\|m^\delta(t)-\bar m^\delta\|_{L^\infty} \leq  Ce^{-\gamma t}\qquad \forall t\geq 1.
$$
\end{Theorem}

\begin{proof} The proof is very close to the proof of Theorem \ref{thm:CvExpMFG}. 
Let  
$$
E:=\{(v,\mu),\; \|Dv(t)\|_{L^\infty}+ \|\mu(t)\|_{L^\infty}\leq \hat Ke^{-\gamma t}\}
$$
where $\hat K>0$ and $\gamma>0$ are to be chosen below. We assume that $\hat K$ is small enough so that
$$
\bar m^\delta >\hat K .
$$
We also assume that the initial condition is close to $\bar m^\delta$, namely $\mu_0:=m_0-\bar m^\delta$ satisfies
$$
\|\mu_0\|_{L^\infty}\leq \hat K^2. 
$$
We consider the solution $(\tilde v,\tilde \mu)$ to \eqref{eq.discApp} with initial condition $\tilde \mu(0)=\mu_0$, 
$$
A(t,x)= -H(x,D(\bar u^\delta+ v))+ H(x,D\bar u^\delta)+ H_p(x, D\bar u^\delta)\cdot Dv +F(x, \bar m^\delta +\mu)-F(x, \bar m^\delta)-\frac{\delta F}{\delta m}(x, \bar m^\delta)(\mu)
$$
and 
$$
B(t,x)= 
(\bar m^\delta +\mu)H_p(x, D(\bar u^\delta+ v))-\bar m^\delta H_p(x, D\bar u^\delta)
- \mu H_p(x,D\bar u^\delta)-\bar m^\delta H_{pp}(x,D\bar u^\delta)Dv.
$$
We note  that 
$$
\|A(t)\|_{L^\infty}+\|B(t)\|_{L^\infty} \leq C\hat K^2 e^{-2\gamma t}.
$$
From Proposition \ref{prop:estiSLdisc} we have 
$$
\|\tilde \mu(t)\|_{L^2}+ \|D\tilde v(t)\|_{L^2} \leq C\hat K^2(1+t)e^{-\theta t}
$$
where $\theta:= 2\gamma \wedge \lambda$. Using  the smoothing properties of $\frac{\delta F}{\delta m}$ and the parabolic regularity of the equation satisfied by  $\tilde v- \lg  \tilde v\rg$, exactly as in Theorem \ref{thm:CvExpMFG} we can upgrade the above estimate into: 
$$
\|\tilde \mu(t)\|_{\infty}+  \|D\tilde v(t)\|_{\infty}\leq C\hat K^2(1+t)e^{-\theta t}. 
$$
So if one chooses $\gamma\in (0, \lambda)$, we infer  that
$$
\|\tilde \mu(t)\|_{L^\infty}+ \|D\tilde v(t)\|_{L^\infty} \leq C\hat K^2 e^{-\gamma t}.
$$
For $\hat K$ small enough, this implies that $(\tilde v, \tilde \mu)$ belongs to $E$. Note that $\tilde v$, $D\tilde v$  and $\tilde \mu$ are bounded in $C^{\alpha/2,\alpha}$ because they solve parabolic equations with bounded coefficients. So the map $(v,\mu)\to (\tilde v,\tilde \mu)$ is compact (say in $W^{1,\infty}\times L^\infty$) and thus has a fixed point $(v^\delta,\mu^\delta)$. Then $(u^\delta,m^\delta):= (\bar u^\delta,\bar m^\delta)+ (v^\delta,\mu^\delta)$ is a solution to \eqref{e.MFGih} which satisfies the decay
$$
\|m^\delta(t)-\bar m^\delta\|_{\infty}+ \|D(u^\delta(t)-\bar u^\delta)\|_{\infty}\leq Ce^{-\gamma t}\qquad \forall t\geq 0.
$$

It remains to remove the   assumption on the initial condition $m_0$. For this we only need to show that there exists a time $T>0$ such that, for any $m_0\in \Pk$, the solution $(u^\delta,m^\delta)$ of \eqref{e.MFGih} satisfies $\|m^\delta(T)-\bar m^\delta\|_\infty\leq \hat K^2$. Indeed, we can then apply the previous result to the restriction of $(u^\delta, m^\delta)$ to the time interval $[T,+\infty)$. 

By the duality relation, we have 
\be\label{more}
C^{-1}\int_{t_1}^{t_2} e^{-\delta t} \|D(u^\delta(t)-\bar u^\delta)\|_{L^2}^2dt \leq -\left[ e^{-\delta t} \inte (u^\delta(t)-\bar u^\delta)(m^\delta(t)-\bar m^\delta)\right]_{t_1}^{t_2} .
\ee
Thus
\be\label{kehjrzbretdf}
C^{-1}\int_0^{+\infty}e^{-\delta t} \|D(u^\delta(t)-\bar u^\delta)\|_{L^2}^2 dt \leq \inte (u^\delta(0)-\bar u^\delta)(m_0-\bar m^\delta) \leq C
\ee
because $u^\delta$ is uniformly Lipschitz continuous in space (see Lemma \ref{lem.udeltamdelta}). As $\mu^\delta:=m^\delta-\bar m^\delta$ satisfies
$$
\partial_t \mu^\delta-\Delta \mu^\delta -\dive(\mu^\delta H_p(x,D u^\delta))= \dive (\bar m^\delta (H_p(x, D \bar  u^\delta)-H_p(x,Du^\delta))), 
$$
and still using the fact that $Du^\delta$ is bounded, Lemma \ref{lemCLLP2.1} implies that, for any $t\geq t_1\geq 1$,  
$$
\|m^\delta(t)-\bar m^\delta\|_{L^2}\leq Ce^{-\omega (t-t_1)} \|m^\delta(t_1)-\bar m^\delta\|_{L^2}+ 
C e^{\delta t/2}\left[\int_{t_1}^t  e^{-\delta s} \|D(u^\delta(s)-\bar u^\delta)\|_{L^2}^2ds \right]^{\frac12}.
$$
Choosing $t_1=1$ (say) and recalling that $m^\delta$ is bounded in $L^\infty$ (Lemma \ref{lemCLLP2.1}), we find
$$
\|m^\delta(t)-\bar m^\delta\|_{L^2}\leq C\, e^{\delta t/2} \qquad \forall t\geq1. 
$$
Let $T\geq 2$ to be chosen below. Coming back to \eqref{kehjrzbretdf}, there exists $t_1\in [1,T]$ and $t_2\in [3T+1,4T]$ such that 
$$
e^{-\delta t_i} \|D(u^\delta(t_i)-\bar u^\delta)\|_{L^2}^2 \leq C/T. 
$$
Then from \rife{more} we deduce
$$
\begin{array}{rl}
\ds  C^{-1}\int_{t_1}^{t_2}  e^{-\delta t} \|D(u^\delta(t)-\bar u^\delta)\|_{L^2}^2dt \; \leq  &
\ds  e^{-\delta t_1}  \|D(u^\delta (t_1)-\bar u^\delta)\|_{L^2}\|m^\delta(t_1)-\bar m^\delta\|_{L^2}\\
& \ds +  e^{-\delta t_2}\|D(u^\delta (t_2)-\bar u^\delta)\|_{L^2}\|m^\delta(t_2)-\bar m^\delta\|_{L^2}
%\\
 \leq 
%& \ds 
C\, T^{-1/2}.
 \end{array}
$$
Then, as $t_1\leq T\leq 3T+1\leq t_2\leq 4T$, we have, for any $t\in [2T, t_2]$,  
\be\label{l2small}
\begin{array}{rl}
\ds \|m^\delta(t)-\bar m^\delta\|_{L^2}\; \leq & \ds  C\,e^{-\omega (2T-t_1)} \|m^\delta(t_1)-\bar m^\delta\|_{L^2}+ 
Ce^{\delta t_2/2} \left[\int_{t_1}^{t_2} e^{-\delta t}  \|D(u^\delta(t)-\bar u^\delta)\|_{L^2}^2dt \right]^{1/2} \\
\leq & \ds C\, e^{-\omega T} e^{\de T/2}+ 
Ce^{2 \delta T} T^{-1/4}.
 \end{array}
\ee
Notice that, by choosing $T$ large, and then $\delta$ small, the above inequality implies that $m^\delta(t)-\bar m^\delta$ is sufficiently small for any $t\in [2T, 3T]$. In order to conclude, we only need to upgrade this estimate to the $L^\infty$-norm.

To this purpose, recall   that $w^\delta:= u^\delta -\bar u^\delta$ solves the equation
$$
 -\partial_t w^\delta+ \de w^\de -\Delta w^\delta + V^\delta \cdot Dw^\delta = F(x,m^\delta(t))-F(x,\bar m^\delta ) 
$$
where $V^\delta= \int_0^1 H_p(x,D \bar u^\delta + s\, D(u^\delta -\bar u^\delta))ds$ is uniformly bounded. Since we have, by Poincar\'e's inequality, 
$$
e^{-\delta t_2} \|  w^\delta(t_2)-\lg w^\delta(t_2)\rg \|_{L^2}^2\leq C e^{-\delta t_2} \|Dw^\delta(t_2)\|_{L^2}^2 \leq C/T,
$$
applying Lemma \ref{lemCLLP2.2}  to $e^{-\delta t}w^\de$ we deduce that, for $t\in [2T, 2T+2]$, 
$$
\begin{array}{l}
\| w^\delta(t)-\lg w^\delta(t)\rg \|_{L^2} \\
\qquad \ds \leq C e^{-\omega(t_2-t)} \|  w^\delta(t_2)-\lg w^\delta(t_2)\rg \|_{L^2}e^{\delta (t-t_2)}  + C\int_t^{t_2} e^{-\omega(s-t)} \|m^\delta(s)-\bar m^\delta\|_{L^2} e^{\de (t-s)}ds
\\ 
\qquad \ds \leq  
 C e^{-\omega(t_2-t)} \, \frac{e^{\delta (t-t_2/2)} }{T^{1/2}}  +
 C\, (e^{-\omega T} e^{\de T/2}+ e^{2 \delta T} T^{-1/4})
 \int_t^{t_2} e^{-\omega(s-t)}e^{\de (t-s)}ds
%e^{-\omega T} + Ce^{2 \delta T } T^{-1/4} +  C \int_{3T}^{t_2} e^{-\omega(s-t)} e^{\delta s/2}ds
\end{array}
$$
where we also used \rife{l2small}. Recalling that $t\in [2T, 2T+2]$ and $t_2\in [3T+1, 4T]$, we have $t-t_2/2 \geq 0 $, so if  $\delta$ is small enough compared to $\omega$ we conclude that
$$
\| w^\delta(t)-\lg w^\delta(t)\rg \|_{L^2} \leq C  \left(  e^{-\omega T/2} + e^{2 \delta T } T^{-1/4}\right).
$$
We apply once more Lemma \ref{lemCLLP2.2} to estimate $Dw^\delta(t)$ in $(2T, 2T+1)$: we deduce that
$$
\|D(u^\delta(t)-\bar u^\delta)\|_{L^2} \leq C  \left( e^{-\omega T/2} + e^{2 \delta T } T^{-1/4}\right)
%C e^{-  (\omega-2\delta) T }  T^{-1/2}  + C\, e^{-\omega T+ \delta 3T/2} + Ce^{2 \delta T } T^{-1/4}  
$$
for every $t\in (2T, 2T+1)$.  In fact, since $D(u^\delta(t)-\bar u^\delta)$ is bounded, a similar estimate actually holds in $L^p$ for all $p<\infty$: 
$$
\|D(u^\delta(t)-\bar u^\delta)\|_{L^p} \leq C  \left( e^{-\omega T/p} + e^{4 \delta T/p } T^{-1/(2p)}\right)
%C e^{-  (\omega-2\delta) T }  T^{-1/2}  + C\, e^{-\omega T+ \delta 3T/2} + Ce^{2 \delta T } T^{-1/4}  
$$

 Recalling the estimate \eqref{l2small}, by parabolic regularity used for the equation of $\mu^\delta$ in the interval $(2T, 2T+1)$, we conclude that the $L^\infty$-norm of $\mu^\delta$ satisfies a  similar estimate for, say, $t\in (2T+1/2, 2T+1)$. In particular, we can fix $T$ large and   $\delta_0>0$ small such that in this interval we have $\|m^\delta(t)-\bar m^\delta\|_{L^\infty}\leq \hat K^2$ for any $\delta\in (0,\delta_0)$.
\end{proof}

Let us underline the following consequence of our estimates on the solution to the linearized system 
\be\label{e.MFGihLSTER}
\left\{\begin{array}{l}
\ds-\partial_t v +\delta v -\Delta v +H_p(x,Du^\delta).Dv = \frac{\delta F}{\delta m}(x, m^\delta(t))(\mu(t))\qquad {\rm in}\; (0,+\infty)\times \T^d\\
\ds \partial_t \mu-\Delta \mu -\dive(\mu H_p(x,D u^\delta))-\dive (m^\delta H_{pp}(x,Du^\delta)Dv)= 0\qquad {\rm in}\; (0,+\infty)\times \T^d\\
\mu(0,\cdot)= \mu_0\; {\rm in}\;  \T^d, \qquad v\; {\rm bounded.} 
\end{array}\right.
\ee
Notice that the system   has been now linearized around   the pair $(u^\delta, m^\delta)$ which solves the discounted MFG system \eqref{e.MFGih}. 

\begin{Corollary}\label{coro.bounds} There exist $\theta,\delta_0>0$ and a constant $C>0$ such that, if  $\delta\in (0,\delta_0)$, then  
the solution $(v,\mu)$ to \eqref{e.MFGihLSTER} with $\inte \mu_0=0$ satisfies 
$$
\|Dv(t)\|_{L^2} \leq C e^{-\theta t}\|\mu_0\|_{L^2}\qquad \forall t\geq 0 
$$
and
$$
\|\mu(t)\|_{L^2} \leq C e^{-\theta t}\|\mu_0\|_{L^2}\qquad \forall t\geq 1. 
$$
In addition, for any $\alpha\in (0,1)$, there is a constant $C$ (independent of $\delta\in (0,\delta_0)$) such that 
$$
\sup_{t\geq 0}\|v(t)\|_{C^{2+\alpha}}\leq C\|\mu_0\|_{(C^{2+\alpha})'}.
$$
\end{Corollary}

\begin{proof} As in the proof of Lemma \ref{lem:prelimDisc}, we have a preliminary estimate:
$$
\|\mu(t)\|_{L^2}+ \|Dv(t) \|_{L^2}\leq C_0\|\mu_0\|_{L^2}e^{\delta t/2}.
$$
We rewrite system \eqref{e.MFGihLSTER} in the form \eqref{eq.discApp} with 
$$
A(t,x):=-\left(H_p(x,D u^\delta)- H_p(x,D \bar u^\delta)\right)\cdot Dv + \frac{\delta F}{\delta m}(x, m^\delta(t))(\mu(t))-\frac{\delta F}{\delta m}(x, \bar m^\delta)(\mu(t))
$$
and 
$$
B(t,x):= -\mu (H_p(x,D u^\delta)-H_p(x,D \bar u^\delta))-(m^\delta H_{pp}(x,Du^\delta)-\bar m^\delta H_{pp}(x,D\bar u^\delta))Dv.
$$
From Theorem \ref{them.Cvexpo.discounted}, we have, for $\delta$ small enough,  
$$
\begin{array}{rl}
\ds \|A(t)\|_{L^2} \;  \leq & \ds Ce^{-\gamma t}\left(\|Dv\|_{L^2} + \|\mu(t)\|_{L^2}\right)\\ 
\leq & \ds Ce^{-(\gamma-\delta) t}\|\mu_0\|_{L^2}\;  \leq \; Ce^{-\gamma t/2}\|\mu_0\|_{L^2}.
\end{array}
$$
In the same way, 
$$
\|B(t)\|_{L^2} \;  \leq \; Ce^{-\gamma t/2}\|\mu_0\|_{L^2}.
$$
Then Proposition \ref{prop:estiSLdisc} implies that 
$$
\|\mu(t)\|_{L^2}+ \|Dv(t)\|_{L^2} \leq C(1+t)e^{-\gamma t/2}\|\mu_0\|_{L^2}.
$$
The above estimates combined with the maximum principle imply that $v$ is bounded in $L^\infty$ by 
$$
\sup_{t\in [0,T]} \|v(t)\|_\infty\leq C \|\mu_0\|_{L^2}.
$$ 
 In order  to change the left-hand side $\|v(t)\|_{\infty}$ into $\|v(t)\|_{C^{2+\alpha}}$ and the
right-hand side $\|\mu_0\|_{L^2}$ into $\|\mu_0\|_{(C^{2+\alpha})'}$, one can proceed  as in Corollary \ref{coro.unifBoundT}. 
\end{proof}

%%%%%%%%%%%%%%%%%%%%%%%%%%%%
%%%%%%%%%%%%%%%%%%%%%%%%%%%
\section{The master cell problem}\label{sec:mastercell}

In this section we study the master cell problem: 
\be\label{CellMAster}
\begin{array}{l}
\ds   \lambda -\Delta_x \chi(x,m) +H(x,D_x \chi(x,m)) -\inte \dive_y(D_m \chi(x,m,y))dm(y)\\
\ds \qquad \qquad  +\inte D_m \chi(x,m,y). H_p(y,D_x\chi(y,m))dm(y) = F(x,m) \qquad {\rm in }\; \T^d\times \Pk.
\end{array}
\ee
We prove that this equation is well defined in a suitable sense: there is a unique constant $\bar \lambda$ for which the master cell problem has a     ``weak" solution in $\T^d\times \Pk$. Moreover we prove that $\bar \lambda$ is also the unique constant for which the
ergodic mean field game system \eqref{e.MFGergo} has a solution $(\bar \lambda, \bar u, \bar m)$. 

Let us stress that  a weak solution of \rife{CellMAster}, according to our next definition, is not necessarily $C^1$ with respect to $m$, so that \rife{CellMAster} is not formulated classically. Instead, the equation is interpreted as it is often done with transport equations, by requiring somehow that  the value of the solution is obtained  through the characteristic curves. By considering weak solutions, we avoid some lengthy and involved estimates which are needed to achieve the $C^1$ character   with respect to $m$. The reader is referred to \cite{CDLL} for this issue. For our purposes, the context of weak solutions is enough to characterize the ergodic limit.

\begin{Definition}\label{defweakmaster}  We say that the pair ($ \lambda, \chi$), with $\lambda\in \R$ and $\chi:\T^d\times \Pk\to \R$ a map, is a weak solution to the master cell problem \eqref{CellMAster} if $\chi$ and $D_x\chi$ are globally Lipschitz continuous in $\T^d\times \Pk$  and if $\chi$ satisfies the following two conditions:

(i) $\chi$ is monotone: 
$$
\inte (\chi(x,m)-\chi(x,m'))d(m-m')(x) \geq 0\qquad \forall m,m'\in \Pk,
$$
 
(ii) for any $m_0\in \Pk$, and any $T>0$, whenever we consider   the unique solution $(u,m)$  to 
\be\label{MFGmasterCell}
\left\{\begin{array}{l}
-\partial_t u + \lambda -\Delta u +H(x,D   u) = F(x, m)\qquad {\rm in}\; (0,T)\times \T^d\\
\partial_tm -\Delta  m -\dive( m H_p(x,Du))= 0\qquad {\rm in}\; (0,T)\times \T^d\\
m(0,\cdot)=m_0, \; u(T,\cdot)= \chi(x, m(T))\qquad {\rm in}\;  \T^d
\end{array}\right.
\ee
then we have $\chi(x,m_0)= u(0,x)$ for any $x\in \T^d$. 
\end{Definition}

Let us make some comments about the above definition. Firstly, the monotonicity condition on $\chi$ ensures the uniqueness of the solution  $(u,m)$ to \eqref{MFGmasterCell}. Secondly, if $\chi=\chi(x,m)$ is a weak solution, then $\chi$ is actually $C^2$ in the space variable $x$ because so is the solution $u$ of \eqref{MFGmasterCell} at time $t=0$. Thirdly, condition (ii) implies that in \eqref{MFGmasterCell}  one  actually  has $\chi(x,m(t))= u(t,x)$ for any $(t,x)\in [0,T]\times \T^d$, so that $m$ solves the McKean-Vlasov equation
\be\label{e.McKV}
\partial_tm -\Delta  m -\dive( m H_p(x,D\chi(x,m(t))))= 0, \qquad m(0,\cdot)=m_0.
\ee
The Lipschitz continuity of $D_x\chi$ ensures that this equation has a unique solution.

\begin{Theorem}\label{thm.MasterCell} There is a unique constant $\bar \lambda\in \R$ for which the master cell problem  \eqref{CellMAster} has a weak solution. The constant $\bar \lambda$ is also the unique constant for which the ergodic MFG problem \eqref{e.MFGergo} has a solution. Besides, if $\chi$ is a solution to \eqref{CellMAster}, then $\chi(\cdot,m)$ is of class $C^2$ for any $m\in \Pk$ and 
$$
D_x\chi(x,\bar m)= D\bar u(x) \qquad \forall x\in \T^d,
$$
where  $(\bar u,\bar m)$ is a solution to \eqref{e.MFGergo}
\end{Theorem}

The proof requires several steps. As usual, we build the solution through the discounted problem, for which we have to show uniform regularity estimates (independent of the discount factor). 

%%%%%%%%%%%%%%%%%%%%%%%%%
\subsection{Estimates for the discounted master equation}

In order to build a solution to the cell problem, we consider, for $\delta>0$, the discounted master equation  
\eqref{e.MasterDiscont}.
Let us recall (see \cite{CDLL}) that $U^\delta$ can be built as follows: for any $m_0\in \Pk$,  let $(u^\delta, m^\delta)$ be the solution to \eqref{e.MFGih}.
Then 
\be\label{rep.Udelta}
U^\delta(x,m_0)= u^\de(0, x). 
\ee

The next Lemma collects standard estimates on $U^\delta$. 
\begin{Lemma}\label{lem:StandEstiUdelta} Let $U^\delta$ be the solution to \eqref{e.MasterDiscont} and  $(u^\delta, m^\delta)$ be a solution to \eqref{e.MFGih}. Then, for any $\alpha\in (0,1)$, there is a constant $C$, independent of $m_0$ and $\delta$, such that 
$$
 \left\|\delta U^\delta( \cdot ,m)\right\|_\infty + \left\|D_xU^\delta(\cdot,m)\right\|_{C^{1+\alpha}}\leq C\qquad \forall  m\in    \Pk.
$$
\end{Lemma}

\begin{proof} As $u^\delta$ is a bounded solution to the first equation in \eqref{e.MFGih}, it is well-known  that
$$
\sup_{(t,x)\in [0,+\infty)\times \T^d} \left| \delta u^\delta(t,x)\right|\leq \sup_{x\in \T^d} |H(x,0)|+ \sup_{(x,m)\in \T^d\times \Pk} |F(x,m)|.
$$
This yields the uniform estimate on $\|\delta U^\delta\|_\infty$. From Lemma \ref{lem.udeltamdelta}, we know that $Du^\delta$ is bounded in $C^{(1+\alpha)/2, 1+\alpha}$ for any $\alpha\in (0,1)$: this implies the same bound on $D_xU^\delta$. 
%Finally $m^\delta$ is uniformly bounded on $[1,+\infty)$ thanks to Lemma \ref{lem.boundm}. 
\end{proof}

The next result states that $U^\delta$ is uniformly Lipschitz continuous with respect to $m$. 

\begin{Proposition}\label{prop:UdeltaLip} Let $U^\delta$ be the solution to \eqref{e.MasterDiscont}. Then, for any $\alpha\in (0,1)$,  there exists a constant $C$, depending on $\alpha$ and on the data only, such that
\be\label{regDm}
\left\| D_mU^\delta(\cdot, m, \cdot)\right\|_{2+\alpha,1+\alpha}  \leq C.
\ee
In particular, $U^\delta(\cdot,\cdot)$ and $D_x U^\delta(\cdot,\cdot)$ are  uniformly Lipschitz continuous. 
\end{Proposition}

\begin{proof} Let us fix $m_0\in \Pk$, $(u^\delta,m^\delta)$ be the solution to \eqref{e.MFGih}. We use the following representation formula (see \cite{CDLL}): for any smooth map $\mu_0$, we have 
\be\label{rep.deltaUdelta}
\inte \frac{\delta U^\delta}{\delta m}(x,m_0,y)\mu_0(y)dy = v(0,x),
\ee
where $(v,\mu)$ is the unique solution to the linearized system 
\be\label{e.MFGihLS}
\left\{\begin{array}{l}
\ds-\partial_t v +\delta v -\Delta v +H_p(x,Du^\delta).Dv = \frac{\delta F}{\delta m}(x, m^\delta(t))(\mu(t))\qquad {\rm in}\; (0,+\infty)\times \T^d,\\
\ds \partial_t \mu-\Delta \mu -\dive(\mu H_p(x,D u^\delta))-\dive (m^\delta H_{pp}(x,Du^\delta)Dv)= 0\qquad {\rm in}\; (0,+\infty)\times \T^d,\\
\mu(0,\cdot)= \mu_0\; {\rm in}\;  \T^d, \qquad v\; {\rm bounded.} 
\end{array}\right.
\ee
If we suppose that $\inte \mu_0=0$, Corollary \ref{coro.bounds}  states that 
$$
\sup_{t\geq 0}\|v(t)\|_{C^{2+\alpha}}\leq C\|\mu_0\|_{(C^{2+\alpha})'}
$$
for any $\alpha>0$. By \eqref{rep.deltaUdelta} and 
$$
D_y  \frac{\delta U^\delta}{\delta m}(x,m_0,y) = D_m U^\delta(x,m_0,y), 
$$
we infer exactly as in \cite{CDLL} that 
$$
\left\|D_m U^\delta(\cdot,m_0,\cdot)\right\|_{2+\alpha,1+\alpha} \leq C.
$$
\end{proof}

\begin{Remark}  We stress that the uniform Lipschitz continuity of  $U^\delta(\cdot,\cdot)$ and $D_x U^\delta(\cdot,\cdot)$ would require milder assumptions than those needed to prove \rife{regDm}.  Indeed, by only using condition (FG2) on the couplings, we can replace the conclusion of Corollary \ref{coro.bounds} with the estimate
$$
\sup_{t\geq 0}\|v(t)\|_{C^1}\leq C\|\mu_0\|_{(C^{1})'}
$$
which would follow as explained in Remark \ref{milder}. With this latter estimate in hand, 
using \rife{rep.deltaUdelta} with  $\mu_0= D_y \psi(y)$ (for $\psi$ smooth), it follows
$$
 \inte D_y D_x \frac{\delta U^\delta}{\delta m}(x,m_0,y)\psi(y)dy  \leq C\, \|\mu_0\|_{(C^{1})'} \leq C \, \|\psi\|_{L^1}
$$
which yields
$$
\| D_m D_x U^\delta(x, m_0)\|_\infty \leq C\,.
$$
Since $ D_{xx}^2 U^\delta(x, m)$ is estimated from Lemma \ref{lem:StandEstiUdelta}, this would imply the Lipschitz uniform bound for $D_x U^\delta(\cdot,\cdot)$.

In the following, we actually only use this information in  order to prove the existence of  a weak solution to the master equation and the convergence of the ergodic limit. 
\end{Remark}

We finally establish that $U^\delta$ is monotone: 
\begin{Lemma}\label{lem.Udeltamonotone}
For any $\delta>0$ the map $U^\delta$ is monotone. 
\end{Lemma}

\begin{proof} Fix $m_0, m_0'\in \Pk$.  Let us recall that $U^\delta(x,m_0)= u^\de(0,x)$ where the pair $(u^\de,m^\de)$ solves \eqref{e.MFGih} with initial condition $m_0$. We denote by $(u',m')$ the solution of \eqref{e.MFGih} with initial condition $m_0'$. Then by duality, we have 
$$
\frac{d}{dt} e^{-\delta t} \inte (u^\de(t,x)-u'(t,x))(m^\de(t,x)-m'(t,x))dx \leq 0, 
$$
where, as $u^\de$ and $u'$ are bounded and $m^\de$ and $m'$ are probability measures, 
$$
\lim_{t\to+\infty} e^{-\delta t} \inte (u^\de(t,x)-u'(t,x))(m^\de(t,x)-m'(t,x))dx= 0. 
$$
This proves that 
$$
\inte (U^\delta(x,m_0)-U^\delta(x,m_0'))d(m_0-m_0')(x) =  \inte (u^\delta(0,x)-u'(0,x))d(m_0-m_0')(x)\geq 0.
$$
\end{proof}

%%%%%%%%%%%%%%%%%%%%%%%%%%
\subsection{Existence of a solution for the master cell problem}

\begin{proof}[{\bf Proof of Theorem \ref{thm.MasterCell}}] Let us start with the proof of the existence of the solution to the master cell problem. The proof of the uniqueness of the ergodic constant is given in Proposition \ref{prop.key} below.

For $\delta>0$, let $U^\delta$ be the solution to the discounted master equation \eqref{e.MasterDiscont}. We have seen in Lemma \ref{lem:StandEstiUdelta} and Proposition \ref{prop:UdeltaLip} that $U^\delta$ and $D_x U^\delta$ are uniformly Lipschitz continuous and that $\delta U^\delta$ is bounded. We set $W^\delta(x,m)= U^\delta(x,m)-U^\delta(0,\bar m)$. Then $W^\delta$ is   bounded and uniformly Lipschitz continuous on the compact space $\T^d\times \Pk$, so that it converges, up to a subsequence, to a continuous map $\chi:\T^d\times \Pk\to \R$. Since $D_xW^\delta$ is also bounded in Lipschitz norm, we deduce that $D_x \chi$ is Lipschitz continuous (in $\T^d\times \Pk$). Moreover $(\delta U^\delta(0, \bar m))$ converges (along the same subsequence, without loss of generality) to some constant $\lambda$. 

Next we prove that $\chi$ is a weak solution to \eqref{CellMAster}. We already know that $\chi$ and $D_x\chi$ are Lipschitz continuous with respect to both variables. In addition,  $\chi$ is monotone thanks to Lemma \ref{lem.Udeltamonotone}.  Let $T>0$, $m_0\in \Pk$ with a smooth density and $(w^\delta,m^\delta)$ be the solution to 
$$
\left\{\begin{array}{l}
-\partial_t w^\delta +\delta w^\delta+ \delta U^\delta(0, \bar m) -\Delta w^\delta +H(x,D   w^\delta) = F(x, m^\delta)\qquad {\rm on}\; (0,T)\times \T^d,\\
\partial_tm^\delta -\Delta  m^\delta -\dive( m^\delta H_p(x,Dw^\delta))= 0\qquad {\rm on}\; (0,T)\times \T^d,\\
m^\delta(0,\cdot)=m_0, \; w^\delta(T,\cdot)= W^\delta(x, m^\delta(T))\qquad {\rm on}\;  \T^d.
\end{array}\right.
$$
By definition we  have $W^\delta(x, m^\delta(T))= U^\delta(x,m^\de(T))-U^\delta(0,\bar m)$ and we know that $U^\delta(x,m^\de(t))= u^\de(t,x)$ for all $t$, where $u^\delta$ is a solution to \eqref{e.MFGih}.  Hence we deduce that  
$$
w^\de(t,x)= u^\de(t,x)- U^\de(0,\bar m)= W^\de(x,m(t))
$$
for all $(t,x)\in (0,T)\times \T^d$.  In particular, by Lemma \ref{lem.udeltamdelta}, $w^\delta$ is uniformly bounded in $C^{1+\alpha/2, 2+\alpha}$ for some $\alpha\in (0,1)$ while $m^\delta$ is uniformly bounded and uniformly continuous on $[0,T]$ with values in $\Pk$.
So there exists a subsequence, still denoted for simplicity by $(w^\delta,m^\delta)$, such that $w^\delta$ converges in $C^{1,2}$ to a map $w$ and  $m^\delta$ converges in $C^0([0,T], \Pk)$ to a map $m$. The pair $(w,m)$ is a solution to 
$$
\left\{\begin{array}{l}
-\partial_t w + \lambda -\Delta w +H(x,D  w) = F(x, m)\qquad {\rm in}\; (0,T)\times \T^d,\\
\partial_tm -\Delta  m -\dive( m H_p(x,Dw))= 0\qquad {\rm in}\; (0,T)\times \T^d,\\
m(0,\cdot)=m_0, \; w(T,\cdot)= \chi(x, m(T))\qquad {\rm in}\; \T^d.
\end{array}\right.
$$
As the solution to this equation is unique (because $\chi$ is monotone), we derive that $(w,m)$ is the unique solution to \eqref{MFGmasterCell}. Moreover, as $w^\delta(0,x)= W^\delta(x,m_0)$, we also have at the limit $w(0,x)= \chi(x, m_0)$. This proves that $\chi$ is a weak solution to \eqref{CellMAster}. 
\end{proof}

Let us now come back to the ergodic MFG problem \eqref{e.MFGergo}. We denote by $(\bar \lambda, \bar u, \bar m)$ the solution to this equation. 

\begin{Proposition}\label{prop.key} Let $(\lambda, \chi)$ be a solution of the ergodic master equation. Then we have $\lambda=\bar \lambda$ and $D_x\chi(x, \bar m)= D\bar u(x)$.
\end{Proposition}

\begin{proof} Let us fix $T>0$ and let $(u,m)$ be the solution to 
\be\label{MFGmasterCell-bar}
\left\{\begin{array}{l}
-\partial_t u +  \lambda -\Delta u +H(x,D   u) = F(x, m)\qquad {\rm in}\; (0,T)\times \T^d,\\
\partial_tm -\Delta  m -\dive( m H_p(x,Du))= 0\qquad {\rm in}\; (0,T)\times \T^d,\\
m(0,\cdot)=\bar m, \; u(T,\cdot)= \chi(x, m(T))\qquad {\rm in}\;  \T^d.
\end{array}\right.
\ee
We have already noticed that $m$ is the solution to the McKean-Vlasov equation
$$
\partial_tm -\Delta  m -\dive( m H_p(x,D_x\chi(x,m(t))))= 0, \qquad m(0,\cdot)=\bar m,
$$ 
which has a unique solution because $D_x\chi$ is Lipschitz continuous. This means that $m$ is defined independently of the horizon $T$. As we know that 
$u(t,x)= \chi(x,m(t))$, the same holds for $u$. 
Then, from the usual  energy inequality applied to $(u-\bar u, m-\bar m)$, we have, for any $0\leq t_1\leq t_2\leq T$, 
\be\label{kjargezlsd}
\int_{t_1}^{t_2} \inte \frac{m+\bar m}{2}|Du-D\bar u|^2 \leq 
-C \left[ \inte (u-\bar u)(m-\bar m)\right]_{t_1}^{t_2}. 
\ee
The right-hand side is bounded because $u(t,\cdot)=\chi(\cdot,m(t))$ and $\bar u$ are bounded, so that
\be\label{khjzbqfsdn}
\int_0^T \inte \bar m |Du-D\bar u|^2\leq C. 
\ee
By Lemma \ref{lem.fromm0meas} we have 
\be\label{boundm-barm}
\sup_{t\in [0,T]} \|m(t)-\bar m\|_{L^2}\leq C. 
\ee
As $\bar m$ is bounded below, \eqref{khjzbqfsdn} implies that there exists $t_T\in [T/2,T]$ such that $\inte |Du(t_T)-D\bar u|^2\leq 2C/T$. 
In particular, for $T$ large enough, we have, by \eqref{kjargezlsd} applied with $t_1=0$ and $t_2= t_T$,  
$$
\begin{array}{rl}
\ds \int_0^1 \inte |Du-D\bar u|^2 \;  \leq & \ds \int_0^{t_T} \inte |Du-D\bar u|^2\leq -C\inte (u(t_T)-\bar u)(m(t_T)-\bar m) \\
\leq & \ds -C\inte (u(t_T)-\bar u-\lg u(t_T)-\bar u\rg)(m(t_T)-\bar m)\\
\leq & \ds C \|Du(t_T)-D\bar u\|_{L^2} \; \leq \; C T^{-1/2},
\end{array}
$$
by Poincar\'e's inequality, \eqref{boundm-barm} and our choice of $t_T$. 
 Letting $T\to \infty$ we can conclude that $Du=D\bar u$ on $[0,1]\times \T^d$. Therefore, $m$ satisfies 
$$
\partial_tm -\Delta  m -\dive( m H_p(x,D\bar u(x)))= 0\; {\rm on }\; (0,1)\times \T^d, \qquad m(0,\cdot)=\bar m.
$$ 
But this equation has $\bar m$ as a unique solution, which shows that $m(t,x)=\bar m(x)$ on $[0,1]\times \T^d$. The McKean-Vlasov equation \eqref{e.McKV} being autonomous, we finally have $m(t)=\bar m$ and $Du(t,x)= D_x\chi(x,\bar m)=D\bar u(x)$ for any $(t,x)\in [0,T]\times \T^d$ and, as a consequence, $\lambda = \bar \lambda$.  
\end{proof}

%%%%%%%%%%%%%%%%%%%%%%%%%%%%%%%%%
%%%%%%%%%%%%%%%%%%%%%%%%%%%%%%%%%
\section{The long time behavior}\label{sec:longtime}

We now fix a solution $\chi$ to the master cell problem and, given a terminal condition $G:\T^d\times \Pk\to \R$ satisfying our standing assumptions (see Subsection \ref{subsec:derivative}), we consider the solution to the backward equation
\be\label{MasterEq}
\left\{\begin{array}{l}
\ds -\partial_t U(t,x,m) -\Delta_x U(t,x,m) +H(x,D_xU(t,x,m)) -\inte \dive(D_mU(t,x,m,y))dm(y) \\
\ds \hspace{1cm}  +\inte D_mU(t,x,m,y). H_p(y, D_x U(t,y,m))dm(y) = F(x,m) \quad  {\rm in }\; (-\infty, 0)\times \T^d\times \Pk,\\
\ds U(0,x,m)= G(x,m) \; {\rm in }\;  \T^d\times \Pk\,.
\end{array}\right.
\ee
We recall that the existence of  a unique classical solution to \eqref{MasterEq} was proved in \cite{CDLL}. Here is our main convergence result.
 
\begin{Theorem}\label{thm.main} Let $\chi$ be a  weak solution to the master cell problem \eqref{CellMAster}. Then,  there exists a constant $c\in \R$ such that 
$$
\lim_{t\to -\infty} U(t,x, m)+\bar \lambda t= \chi(x,m)+c, 
$$
uniformly with respect to $(x,m)\in \T^d\times \Pk$. 

Moreover, we also have that $ D_x U(t,x, m) \to D_x \chi(x,m)$ as $T\to \infty$, uniformly with respect to $(x,m)$.
\end{Theorem} 

Theorem \ref{thm.main} implies  the convergence of the solution of the MFG system as $T\to +\infty$. 

\begin{Corollary} \label{cor:krjehnrf} Let $c$ be the constant given in Theorem \ref{thm.main}. For $T>0$ and $m_0\in\Pk$, let $(u^T,m^T)$ be the solution to \eqref{intro.MFG}. Then, for any $t\geq 0$,  
$$
\lim_{T\to+\infty} u^T(t,x) -\bar \lambda (T-t)= \chi(x,m(t))+c,
$$
where the convergence is uniform in $x$ and $m$ solves 
\be\label{mvla}
\partial_t m-\Delta m-\dive(mH_p(x,D_x\chi(x,m)))=0, \qquad m(0)=m_0.
\ee
Moreover, for any $\delta\in(0,1)$, 
$$
\lim_{T\to+\infty} u^T(\delta T,x) -(1-\delta)\bar \lambda T= \chi(x, \bar m)+c,
$$
where $(\bar u, \bar m)$ solves \eqref{e.MFGergo} and where the convergence is uniform in $x$. 
\end{Corollary}

In particular, when $t=0$, we get
$$
\lim_{T\to+\infty} u^T(0,x) -\bar \lambda T= \chi(x,m_0)+c.
$$

\begin{proof}[Proof of Corollary \ref{cor:krjehnrf}] We know that 
$u^T(t,x)= U(t-T,x,m^T(t))$ and that $m^T$ solves the McKean-Vlasov equation
$$
\partial_t m^T-\Delta m^T-\dive(m^TH_p(x,D_xU(t-T,x,m)))=0, \qquad m^T(0)=m_0.
$$
As $x\to D_xU(t,x,m)$ is bounded in $C^1$ (see Proposition \ref{prop.reguU} below),  we know from Theorem \ref{thm.main} that, as $T\to +\infty$,   $(D_xU(t-T,\cdot,\cdot))$ converges  uniformly to $D_x\chi$. So, for any $t\geq 0$, $m^T$ converges  in $C^0([0,t],\Pk)$ towards $m$ solution of \rife{mvla}. Then again by Theorem \ref{thm.main}, we have 
$$
\lim_{T\to+\infty}  u^T(t,x) +\bar \lambda(t-T)= 
\lim_{T\to+\infty}  U(t-T,x,m^T(t))+\bar \lambda(t-T)= \chi(x,m(t))+c. 
$$

Let us now fix $\delta>0$. From Theorem \ref{thm:CvExpMFG}, we have that $m^T(\delta T)$ converges (exponentially fast) to $\bar m$. Hence, by Theorem \ref{thm.main} again, we have 
$$
\ds \lim_{T\to+\infty}  u^T(\delta T,x) -(1-\delta)\bar \lambda T = \lim_{T\to+\infty}  U(-(1-\delta)T,x,m^T(\delta T))-(1-\delta)\bar \lambda T= \chi(x,\bar m)+c. 
$$

\end{proof} 
 
 The proof of Theorem \ref{thm.main} relies on estimates on $U(t,\cdot,\cdot)$ (independent of $t$) developed in the next section.  

 %%%%%%%%%%%%%%%%%%%%%%%%%%%%%%%%%%%%%%
\subsection{Lipschitz estimates of the solution  $U$}

We collect here the main estimates satisfied by the solution of \eqref{MasterEq}. They actually follow from the estimates developed in Section \ref{stimeT} for the solution $(u,m)$ of the MFG system.

\begin{Proposition}\label{prop.reguU}
Let $U$ be a solution to the master equation \eqref{MasterEq}.
Then  there exists a constant $C$ such that 
\be\label{liuzehaezsdj}
\sup_{t\leq0, \ m\in \Pk}  \| U(t,\cdot, m)+ \bar \lambda t \|_{C^{2+\alpha}} + \|D_mU(t, \cdot, m, \cdot)\|_{2+\alpha,1+\alpha} \leq C,
\ee
while 
$$
\sup_{(x,m)\in \T^d\times \Pk} |U(t,x,m)-U(s,x,m)| \leq C|t-s|^{\frac12}\qquad \forall s,t\leq 0, \; |s-t|\leq 1.
$$ 
\end{Proposition}
 
\begin{proof} Let us recall that, for any $t_0\leq 0$ and $m_0\in \Pk$, one has $U(t_0, x,m_0)=u(t_0,x)$, where $(u,m)$ is the solution to the MFG system 
$$
\left\{\begin{array}{l}
\ds -\partial_tu -\Delta  u +H(x,D   u) = F(x, m)\quad {\rm in}\; (t_0,0)\times \T^d,\\
\ds \partial_t m -\Delta  m -\dive( m H_p(x,D  u))= 0\quad {\rm in}\; (t_0,0)\times \T^d,\\
\ds m(t_0,\cdot)=m_0, \; u(0, \cdot)= G(x,m(0))\quad {\rm in}\;  \T^d.
\end{array}\right.
$$
 By Lemma \ref{lem.SemiConc}, we have the Lipschitz bound $\ds \|Du\|_\infty\leq C$, uniform with respect to the horizon $t_0$. This proves that  $\|D_xU\|_\infty \leq C$ and, in turn, that $m$ is uniformly H\"older continuous in time with values in $\Pk$, see \rife{holder-m}.  Furthermore, from Theorem \ref{thm:CvExpMFG} we get an estimate for  $U(t-T,x,m)$ at time $t=0$; namely, that there exists a constant $C$, independent of $T$, such that   
 $$
 \| D_x U(-T,\cdot, m_0)\|_{C^{1+\alpha}} \leq C
 $$
 and
$$
 \left \|U(-T,x,m_0) -\bar \lambda T\right\|_\infty \leq C\,.
$$
Therefore, we deduce that
$$
\sup_{t\leq0, \ m\in \Pk} \| U(t,\cdot, m)+\bar \lambda t \|_{C^{2+\alpha}} \leq C.
$$
Following \cite{CDLL}, the derivative of $U$ with respect to $m$ can be represented as 
\be\label{kjahzsduil}
\inte \frac{\delta U}{\delta m}(t_0,x,m_0,y)\mu_0(y)dy = v(t_0,x),
\ee
where, for any smooth map $\mu_0:\T^d\to \R$,  $(v,\mu)$ solves the linearized problem
$$
\left\{\begin{array}{l}
\ds -\partial_tv -\Delta  v +H_p(x,D   u).Dv = \frac{\delta F}{\delta m}(x, m)(\mu)\quad {\rm in}\; (t_0,0)\times \T^d,\\
\ds \partial_t \mu -\Delta  \mu -\dive( \mu H_p(x,D u))-\dive (mH_{pp}(x,Du)Dv)= 0\quad {\rm in}\; (t_0,0)\times \T^d,\\
\ds \mu(t_0,\cdot)=\mu_0, \; v(0, \cdot)= \frac{\delta G}{\delta m}(x,m(0))(\mu(0))\quad {\rm in}\;  \T^d.
\end{array}\right.
$$
Our aim is to provide estimates on $v$ in order to show the uniform Lipschitz regularity of $U$ with respect to $m$. We assume that $\inte\mu_0=0$ since we are only interested in $D_mU= D_y\frac{\delta U}{\delta m}$. Then Corollary \ref{coro.unifBoundT} states that 
$$
\sup_{t\in [0,T]} \|v(t)\|_{C^{2+\alpha}} \leq C\|\mu_0\|_{(C^{2+\alpha})'}. 
$$
This proves that 
$$
\left\| \inte \frac{\delta U}{\delta m}(t_0,\cdot,m_0,y)\mu_0(y)dy\right\|_{C^{2+\alpha}} \leq C \|\mu_0\|_{(C^{2+\alpha})'},
$$
for any smooth map $\mu_0$ with $\inte \mu_0=0$. 
Therefore, as in \cite{CDLL},   we obtain
\be\label{DmUt}
\left\|D_m U(t_0,\cdot,m_0,\cdot)\right\|_{2+\alpha,1+\alpha} \leq C.
\ee

It remains to check the time regularity of $U$. As $U(t, x,m(t))=u(t,x)$ and $U$ is  uniformly Lipschitz continuous in $m$, we have, for $t_0\leq s\leq t_0+1$, 
$$
\begin{array}{rl}
\ds \left|U(s,x,m_0)-U(t_0,x,m_0)\right| \; \leq  & \ds C\dk(m_0,m(s))+ \left|U(s,x,m(s))-U(t_0,x,m_0)\right| \\
\leq & \ds  C |s-t_0|^{\frac12}+ \left|u(s,x)-u(t_0,x)\right|  \leq C |s-t_0|^{\frac12},
\end{array}
$$
where we used the uniform regularity of $m$ in time (since $H_p(\cdot,Du)$ is bounded) for the second inequality, and the uniform regularity of $u$ in the last one. 
\end{proof}

 \begin{Remark} 
 We stress that if we only use the regularity condition (FG2) on the couplings, then we can replace the conclusion of Corollary \ref{coro.unifBoundT} with the first order estimate \rife{c1} and obtain, rather than \rife{DmUt}, the milder estimate $
 \left\|D_m D_x U(t ,x,m )\right\|_{\infty} \leq C$. This is actually enough to conclude with the uniform Lipschitz bound for $U$ and $D_xU$, which is what is only needed in the proof of Theorem \ref{thm.main}.
 \end{Remark}

%%%%%%%%%%%%%%%%%%%%%%%%%%%%%%%%%%%%%%
 \subsection{Proof of Theorem \ref{thm.main}}
 
 We are now ready to prove our main result.

\begin{proof}[Proof of Theorem \ref{thm.main}.] 

Let $\chi$  be  a weak solution to the master cell problem \eqref{CellMAster}. 
%Applying Theorem \ref{thm:CvExpMFG} to  both the system with  $u(T)=G(x,m(T))$ and the system with   $u(T)=\chi(x,m(T))$, we deduce that there exists a constant $C>0$  such that 
%\be\label{eq.bounds}
%\chi(x,m)+\bar \lambda t-C\leq U(-t,x,m) \leq \chi(x,m)+\bar \lambda t+C.
%\ee
For $T>0$, let us consider 
$$
U^T(t,x,m)= U(t-T, x,m)\quad {\rm for }\; (t,x,m)\in (-\infty, T]\times \T^d\times \Pk.
$$ 
Then $U^T$ solves 
$$
\left\{\begin{array}{l}
\ds -\partial_t U^T -\Delta_x U^T +H(x,D_xU) -\inte \dive(D_mU^T(t,x,m,y))dm(y)\\
\ds \hspace{0.4cm} +\inte D_mU^T(t,x,m,y). H_p(D_xU(t,y,m,y))dm(y) = F(x,m) \quad   {\rm in }\; (-\infty, T)\times \T^d\times \Pk,\\
\ds U^T(T,x,m)= G(x,m) \; {\rm in }\;  \T^d\times \Pk.
\end{array}\right.
$$
By the Lipschitz regularity of $U$ and $D_xU$ and the bound in \eqref{liuzehaezsdj} (Proposition \ref{prop.reguU}),  the family 
$\{U^T(\cdot,\cdot ,\cdot)+\bar \lambda (\cdot-T)\}_T$ is relatively compact in $C^0(\R\times \T\times \Pk)$. 
Let  $T_n\to+\infty$ be any sequence such that   $(t,x,m)\to U^{T_n}(t,x,m)+\bar \lambda (t-T_n)$ locally uniformly converges to some $V(t,x,m)$. Then $V$ is a weak solution to 
\be\label{eq.lsunflnd}
\left\{\begin{array}{l}
\ds -\partial_t V +\bar \lambda-\Delta_x V +H(x,D_xV) -\inte \dive(D_mV(t,x,m,y))dm(y) \\
\ds \hspace{.5cm}+\inte D_mV(t,x,m,y). H_p(y,D_xU(t,y,m,y))dm(y) = F(x,m)  \quad   {\rm in }\; \R\times \T^d\times \Pk.
\end{array}\right.
\ee
in the sense that $V$ satisfies similar requirements as in Definition \ref{defweakmaster}. Namely, $V$ and $D_xV$ are uniformly Lipschitz continuous in $x$ and $m$, $1/2-$H\"older continuous in the time variable, $V$ is monotone in $m$ and satisfies that, for any $t_1\leq t_2$ and if $(u,m)$ solves the MFG system:
\be\label{lareblzredtf}
\left\{\begin{array}{l}
-\partial_t u +\bar \lambda -\Delta u +H(x,D   u) = F(x, m)\; {\rm in }\; (t_1, t_2)\times \T^d,\\
\partial_tm -\Delta  m -\dive( m H_p(x,Du))= 0\; {\rm in }\; (t_1, t_2)\times \T^d,\\
m(t_1,\cdot)=m_0, \; u( t_2,\cdot)= V( t_2 , x , m( t_2)) \; {\rm in }\;  \T^d,
\end{array}\right.
\ee
we have  $V(t_1, x, m_0)= u(t_1,x)$ (and so $V(t,x,m(t))= u(t,x)$ for any $t\in [t_1, t_2]$). 

Our goal is to show that $V(t,x,m)- \chi(x,m)$ is constant.   Let us recall that Proposition \ref{prop:CvuT(0)} implies that $U^T(0, x,\bar m)-\bar\lambda T-\bar u$ converges to a constant $\bar c$ as $T\to+\infty$. Hence $V(0,x,\bar m)= \bar u(x)+ \bar c$. Since $\chi(x,\bar m)=\bar u$, this  shows that, if   $V(t,x,m)- \chi(x,m)$ will be proved to be constant, then this constant will be equal to $\bar c$, and independent of the subsequence $(T_n)$.

Let us fix $m_0\in   \Pk$. Let $T>0$  be large and $(u,m)$ be the solution to the MFG system \eqref{lareblzredtf} with $t_1=0$ and $t_2=T$.  
We note that  $m$ is the unique solution to the McKean-Vlasov equation
\be\label{McKV}
\begin{cases}
\partial_tm -\Delta  m -\dive( m H_p(x,D_xV(t,x,m)))= 0  & \; {\rm on}\; [0,T]\times \T^d, \; 
\\
m(0)=m_0 &   \; {\rm n}\; \T^d. 
\end{cases}
\ee
In particular, since $V$ and $D_x V$ are globally Lipschitz  in $m$, this implies that  $m$ and $u$ are defined independently of the horizon $T$ (meaning that, for $t\in [0, T]$, $u(t,\cdot):=V(t,\cdot, m(t))$ and $m(t,\cdot)$ do not depend on $T$). 

In the same way we define $(\tilde u, \tilde m)$ to be the solution to the MFG system 
$$
\left\{\begin{array}{l}
-\partial_t \tilde u +\bar \lambda -\Delta \tilde u +H(x,D  \tilde  u) = F(x, \tilde m)\; {\rm in }\; (0,T)\times \T^d,\\
\partial_t\tilde m -\Delta \tilde  m -\dive(\tilde  m H_p(x,D\tilde u))= 0\; {\rm in }\; (0,T)\times \T^d,\\
\tilde m(0,\cdot)=m_0, \; \tilde u(T,\cdot)= \chi(x, \tilde m(T))  \; {\rm in }\;  \T^d.
\end{array}\right.
$$
As before we note that $(\tilde u, \tilde m)$ does not depend on the horizon $T$, that $\tilde u(t,x)= \chi(x,\tilde m(t))$ for any $t\in [0,T]$ and that $\tilde m$ is the unique solution to 
the McKean-Vlasov equation
\be\label{McKV2}
\partial_t\tilde m -\Delta  \tilde m -\dive( m H_p(x,D_x\chi(x,\tilde m))= 0 \; {\rm on}\; [0,T], \; \tilde m(0)=m_0.
\ee
Using the result of Theorem \ref{thm:CvExpMFG} with both $G(x,\cdot)= V(T,x,\cdot)$ and $G= \chi(x,\cdot)$, we have 
(changing $u$ into $u+\bar \lambda (T-t)$ and $\tilde u $ into $\tilde u+\bar \lambda (T-t)$):
$$
\|m(t) -\bar m\|_\infty+ \|\tilde m(t)-\bar m\|_\infty \leq C(e^{-\gamma t}+ e^{-\gamma (T-t)}), \qquad t\in [1,T],
$$
where $(\bar u,\bar m)$ is the solution to the ergodic MFG system \eqref{e.MFGergo}. 
But since $m$ and $\tilde m$ do  not depend on the horizon $T$, here we can let first $T\to \infty$, and then $t\to \infty$, so we conclude that  both $m(t)$ and $\tilde m(t)$ converge to $\bar m$ as $t\to+\infty$. \\
Applying once more the standard estimates on the MFG systems, we have
%, for any $t_0\leq t_1\leq t_2\leq t_0+T$,  
$$
\begin{array}{rl}
\ds \int_{0}^{T} \inte (m+\tilde m) \left| Du-D\tilde u\right|^2 \; \leq & \ds 
-C\left[ \inte (u-\tilde u)(m-\tilde m)\right]_{0}^{T}\\
 =& \ds  -C \inte (u( T)-\tilde u( T))(m( T)-\tilde m( T))
\end{array}
$$
since $m(0)=\tilde m(0)=m_0$. 
As $u$ and $\tilde u$ are uniformly Lipschitz continuous in space and $m(T)$ and $\tilde m(T)$ have the same limit $\bar m$ as $T\to+\infty$, we deduce that 
$$
\lim_{T\to+\infty} \int_0^T \inte (m+\tilde m) \left| Du-D\tilde u\right|^2 =0.
$$
In particular, as $m$ (and $\tilde m$) are regular and bounded below by a positive constant on intervals of the form $[\ep, T]$ with $\ep>0$, we deduce that $Du =D\tilde u $ on $[\ep,T]$ and thus on $[0,T]$. Therefore   $m$ and $\tilde m$ solve the same equation, which implies  $m(t)=\tilde m(t)$ for any $t\geq 0$.  
Coming back to the equations satisfied by $u$ and $\tilde u$ gives $\partial_t u=\partial_t \tilde u$, so that there is a constant $c$ such that $u(t,x)=\tilde u(t,x)+c$. In other words
$$
V(t,x,m(t))=\chi(x,m(t)) + c\qquad \forall t \geq 0\,.
$$
Notice that the above conclusion holds for any given $m_0 \in \Pk$ and  the constant $c$ could depend on $m_0$ at this stage. But we are going to show that this is actually not the case.

Indeed, let us choose $m_0=\bar m$. Then Proposition \ref{prop.key} says that $m(t)= \tilde m(t)= \bar m$. We denote by $\bar c$ the constant found above, i.e. $u(t,x)=\tilde u(t,x)+\bar c$. By definition, this implies that $V(t,x,\bar m)= \chi(x,\bar m)+ \bar c$.
Now, for any $m_0\in \Pk$,  we recall that the solution $m(t)=\tilde m(t)$ converges to $\bar m$ as $t\to+\infty$. By  the uniform Lipschitz continuity of $\chi$ and $V$ with respect to $m$ (uniform in $(t,x)$), this implies that
$$
|V(t,x,m(t))-V(t,x,\bar m)| + |\chi(x,m(t))-\chi(x,\bar m)| \to 0 \qquad \hbox{as $t\to \infty$.}
$$
Since
$$
|c-\bar c|= |V(t,x,m(t))-\chi(x,m(t))-(V(t,x,\bar m)-\chi(x,\bar m))|,   
$$
by letting $t\to\infty$ we deduce that $c= \bar c$.  In particular, we have proved that
$$
V(0,x,m_0)=\chi(x,m_0) + \bar c \qquad \forall m_0\in \Pk\,.
$$
Finally, we can apply the above reasoning to  the  translation $V( \cdot + t_0, x,m)$, for any $t_0\in \R$. It turns out that  $\bar c= \lim\limits_{t\to \infty} V(t+t_0,x,m(t))- \chi(x,m(t))$, which  is clearly independent of $t_0$. Therefore we conclude that 
$$
V(t_0,x,m_0)= \chi(x,m_0)+\bar  c \qquad \forall (t_0,x,m_0)\in \R\times \T^d\times \Pk,
$$
and the proof is complete.
\end{proof} 

Let us point out that any   weak solution of the ergodic master equation solves \eqref{eq.lsunflnd}. So the above proof actually shows that two solutions of the ergodic master equation differ only by a  constant: 

\begin{Corollary}\label{cor:uniquergo} If $\chi_1$ and $\chi_2$ are weak solutions  of the ergodic master equation \eqref{CellMAster}, then there exists a constant $\bar c$ such that 
$$
\chi_2(x,m)=\chi_1(x,m)+\bar c \qquad \forall (x,m)\in \T^d\times \Pk.
$$
\end{Corollary}

%%%%%%%%%%%%%%%%%%%%%%%%%%%%%
\section{The discounted problem}\label{sec:discount}

In this section we investigate the behavior, as $\delta\to 0^+$, of the solution $U^\delta$ of  the discounted master equation \eqref{e.MasterDiscont}. 
 Our main result is: 

\begin{Theorem} \label{thmdisc} Let $U^\delta$ be the solution to the discounted master equation \eqref{e.MasterDiscont} and $(\bar \lambda,\bar u,\bar m)$ the solution of the ergodic problem \eqref{e.MFGergo}. Then, as $\delta\to 0^+$, $U^\delta-\bar \lambda/\delta$ converges uniformly to the solution $\chi$ to the master cell problem \eqref{CellMAster} such that $\chi(x,\bar m)= \bar u(x) + \bar \theta$, where  $\bar \theta$ is the unique constant   for which the following linearized ergodic problem has a solution    $(\bar v,\bar \mu)$: 
\be\label{eq.LMFGergo}
\left\{\begin{array}{l}
\ds \bar u +\bar \theta -\Delta \bar v +H_p(x,D\bar u).D\bar v = \frac{\delta F}{\delta m}(x, \bar m)(\bar \mu)\qquad {\rm in}\; \T^d,\\
\ds -\Delta \bar \mu -\dive(\bar \mu H_p(x,D \bar u))-\dive (\bar m H_{pp}(x,D\bar u)D\bar v)= 0\qquad {\rm in}\; \T^d,\\
\inte \bar \mu= \inte \bar v=0.
\end{array}\right.
\ee
\end{Theorem}

Let us comment  a bit more on the normalization condition $\chi(x,\bar m)= \bar u(x) + \bar \theta$ which selects the unique limit of the discounted master equation \eqref{e.MasterDiscont}, according to the above result.  As we shall see in the next section, given any (not necessarily normalized with zero average) solution $\bar u$ to 
\be\label{baru}
\bar \lambda -\Delta \bar u + H(x,D\bar u) = F(x,\bar m)\qquad {\rm in}\; \T^d\,,
\ee
 there is a unique constant $\bar \theta$ for which \eqref{eq.LMFGergo}  admits a  solution. However, since $\bar u$ is unique up to addition of a constant,  the sum  $\bar u  + \bar \theta$ will be uniquely determined. Indeed, by changing $\bar u$ through addition of a constant, the value $\bar \theta$ will be translated accordingly. 
In other words, one can say that the limit of $U^\delta-\bar \lambda/\delta$ is the solution $\chi$ of the master cell problem \eqref{CellMAster} such that $\chi(x,\bar m)$ coincides with the unique solution of \eqref{baru} for  which the constant $\bar \theta$ vanishes. 

Exactly as for the time dependent problem, we can infer from   Theorem \ref{thmdisc} the limit behavior of the solution of the discounted MFG system: 

\begin{Corollary}\label{cor.discount}
Let $m_0\in \Pk$ and, for $\delta>0$, $(u^\delta,m^\delta)$ be the solution to the  discounted MFG system \eqref{e.MasterDiscont}. Then 
$$
\lim_{\delta \to 0} u^\delta(0,x)-\bar \lambda/\delta= \chi(x,m_0),
$$
uniformly with respect to $x$, where $\chi$ is the solution of the ergodic  cell problem \eqref{CellMAster}  given in Theorem \ref{thmdisc}. 
\end{Corollary}

%%%%%%%%%%%%%%%%%
\subsection{An additional ergodic system}

Given a solution $\bar u$ of the MFG ergodic problem \eqref{e.MFGergo}, we investigate the   ergodic problem \eqref{eq.LMFGergo}. The heuristic justification of \eqref{eq.LMFGergo} is that we expect the solution $(\bar u^\delta,\bar m^\delta)$ of \eqref{barudeltamdelta} to be of the form 
\be\label{heu.udelta}
\bar u^\delta\sim \frac{\bar \lambda}\de+ \bar u+\bar \theta +\delta \bar v, \qquad \bar m^\delta\sim \bar m+\delta \bar \mu,
\ee
and, in view of \eqref{barudeltamdelta}, the equation satisfied by $(\bar \theta, \bar v, \bar \mu)$ should be \eqref{eq.LMFGergo}.

We start the proof of the existence for \eqref{eq.LMFGergo} as usual, by a discounted problem:

\begin{Lemma}\label{lem.DiscLMFGergo} Let $A, B\in L^\infty(\T^d)$. For $\delta>0$ small, there is a unique solution $(v^\delta, \mu^\delta)\in W^{1,\infty}(\T^d)\times L^\infty(\T^d)$ to the discounted system 
\be\label{eq.DiscLMFGergo}
\left\{\begin{array}{l}
\ds \bar u +\delta v^\delta -\Delta v^\delta +H_p(x,D\bar u).Dv^\delta = \frac{\delta F}{\delta m}(x, \bar m)(\mu^\delta)+A\qquad {\rm in}\; \T^d,\\
\ds -\Delta \mu^\delta -\dive(\mu^\delta H_p(x,D \bar u))-\dive (\bar m H_{pp}(x,D\bar u)Dv^\delta)= \dive (B)\qquad {\rm in}\; \T^d,\\
\end{array}\right.
\ee
with $\inte \mu^\de=0$. 
Moreover, there is a constant $C>0$ (independent of $\delta$, $A$ and $B$) such that 
$$
\|\delta v^\delta\|_\infty+ \|Dv^\delta\|_\infty+\|\mu^\delta\|_\infty\leq C\left(1+\|A\|_\infty+\|B\|_\infty\right).
$$
\end{Lemma}

\begin{proof} Existence of a solution runs   with  a standard fixed point, so we omit it. 
The duality relation gives (using Poincar\'e's inequality)
\begin{eqnarray*}
C^{-1}\|Dv^\delta\|_{L^2}^2 & \leq & \inte (\bar u +\delta v^\delta-A)\mu^\delta + B\cdot Dv^\delta\\
& \leq & (\|D\bar u\|_{L^2} +\delta \|Dv^\delta\|_{L^2}+\|A\|_{L^2})\|\mu^\delta\|_{L^2} + \|B\|_{L^2}\|Dv^\delta\|_{L^2}, 
\end{eqnarray*}
so that 
$$
\|Dv^\delta\|_{L^2}\leq C\left((\|D\bar u\|_{L^2}^{1/2}+\|A\|_{L^2}^{1/2})\|\mu^\delta\|_{L^2}^{1/2} + \delta \|\mu^\delta\|_{L^2}+ \|B\|_{L^2} \right).
$$
By Corollary \ref{cor.StaMeas}, we have
$$
\|\mu^\delta\|_{L^2}\leq C(\|Dv^\delta\|_{L^2}+\|B\|_{L^2})\leq C\left((\|D\bar u\|_{L^2}^{1/2}+\|A\|_{L^2}^{1/2})\|\mu^\delta\|_{L^2}^{1/2} + \delta \|\mu^\delta\|_{L^2}+ \|B\|_{L^2} \right).
$$
So, for $\delta>0$ small enough, we obtain
$$
\|\mu^\delta\|_{L^2}\leq C\left(\|D\bar u\|_{L^2}+\|A\|_{L^2} + \|B\|_{L^2} \right).
$$
This implies the same bound for $Dv^\delta$ and, by the maximum principle, the estimate 
$$
\|\delta v^\delta\|_\infty\leq  C (\|\bar u\|_{L^\infty}+ \|D\bar u\|_{L^2}+ \|B\|_{L^2}+\|A\|_{L^\infty}).
$$
Moreover, considering the equation satisfied by $w:=v^\delta-\lg v^\delta\rg$, we have by local regularity for weak solutions (Theorem 8.17 of \cite{GT}) and Poincar\'e inequality: 
$$
\|v^\delta-\lg v^\delta\rg\|_\infty\leq C (1+\|v^\delta-\lg v^\delta\rg\|_{L^2})\leq C(1+\|Dv^\delta\|_{L^2})\leq  C\left(1+\| \bar u\|_{W^{1,\infty}}+\|A\|_{L^\infty} + \|B\|_{L^2} \right).
$$
Then by classical elliptic regularity (Theorem 8.32 of \cite{GT}), we have, for any $\alpha\in (0,1)$,  
$$
\|v^\delta-\lg v^\delta\rg \|_{C^{1+\alpha}}\leq C \left(1+\| \bar u\|_{W^{1,\infty}}+\|A\|_{L^\infty} + \|B\|_{L^2} \right).
$$
We can now apply the  local regularity for weak solutions to $\mu^\de$ (Theorem 8.17 of \cite{GT}) and infer that 
$$
\|\mu^\delta\|_{C^{\alpha}}\leq C(\|Dv^\delta\|_\infty+ \|B\|_\infty)\leq C \left(\| \bar u\|_{W^{1,\infty}}+\|A\|_{L^\infty} + \|B\|_{L^\infty} \right).
$$
\end{proof}

\begin{Proposition}\label{prop.LMFGergo}  Let  $(\bar \lambda, \bar u,\bar m)$ be a solution of the ergodic system \eqref{e.MFGergo} and $(v^\delta,\mu^\delta)$ be the solution to \eqref{eq.DiscLMFGergo} for $A$ and $B$ satisfying 
$$
\|A\|_\infty+\|B\|_\infty \leq C\delta, 
$$
for some constant $C$. Then, as $\delta\to 0^+$,  
$$
\delta \lg v^\delta\rg \to\bar \theta, \qquad (v^\delta-\lg v^\delta)\rg \xrightarrow{L^\infty}  \bar v , \qquad \mu^\delta\xrightarrow{L^\infty} \bar \mu
$$
where $(\bar \theta, \bar v, \bar \mu)$ is the unique solution to \eqref{eq.LMFGergo}.
\end{Proposition}

\begin{proof}[Proof of Proposition \ref{prop.LMFGergo}] Passing to the limit in \eqref{eq.DiscLMFGergo} (up to a subsequence) provides a constant $\bar \theta$ (limit of $\delta \lg v^\delta\rg$), a map $\bar v\in W^{1,\infty}$ (limit of $v^\delta-\lg v^\delta\rg$) and a map  $\bar \mu\in L^\infty$ (limit of $\mu^\delta$) which solve \eqref{eq.LMFGergo}. Uniqueness of $D\bar v$ (and hence of $\bar v$) and of $\bar \mu$  can be established by a standard duality argument. Then $\bar \theta$ is unique by the equation. The full convergence of $(\delta \lg v^\delta\rg, v^\delta-\lg v^\delta\rg, \mu^\delta)$ holds  by uniqueness of the limit. 
\end{proof}

%%%%%%%%%%%%%%%%%%%%%%%%%%%%%%%%%%%%%%
\subsection{Proof of Theorem \ref{thmdisc}}

The proof of Theorem \ref{thmdisc} consists mostly in showing that the heuristic relation \eqref{heu.udelta} holds. 

\begin{Proposition}\label{prop.cvbarudelta}
 Let $(\bar \lambda, \bar u,\bar m)$, $(\bar u^\delta, \bar m^\delta)$ and  $(\bar \theta, \bar v, \bar \mu)$ be respectively solutions to \eqref{e.MFGergo}, \eqref{barudeltamdelta} and  \eqref{eq.LMFGergo}. Then
$$
\lim_{\delta \to 0^+} \|\bar u^\delta-\frac{\bar \lambda}\de-\bar u-\bar \theta\|_{\infty}+\|\bar m^\delta-\bar m\|_\infty=0.
$$
\end{Proposition}

\begin{proof} The argument is very close to the proof of the exponential rate (see Theorem \ref{thm:CvExpMFG}). Let 
$$
E=\{ (v,\mu)\in W^{1,\infty}(\T^d)\times L^\infty(\T^d), \; \| \delta  v\|_\infty+\|Dv\|_\infty+\|\mu\|_\infty\leq \hat C\},
$$
where $\hat C$ is to be chosen below. 
%Let $C_0>0$ be such that $C_0^{-1}\leq \bar m$. 
For 
%$\delta$ so small that $\delta \hat C\leq C_0^{-1}$ and 
$(v,\mu)\in E$, we consider the solution  $(\hat v,\hat \mu)$ to \eqref{eq.DiscLMFGergo} with 
$$
\begin{array}{rl}
\ds A(x)\; := & \ds \delta ^{-1}\left(-(H(x,D(\bar u+\delta v))-H(x,D\bar u)-\delta H_p(x,D\bar u)\cdot Dv)\right.\\
& \qquad \qquad \left. + F(x,\bar m+\delta \mu)-F(x,\bar m)-
\delta\frac{\delta F}{\delta m}(x,\bar m)(\mu)\right),
\end{array}
$$
$$
 B(x):= \delta^{-1} \left((\bar m+\delta \mu)H_p(x,D(\bar u+\delta v))-\bar mH_p(x,D\bar u)-\delta \mu H_p(x,D\bar u)-\delta \bar m H_{pp}(x,\bar m)Dv\right).
$$
As 
$$
\|A\|_\infty+\|B\|_\infty\leq C\hat C^2\delta, 
$$
we have, by Lemma \ref{lem.DiscLMFGergo} (and for $\delta$ small enough), 
$$
\|\delta \hat v\|_\infty+ \|D\hat v\|_\infty+\|\hat  \mu\|_\infty\leq C\left(1+\|A\|_\infty+\|B\|_\infty\right)\leq C(1+\hat C^2\delta).
$$
We can choose $\hat C$ such that, for $\delta$ small enough, the right-hand side is less than $\hat C$. Then we can easily conclude that the map 
$(v,\mu)\to (\hat v,\hat \mu)$ has a fixed point $(v^\delta, \mu^\delta)$.  Note that $(\frac{\bar \lambda}\de+\bar u+\delta v^\delta, \bar m+\delta \mu^\delta)$ solves \eqref{barudeltamdelta} and therefore is equal to $(\bar u^\delta, \bar m^\delta)$. Hence, by  Proposition \ref{prop.LMFGergo}, we deduce
\begin{align*}
\|\bar u^\delta-\frac{\bar \lambda}\de-\bar u-\bar \theta\|_{\infty}& = \| \delta v^\de-\bar \theta\|_{\infty}
\\
& \leq \| \delta (v^\de-\lg v^\de \rg )\|_\infty + | \delta \lg v^\de \rg-\bar \theta| \to 0
\qquad\mbox{\rm as $\delta\to 0$},
\end{align*}
which completes the proof. 
\end{proof}

\begin{proof}[Proof of Theorem \ref{thmdisc}] Recall that we have uniform Lipschitz estimates on $U^\delta$ and on $D_xU^\delta$ (Lemma \ref{lem:StandEstiUdelta} and Proposition \ref{prop:UdeltaLip}) and that any converging subsequence is a weak solution of the ergodic master equation (proof of Theorem \ref{thm.MasterCell}). Therefore, we only need to show that $U^\delta-\de^{-1} \bar \lambda$ has a limit when evaluated 
at some value. For this, let $(\bar u^\delta,\bar m^\delta)$ be the solution to \eqref{barudeltamdelta}. As $(\bar u^\delta,\bar m^\delta)$ is also a stationary solution to \eqref{e.MFGih}, we have  
$$
U^\delta(x,\bar m^\delta)= \bar u^\delta(x)\qquad \forall x\in \T^d. 
$$
We have seen in Proposition \ref{prop.cvbarudelta} that, as $\delta\to 0$, $\bar m^\delta$ converges to $\bar m$ while $\bar u^\delta-\de^{-1}\bar \lambda$ converges to $\bar u+\bar \theta$. This completes the proof. 
\end{proof}

%{\color{red}
%\section{APPENDIX A SUPPRIMER}
%
%Theorem IV.9.1 de \cite{LSU} DONNE L'existence de solutions avec des bornes du type (pour coef bornés)
%$$
%\|Dv\|_{C^{\alpha/2,\alpha}}\leq C donnee \ initiale
%$$ 
%\cite[Theorem III.8.1 p. 196]{LSU}: Estimations $L^\infty_{loc}$ \\ 
%\cite[Theorem III.10.1 p. 196]{LSU}: local H\"older regu for $u$\\
%Theo III.11.1 de LSU: regu $C^{\alpha/2,\alpha}$ locale de $Du$ pour eq a coef bornes. 
%
%}

 \end{document}